\documentclass{amsart}
\usepackage{amsmath,mathrsfs,amssymb,amsthm,amscd}
\usepackage{mathtools}
\usepackage[]{fontenc}
\usepackage[all]{xy}
\usepackage{hyperref}
\usepackage{graphics}
\usepackage{xcolor}
\usepackage{tikz}
\usepackage{tikz-3dplot}
\usepackage{orcidlink}
\usepackage{a4wide}
\usepackage{enumitem}
\usepackage{bbm}
\usetikzlibrary{patterns}
\usetikzlibrary{3d}
\usetikzlibrary{arrows,automata,backgrounds}
\tikzstyle pt=[fill,black]
\tikzstyle ptbl=[fill,black]
\setlist[enumerate,1]{label=(\roman*)}
\setlist[enumerate,2]{label=(\alph*)}
\swapnumbers

\newcommand{\C}{\mathbb{C}}

\renewcommand{\H}{\mathbb{H}}

\newcommand{\N}{\mathbb{N}}
\renewcommand{\P}{\mathbb{P}}
\newcommand{\Q}{\mathbb{Q}}

\newcommand{\R}{\mathbb{R}}
\newcommand{\T}{\mathbb{T}}
\newcommand{\Z}{\mathbb{Z}}

\newcommand{\caB}{\mathcal{B}}

\newcommand{\caD}{\mathcal{D}}
\newcommand{\caE}{\mathcal{E}}

\newcommand{\caK}{\mathcal{K}}
\newcommand{\caL}{\mathcal{L}}
\newcommand{\caM}{\mathcal{M}}

\newcommand{\caO}{\mathcal{O}}

\newcommand{\caU}{\mathcal{U}}
\newcommand{\caV}{\mathcal{V}}

\newcommand{\caX}{\mathcal{X}}
\newcommand{\caY}{\mathcal{Y}}

\newcommand{\sA}{\mathscr{A}}
\newcommand{\sB}{\mathscr{B}}

\newcommand{\sG}{\mathscr{G}}

\newcommand{\sJ}{\mathscr{J}}

\newcommand{\sL}{\mathscr{L}}

\newcommand{\sP}{\mathscr{P}}

\newcommand{\cc}{\mathrm{c}}
\newcommand{\dd}{\mathrm{d}}
\newcommand{\ddc}{\dd\dd^{\cc}}

\newcommand{\an}{\mathrm{an}}
\newcommand{\can}{\mathrm{can}}

\newcommand{\loc}{\mathrm{loc}}
\newcommand{\model}{\mathrm{mod}}
\newcommand{\adel}{\mathrm{adel}}

\newcommand{\nef}{\mathrm{nef}}
\newcommand{\inte}{\mathrm{int}}

\def\WbDiv{\operatorname{W-b-Div}_{\R}}
\def\CbDiv{\operatorname{C-b-Div}_{\R}}

\DeclareMathOperator{\tr}{tr}
\DeclareMathOperator{\rk}{rk}

\DeclareMathOperator{\Pic}{Pic}
\DeclareMathOperator{\bPic}{\mathbf{Pic}}
\DeclareMathOperator{\PicQ}{Pic_{\Q}}
\DeclareMathOperator{\bPicQ}{\mathbf{Pic}_{\Q}}
\DeclareMathOperator{\Pich}{\widehat{Pic}}
\DeclareMathOperator{\bPich}{\mathbf{\widehat{Pic}}}
\DeclareMathOperator{\PichQ}{\widehat{Pic}_{\Q}}
\DeclareMathOperator{\bPichQ}{\mathbf{\widehat{Pic}}_{\Q}}
\DeclareMathOperator{\Div}{Div}
\DeclareMathOperator{\DivR}{Div_{\R}}
\DeclareMathOperator{\DivQ}{Div_{\Q}}
\DeclareMathOperator{\height}{h}
\DeclareMathOperator{\dv}{div}
\DeclareMathOperator{\divh}{\widehat{div}}
\DeclareMathOperator{\DivhR}{\widehat{Div}_{\R}}

\DeclareMathOperator{\Sp}{Sp}
\DeclareMathOperator{\SL}{SL}
\DeclareMathOperator{\GL}{GL}

\DeclareMathOperator{\im}{Im}
\DeclareMathOperator{\cl}{cl}

\DeclareMathOperator{\Spec}{Spec}

\DeclareMathOperator{\Hom}{Hom}

\DeclareMathOperator{\id}{\mathbbm{1}}

\DeclareMathOperator{\Capa}{Cap}
\DeclareMathOperator{\PSH}{PSH}

\numberwithin{equation}{section}

\theoremstyle{plain}
\newtheorem{proposition}{Proposition}[section]
\newtheorem{lemma}[proposition]{Lemma}
\newtheorem{theorem}[proposition]{Theorem}
\newtheorem{corollary}[proposition]{Corollary}

\theoremstyle{definition}
\newtheorem{notation}[proposition]{Notation}
\newtheorem{definition}[proposition]{Definition}
\newtheorem{remark}[proposition]{Remark}

\newtheorem{example}[proposition]{Example}

\begin{document}

\title{On the height of the universal abelian variety}
\author{Jos\'e Ignacio Burgos Gil}
\thanks{J.\,I.~Burgos~Gil was partially supported by MINISTERIO DE CIENCIA E INNOVACION (MICIU/AEI/10.13039/501100011033), 
research projects PID2019-108936GB-C21, PID2022-142024NB-I00, and ICMAT Severo Ochoa projects CEX2019-000904-S
and CEX2023-001347-S}
\address{J.\,I.~Burgos~Gil,
  \orcidlink{0000-0003-2640-2190}:  0000-0003-2640-2190, ICMAT (CSIC-UAM-UC3-UCM) Madrid}
\email{burgos@icmat.es}

\author{J\"urg Kramer}
\thanks{J.~Kramer acknowledges support from the DFG Cluster of Excellence MATH+}
\address{J.~Kramer, \orcidlink{0009-0004-9455-4414}: 0009-0004-9455-4414,
  Humboldt-Universit\"at zu Berlin} 
\email{kramer@math.hu-berlin.de}
\date{\today}

\begin{abstract}
In this paper we extend the arithmetic intersection theory of adelic divisors on quasi-projective varieties developed by X.~Yuan and S.\,W.
Zhang to cover certain adelic arithmetic divisors that are not nef nor integrable. The key concept used in this extension is the relative finite 
energy introduced by T.~Darvas, E.~Di Nezza, and C.\,H.~Lu. As an application we compute the arithmetic self-intersection number of the 
line bundle of Siegel--Jacobi forms on the universal abelian variety endowed with its invariant hermitian metric. The techniques developed 
in this paper can be applied in many other situations like mixed Shimura varieties or the moduli space of stable marked curves.
\end{abstract}

\maketitle
\setcounter{tocdepth}{2}
\tableofcontents

\section{Introduction} 
\label{first-section}

\subsection{Background}
Let $\sA_{g}(\C)$ denote the moduli space of principally polarized abelian varieties of dimension $g$ over the complex numbers and
$\omega$ the Hodge bundle on $\sA_{g}(\C)$. The calculation of the geometric degree of $\omega$ has been of fundamental interest in 
the past. Since Chern--Weil theory holds on $\sA_{g}(\C)$, this degree computation amounts to calculating the volume of $\sA_{g}(\C)$ 
with respect to the natural invariant metric. The validity of Chern--Weil theory is due to the fact that the natural invariant metric has only 
mild logarithmic singularities when approaching the boundary of $\sA_{g}(\C)$. The volume computation has been carried out by C.\,L.
Siegel in~\cite{Siegel} and leads to the formula
\begin{equation}
\label{siegel1}
\deg(\omega)=(-1)^{d'}\zeta_{\Q}(1-2g)\cdot\zeta_{\Q}(3-2g)\cdot\ldots\cdot\zeta_{\Q}(-1),
\end{equation}
where $d'=g(g+1)/2=\dim_{\C}(\sA_{g}(\C))$ and $\zeta_{\Q}(s)$ is the classical Riemann zeta function. 

Shifting now our attention to arithmetic intersections, we need first to replace the complex varieties under consideration by arithmetic
varieties. Thus, we consider the moduli stack $\sA_{g}$ of principally abelian schemes of dimension $g$ over $\Spec(\Z)$. We point 
out that in order to avoid to work with stacks we will have to introduce level structures; however, for the sake of exposition, we do not 
use level structures here. Considering then the Hodge bundle $\overline{\omega}=(\omega,\Vert\cdot\Vert_{\mathrm{inv}})$ on $\sA_
{g}$ equipped with the natural invariant metric, it was striking when U.~K\"uhn in~\cite{Kuehn:gainc} and J.-B.~Bost in~\cite{Bost} 
independently computed in the case $g=1$ the arithmetic degree $\widehat{\deg}(\overline{\omega})$ of $\overline{\omega}$, roughly 
speaking, as the logarithmic derivative of its geometric degree. More precisely, they obtained the formula
\begin{equation}
\label{ulf}
\widehat{\deg}(\overline{\omega})=-\deg(\omega)\Bigg(\frac{\zeta'_{\Q}(-1)}{\zeta_{\Q}(-1)}+\frac{1}{2}\Bigg).
\end{equation}
This result was instrumental for various subsequent developments: Let us mention first the establishment of an arithmetic intersection
theory generalizing the one of H.~Gillet and C.~Soul\'e~\cite{MR1087394} in the series of papers~\cite{BurgosKramerKuehn:accavb}
and~\cite{BurgosKramerKuehn:cacg} taking into account logarithmically singular metrics. Second, the result was important for the further
development of the so-called ``Kudla Program'', see for example~\cite{Kudla}. As a third ingredient, we refer to the conjectural generalization 
of formula~\eqref{ulf} in the form (see also~\cite{MR})
\begin{equation}
\label{siegel2}
\widehat{\deg}(\overline{\omega})=-\deg(\omega)\Bigg(\frac{\zeta'_{\Q}(1-2g)}{\zeta_{\Q}(1-2g)}+\frac{\zeta'_{\Q}(3-2g)}{\zeta_{\Q}(3-2g)}
+\ldots+\frac{\zeta'_{\Q}(-1)}{\zeta_{\Q}(-1)}+a\log(2)+b\Bigg)
\end{equation}
with $a,b\in\Q$; this conjectural formula has been verified for $g=2$ (see~\cite{JvP}) and for certain Hilbert modular surfaces (see~\cite
{BBGK}).

After having established an arithmetic intersection theory on the moduli stack $\sA_{g}$, the question arises if such a theory could also 
be developed on the universal abelian scheme $\pi\colon\sB_{g}\rightarrow\sA_{g}$. An interesting example would be the calculation of
the arithmetic degree of the line bundle $\sL$ corresponding to the principal polarization of $\sB_{g}$, since it is related to $\omega$ by 
means of the ``formule cl\'e'' stating that $\pi_{\ast}\sL\cong\omega^{\otimes-1/2}$.

However, it turned out that even in the geometric setting the degree computations on the universal abelian variety are more complicated
than on the moduli space of universal abelian varieties. This is the reason, why our next focus was to understand the situation on the
universal abelian variety $\pi\colon\sB_{g}(\C)\rightarrow\sA_{g}(\C)$ together with the line bundle $J=\pi^{\ast}\omega\otimes L$, where 
$L$ is the rigidified symmetric line bundle corresponding to twice the principal polarization. As a first step, we investigated the case $g=1$, 
where $\sB_{1}(\C)$ becomes an elliptic surface. There, it turns out that the natural invariant metric on $L$ acquires singularities which 
are more severe than logarithmic ones, which leads to the fact that Chern--Weil theory no longer holds on $\sB_{1}(\C)$. In the respective 
investigations~\cite{bkk}, we found that by replacing divisors on $\sB_{1}(\C)$ by b-divisors given by a limit of divisors on a sequence of 
blow-ups of $\sB_{1}(\C)$ along its boundary where the metric becomes singular, allows us to define sequences of intersection numbers 
which converge in such a way that Chern--Weil theory continues to hold in the limit. This work was generalized in~\cite
{botero21:_chern_weil_hilber_samuel} to arbitrary genus $g\ge 1$ and
in \cite{botero22:_rings_SJ} the authors established the formula
\begin{equation}
\label{degdeg}
\deg(\text{b-div}(J))=2^{g}\cdot\frac{d!}{d'!}\cdot\deg(\omega),
\end{equation}
where $d=d'+g=\dim_{\C}(\sB_{g}(\C))$ and $\text{b-div}(J)$ is the b-divisor associated to the line bundle $J$. This shows that the degree 
of the b-divisor associated to the line bundle $J$ on $\sB_{g}(\C)$ is up to an elementary rational factor given by the degree of the Hodge 
bundle on $\sA_{g}(\C)$.

Of course, it is now immediate to ask if there is a framework such that formula~\eqref{degdeg} carries over to the arithmetic setting. Thus, 
we next consider the universal abelian scheme $\sB_{g}$ with the line bundle $\overline{\sJ}=\pi^{\ast}\overline{\omega}\otimes\overline{\sL}$, 
where, as before, $\overline{\sL}=(\sL,\Vert\cdot\Vert_{\mathrm{inv}})$ is the rigidified symmetric line bundle on $\sB_{g}$ corresponding to 
twice the principal polarization equipped with the natural invariant metric. As mentioned above, Chern--Weil theory fails to hold on $\sB_{g} 
(\C)$, which is one of the main reasons that the generalized arithmetic intersection theory developed in~\cite{BurgosKramerKuehn:accavb} 
no longer applies to define arithmetic intersections on $\sB_{g}$. Thus, the question arises, if there is a further extension of arithmetic 
intersection theory that allows one to define and compute the arithmetic degree of the line bundle $\overline{\sJ}$. More specifically, one 
might ask, if formula~\eqref{degdeg} has an arithmetic analog. We are able to answer both questions affirmatively in this paper.

\subsection{Main results}
In order to find a promising attempt to establish a generalized arithmetic intersection theory for hermitian line bundles equipped with singular 
metrics on $\sB_{g}$, it is worth noticing that the use of b-divisors in the geometric setting as described above can be extended to the 
arithmetic setting by means of the theory of adelic arithmetic line bundles over quasi-projective arithmetic varieties developed by X.~Yuan 
and S.\,W.~Zhang in~\cite{YuanZhang:adelic}. In their theory, X.~Yuan and S.\,W.~Zhang defined arithmetic intersection numbers for so-called 
\emph{adelic} arithmetic divisors as limits of ``classical'' arithmetic intersection numbers as long as the adelic arithmetic divisors are sufficiently 
positive, i.\,e., nef as arithmetic divisors. The arithmetic intersection numbers can be extended by linearity to \emph{integrable} adelic arithmetic 
divisors, that are differences of nef adelic arithmetic divisors.

It turns out that the hermitian line bundle $\overline{\sL}$ discussed above, defines an adelic arithmetic divisor on $\sB_{g}$, which is nef so 
that its arithmetic self-intersection number is well-defined. On the other hand, even though the hermitian line bundle $\overline{\omega}$ also 
gives rise to an adelic arithmetic divisor on $\sA_{g}$, we show as a first main result in Theorem~\ref{thm:12} that it is not integrable, which 
prevents us from defining its arithmetic self-intersection number using the theory of X.~Yuan and S.\,W.~Zhang, although the existence of such 
an intersection number is known from~\cite{BurgosKramerKuehn:cacg}. Therefore, in order to have a well-defined arithmetic self-intersection 
product of $\overline{\sJ}$, we need to ``merge'' the results of~\cite{BurgosKramerKuehn:cacg} with those of~\cite{YuanZhang:adelic}. To this 
end, in our second main result, we extend the theory of X.~Yuan and S.\,W.~Zhang in Theorem~\ref{thm:fund} to be able to define arithmetic 
self-intersection numbers for certain adelic arithmetic divisors, which are not necessarily nef nor integrable like $\overline{\omega}$. This 
extension uses two recent analytical tools: The theory of non-pluripolar products of positive currents developed in~\cite{BEGZ} and the theory 
of relative energy introduced in~\cite{darvas18:mnpp} (that uses non-pluripolar products in an essential way). We refer to the recent survey
\cite{darvas23:_relat_kaehl} and the references therein.

The relation between the theory developed in~\cite{YuanZhang:adelic} and the relative energy is as follows. Let $\overline{\caD}=(\caD,g_{D})$ 
and $\overline{\caD}'=(\caD,g_{D}')$ be two nef adelic arithmetic divisors on an arithmetic variety $\caX$ that share the same geometric part
and only differ in the Green function's part. Then, by choosing a ``sufficiently large'' arithmetic reference divisor $\overline{\caE}=(\caE,g_{E})$ 
and by putting 
\begin{displaymath}
\omega_{E}=\ddc g_{E}+\delta_{E},\quad\varphi=g_{D}-g_{E},\quad\varphi'=g'_{D}-g_{E},
\end{displaymath}
the functions $\varphi$ and $\varphi'$ turn out to be $\omega_{E}$-plurisubharmonic on $\caX(\C)$ with $\varphi$ having finite relative 
energy $I_{\varphi'}(\varphi)$ with respect to $\varphi'$. We are now able to establish the relationship
\begin{displaymath}
\overline{\caD}^{d+1}=\overline{\caD}'^{d+1}+I_{\varphi'}(\varphi)
\end{displaymath}
between the corresponding arithmetic self-intersection numbers; here $d$ is the relative dimension of $\caX$. The relative energy $I_
{\varphi'}(\varphi)$ is computed by means of non-pluripolar products. This shows that we can define an arithmetic self-intersection number 
for an adelic arithmetic divisor $(D,g_{D}'')$, even if it is not nef nor integrable, as long as the function $\varphi''=g_{D}''-g_{E}$ is $\omega_
{E}$-plurisubharmonic and $\varphi''$ has finite relative energy with respect to $\varphi'$. One can argue that this is the largest possible 
extension of the theory given in~\cite{YuanZhang:adelic} to semipositive adelic divisors since one can say that a divisor with infinite relative 
energy should also have infinite arithmetic self-intersection. Many singular metrics that appear in practice are examples of metrics with finite 
relative energy, so we expect that this theory can be applied in many different situations. 

Our third main result in Theorem~\ref{thm:12}  states that the arithmetic self-intersection number of the adelic arithmetic line bundle $\overline
{\sJ}=\pi^{\ast}\overline{\omega}\otimes\overline{\sL}$, i.\,e., its arithmetic degree, is a well-defined real number. We compute this number as 
the fourth main result in Theorem~\ref{thm:10} and obtain the formula
\begin{equation}
\label{siegel3}
\widehat{\deg}(\overline{\sJ})=2^{g}\cdot\frac{(d+1)!}{(d'+1)!}\cdot\widehat{\deg}(\overline{\omega}),
\end{equation}
which confirms that the arithmetic analog of formula~\eqref{degdeg} holds. We would like to emphasize that formula~\eqref{siegel3} might 
be useful to establish a proof of the conjectural formla~\eqref{siegel2} in the spirit of C.\,L. Siegel's orginal inductive proof of formula~\eqref
{siegel1}, which consists of the following two steps: (1)~Start with the quantity $\deg(\omega)$ on $\sA_{g+1}(\C)$ and sweep it  to 
the boundary $\partial\sA_{g+1}(\C)\cong\sB_{g}(\C)$ using Eisenstein series. (2)~Push the resulting quantity down from $\sB_{g}(\C)$ to
$\sA_{g}(\C)$, which allows to relate $\deg(\omega)$ on $\sA_{g+1}(\C)$ and on $\sA_{g}(\C)$. Formula~\eqref{siegel3} constitutes the
arithmetic analog of Siegel's step~(2).

\subsection{Outline}
In Section~\ref{second-section}, we gather the analytic prerequisites needed for the remainder of the paper. Namely, we review the theory 
of plurisubharmonic functions and its relation with semipositive hermitian metrics and Green functions. We also discuss the non-pluripolar 
product of positive currents as well as the notion of relative energy. We end the section by recalling the notions of algebraic and almost 
algebraic singularities.

In Section~\ref{third-section}, we give a summary of the theory of adelic arithmetic divisors. We divide the discussion into three cases: The 
geometric case, the local arithmetic case, and the global arithmetic case. We also discuss the relationship between adelic arithmetic line 
bundles and adelic arithmetic divisors, and finally investigate the functorial properties that are needed for the computation of arithmetic 
self-intersection numbers (arithmetic degrees).

Section~\ref{sec:extension-yuan-zhang} is devoted to the extension of the theory to the finite relative energy case. To this end, we first show
that the difference of two arithmetic self-intersection numbers can be expressed in terms of the relative energy, which is the content of Theorem~
\ref{thm:fund}. We then give a definition of arithmetic self-intersection numbers in terms of relative energy. We end the section by showing that 
some of the functorial properties are also satisfied by this generalized arithmetic intersection product.

Finally, in Section~\ref{seventh-section}, we recall the theory of Siegel--Jacobi forms and prove in Theorem~\ref{thm:12} that the adelic arithmetic
Hodge bundle $\overline{\omega}$ is not integrable, in general. We end the section by computing the arithmetic self-intersection number of the 
adelic arithmetic line bundle of Siegel--Jacobi forms in Theorems~\ref{thm:9} and~\ref{thm:10}.

\subsection*{Acknowledgements}
We are grateful to the Humboldt-Universit\"at zu Berlin, to the Instituto de Ciencias Matem\'aticas (ICMAT), to the Mathematisches Forschungsinstitut 
Oberwolfach (MFO), and to the Centre International de Rencontres Math\'ematiques (CIRM) for their hospitality during the preparation of this 
paper.  We warmly thank A.~Botero, S.~Boucksom, T.~Darvas, W.~Gubler, K.~K\"unnemann, G.~Peralta, X.~Yuan, and S.\,W.~Zhang for many 
helpful conversations.

\section{Analytical prerequisites}
\label{second-section}

In this section we summarize the analytical prerequisites that are fundamental for this paper. Here, $X$ will denote a projective complex 
manifold of dimension $d$.

We refer to~\cite{dem-reg},~\cite{dem}, and~\cite{ochia}, but also to~\cite{bouck},~\cite{bfj-sing}, and~\cite{kim} for definitions and proofs 
of the analytic properties discussed in this section.

\begin{notation}
We define the twisted differential $\dd^{\cc}$ by the formula
\begin{displaymath}
\dd^{\cc}\coloneqq\frac{i}{4\pi}(\bar{\partial}-\partial).
\end{displaymath}
In this way, we obtain
\begin{displaymath}
\ddc=\frac{i}{2\pi}\,\partial\bar{\partial}.
\end{displaymath}
This is the usual convention used in Arakelov theory as, for example, in~\cite{LAG}. However, we point out that it differs by a factor of $1/2$
from the convention used in~\cite{demailly12:cadg}.
\end{notation}

\subsection{Plurisubharmonic functions}

\begin{definition}
\label{def:5}
Let $V$ be an open coordinate subset of $X$, which we identify with an open subset of $\C^{n}$. A function $\varphi\colon V\rightarrow
\R\cup\{-\infty\}$ is called \emph{plurisubharmonic}, if it satisfies the following two conditions:
\begin{enumerate}
\item
\label{item:6}
The function $\varphi$ is upper semi-continuous and not identically equal to $-\infty$ on every connected component of $V$.
\item
\label{item:7}
For every $z\in V$, every $\varepsilon>0$ such that the closed ball $B(z,\varepsilon)$ is contained in $V$, and every $w\in\C^{n}$ with
$\vert w\vert\le\varepsilon$, the inequality
\begin{displaymath}
\varphi(z)\le\frac{1}{2\pi}\int_{0}^{2\pi}\varphi(z+we^{it})\dd t
\end{displaymath}
holds.
\end{enumerate}
A function $\varphi\colon V\rightarrow\R\cup\{-\infty\}$ on an arbitrary open subset $V$ of $X$ is called \emph{plurisubharmonic}, if
$V$ can be covered by open coordinate subsets $\{V_{j}\}_{j\in J}$ such that $\varphi\vert_{V_{j}}$ is plurisubharmonic on $V_{j}$ for
all $j\in J$.
\end{definition}

\begin{remark}
\label{rem:1}
Let $V\subseteq X$ denote an open subset of $X$. As usual, we then denote by $L_{\loc}^{1}(V)$ the set of functions defined on $V$ 
that are locally integrable on $V$. If $\varphi\in L_{\loc}^{1}(V)$, then $\varphi$ defines a distribution on $V$ and thus a $(0,0)$-current 
on $V$, which we denote again by $\varphi$. Viewing the locally integrable function $\varphi$ as a $(0,0)$-current, allows the formation 
$\ddc\varphi$, which then becomes a closed $(1,1)$-current.
\end{remark}

Plurisubharmonic functions can be characterized in terms of positive currents. For the notion of positive currents and their properties, 
we refer to~\cite[Chapter~III]{demailly12:cadg}. The next result is~\cite[Chapter~I, Theorem~5.8]{demailly12:cadg}.

\begin{proposition}
\label{prop:1}
Let $V\subseteq X$ denote an open subset and $\varphi\colon V\rightarrow\R\cup\{-\infty\}$ be a function.
\begin{enumerate}
\item 
\label{item:1}
If $\varphi$ is plurisubharmonic on $V$, then $\varphi$ belongs to $L_{\loc}^{1}(V)$ and $\ddc\varphi$ is a closed positive $(1,1)
$-current.
\item 
\label{item:2}
Conversely, if $\varphi$ belongs to $L_{\loc}^{1}(V)$ and $\ddc\varphi$ is a closed positive $(1,1)$-current, then there exists a 
unique plurisubharmonic function $\varphi_{0}$ that agrees almost everywhere on $V$ with $\varphi$.
\end{enumerate}
\end{proposition}

\begin{remark}
\label{rem:2}
We note that part~\ref{item:2} of Proposition~\ref{prop:1} is the best we can achieve so that altering a function in a set of measure 
zero does not change the associated distribution, but may result that the function is no longer plurisubharmonic. Thus, in the sequel, 
when talking about plurisubharmonic functions, we will tacitly use the $L_{\loc}^{1}$-topology, i.\,e., we will identify two functions that
define the same distribution.
\end{remark}

Since $X$ is compact, any plurisubharmonic function on the whole $X$ is known to be constant. To have a richer global theory, we 
need to allow some more flexibility. 

\begin{definition}
\label{def:omega}
Let $\omega_{0}$ be a smooth closed positive $(1,1)$-form on $X$. A function $\varphi\colon V\rightarrow\R\cup\{-\infty\}$ defined on
an open subset $V\subseteq X$ is called \emph{$\omega_{0}$-plurisubharmonic}, if it satisfies the following two conditions:
\begin{enumerate}
\item
\label{item:8}
The function $\varphi$ is upper semi-continuous and locally integrable on $V$.
\item 
\label{item:9}
The current $\ddc\varphi+\omega_{0}$ is a closed positive $(1,1)$-current on $V$.
\end{enumerate}
The space of $\omega_{0}$-plurisubharmonic functions on $X$ is denoted by $\PSH(X,\omega_{0})$.
\end{definition}

When working in the $L_{\loc}^{1}$-topology, there is no need to insist on the upper-semicontinuity. In fact, a consequence of
Proposition~\ref{prop:1}~\ref{item:2} is the following global result.

\begin{corollary}
\label{cor:1}
Let $V\subseteq X$ be an open subset, $\omega_{0}$ be a smooth closed positive $(1,1)$-form, and $\varphi\colon V\rightarrow
\R\cup\{\pm\infty\}$ be a locally integrable function such that $\ddc\varphi+\omega_{0}$ is a (closed) positive $(1,1)$-form. Then,
 there exists a unique $\omega_{0} $-plurisubharmonic function that agrees with $\varphi$ almost everywhere.
\end{corollary}

The plurisubharmonic functions can be classified by their type of singularity.

\begin{definition}
\label{def:14} 
Let $\varphi_{1}$ and $\varphi_{2}$ be two $\omega_{0}$-plurisubharmonic functions on $X$. We say that \emph{$\varphi_{1}$ is more 
singular than $\varphi_{2}$}, if there is a constant $C$ such that  $\varphi_{1}\le\varphi_{2}+C$. We say that \emph{$\varphi_{1}$ and 
$\varphi_{2}$ have equivalent singularities}, if at the same time $\varphi_{1}$ is more singular than $\varphi_{2}$ and $\varphi_{2}$ is 
more  singular than $\varphi_{1}$. An equivalence class of singularities is called a \emph{singularity type}. If $\varphi$ is an $\omega_{0}
$-plurisubharmonic function, then $[\varphi]$ denotes its singularity type. If $\varphi_{1}$ is more singular than $\varphi_{2}$, we write
$[\varphi_{1}]\le[\varphi_{2}]$.
\end{definition}

Note that $[\varphi_{1}]=[\varphi_{2}]$ if and only if the difference $\varphi_{1}-\varphi_{2}$ is bounded.  

\begin{remark}
Definition~\ref{def:14} is tailored along the assumption that $X$ is compact. When $X$ is not compact, we have to replace the positive  
constant $C$ above by a locally bounded function.
\end{remark}

\subsection{Semipositive hermitian metrics}

\begin{remark}
If $L$ is a line bundle on $X$, we recall that a smooth hermitian metric $\Vert\cdot\Vert_{0}$ on $L$ is given by hermitian metrics on the 
fibers of $L$ that vary smoothly across $X$. Moreover, if $s$ denotes a non-trivial rational section of $L$, we have the equation
\begin{displaymath}
\ddc(-\log\Vert s\Vert_{0}^{2})+\delta_{\dv(s)}=\omega_{0},
\end{displaymath}
where $\omega_{0}$ is a closed smooth real $(1,1)$-form on $X$, the so-called curvature form of the hermitian line bundle $\overline{L}=
(L,\Vert\cdot\Vert_{0})$, which is known to be independent of the choice of $s$.
\end{remark}

\begin{definition}
Let $L$ be a line bundle on $X$ equipped with a smooth hermitian metric $\Vert\cdot\Vert_{0}$. We call a hermitian metric $\Vert\cdot
\Vert$ \emph{a singular hermitian metric on $L$}, if it is of the form
\begin{displaymath}
\Vert\cdot\Vert=\Vert\cdot\Vert_{0}\cdot e^{-\varphi/2},
\end{displaymath}
where $\varphi$ is a locally integrable function on $X$. Two such metrics will be identified if they agree almost everywhere on $X$. 
\end{definition}

\begin{remark}
Let $L$ be a line bundle on $X$ equipped with a singular hermitian metric $\Vert\cdot\Vert=\Vert\cdot\Vert_{0}\cdot e^{-\varphi/2}$ with 
$\varphi\in L_{\loc}^{1}(X)$ as in the above definition. Then, for any open subset $V\subseteq X$ and any non-vanishing section $s$ of 
$L$ over $V$, we observe that
\begin{displaymath}
-\log\Vert s\Vert^{2}=-\log\Vert s\Vert_{0}^{2}+\varphi, 
\end{displaymath}
which shows that $-\log\Vert s\Vert^{2}\in L_{\loc}^{1}(V)$. 
\end{remark}

\begin{definition}
\label{def:1}
Let $L$ be a line bundle on $X$ equipped with a singular hermitian metric $\Vert\cdot\Vert$. Then, the singular metric $\Vert\cdot\Vert$ is 
called \emph{semipositive}, if for any open subset $V\subseteq X$ and any non-vanishing section $s$ of $L$ over $V$, the locally integrable 
function $-\log\Vert s \Vert^{2}$ is plurisubharmonic on~$V$.
\end{definition}

\begin{remark}
Let $L$ be a line bundle on $X$ equipped with a singular hermitian metric $\Vert\cdot\Vert$. Let then $s$ and $s'$ be non-vanishing sections 
of $L$ over $V$ and $V'$, respectively. Since $s$ and $s'$ differ only
by a unit on the intersection $V\cap V'$, we arrive at the equality of currents
\begin{displaymath}
\ddc(-\log\Vert s\Vert^{2})\big\vert_{V\cap V'}=\ddc(-\log\Vert s'\Vert^{2})\big\vert_{V\cap V'}
\end{displaymath}
on $V\cap V'$. This shows that the $(1,1)$-current $\ddc(-\log\Vert s\Vert^{2})$ gives rise to a $(1,1)$-current on the whole of $X$.
\end{remark}

\begin{lemma}
Let $L$ be a line bundle on $X$ equipped with a smooth hermitian metric $\Vert\cdot\Vert_{0}$ with curvature form $\omega_{0}$. Then, 
there is a bijective correspondence
\begin{displaymath}
\big\{\text{semipositive metrics $\Vert\cdot\Vert$ on $L$}\big\}\quad\longleftrightarrow\quad\big\{\text{$\omega_{0}$-plurisubharmonic 
functions $\varphi$ on $X$}\big\},
\end{displaymath}
given by the assignment $\Vert\cdot\Vert\mapsto\varphi$, where $\varphi$ is locally, on a open subset $V\subseteq X$, determined by
\begin{displaymath}
\varphi=-\log\bigg(\frac{\Vert s\Vert^{2}}{\Vert s\Vert_{0}^{2}}\bigg),
\end{displaymath}
with $s$ being a non-vanishing section of $L$ over $V$.
\end{lemma}
\begin{proof}
We start by noting that the local definition of $\varphi$ globalizes, thus $\varphi$ is defined on the whole of $X$. Next, given a semipositive 
metric $\Vert\cdot\Vert$ on $L$, we have for any open subset $V\subseteq X$ and any non-vanishing section $s$ of $L$ over $V$ that
\begin{displaymath}
\varphi=-\log\Vert s\Vert^{2}+\log\Vert s\Vert_{0}^{2}
\end{displaymath}
on $V$. Since the locally integrable function $-\log\Vert s\Vert^{2}$ is plurisubharmonic on $V$ by Definition~\ref{def:1}, part~\ref{item:8} of
Definition~\ref{def:omega} is fulfilled. Furthermore, since $-\ddc\log\Vert s\Vert_{0}^{2}=\omega_{0}$, we find that
\begin{displaymath}
\ddc\varphi+\omega_{0}=\ddc(-\log\Vert s\Vert^{2})+\ddc(\log\Vert s\Vert_{0}^{2})+\omega_{0}=\ddc(-\log\Vert s\Vert^{2}),
\end{displaymath}
which, again using the plurisubharmonicity of $-\log\Vert s\Vert^{2}$, is a positive $(1,1)$-current on $V$, and thus verifies part~\ref{item:9} 
of Definition~\ref{def:omega}. In total, we have proven that $\varphi$ is an $\omega_{0}$-plurisubharmonic function on $X$.

On the other hand, if $\varphi$ is an $\omega_{0}$-plurisubharmonic function on $X$, we define $\Vert\cdot\Vert=\Vert\cdot\Vert_{0}\cdot e^
{-\varphi/2}$ and easily check that $\Vert\cdot\Vert$ becomes a semipositive metric on $L$.
\end{proof}

\subsection{Green functions}
\label{sec:sc-divisors-green}

For our purposes, it will be convenient to translate the above results into the language of divisors and Green functions. Note that for us a
Green function will be more appropriately called a Green current (of degree zero).  

\begin{definition}
\label{def:green}
Let $D$ be a divisor on $X$. A \emph{Green function for $D$} is the current associated to a locally integrable function $g\colon X\rightarrow
\R\cup\{\pm\infty\}$. A Green function $g$ for $D$ is called \emph{of smooth, of continuous, of locally bounded, or of plurisubharmonic type}, if 
$g$ can be chosen in such a way that  for any open subset $V\subseteq X$ and any local equation $f_{D}$ for $D$ on $V$, the function
\begin{displaymath}
g+\log\vert f_{D}\vert^{2}
\end{displaymath}
is smooth, continuous, locally bounded, or plurisubharmonic on $V$, respectively.  
\end{definition}

Note that two locally integrable functions on $X$ define the same Green function for $D$ if and only if they agree almost everywhere.

\begin{definition}
Let $g$ be a Green function for the divisor $D$. Then, we define the $(1,1)$-current
\begin{displaymath}
\omega_{D}(g)\coloneqq\ddc g+\delta_{D}.
\end{displaymath}
If $g$ is a Green function for $D$ of smooth type, then $\omega_{D}(g)$ is a smooth $(1,1)$-form. Moreover, if $g$ is a Green function for 
$D$ of plurisubharmonic type, then $\omega_{D}(g)$ is a closed positive $(1,1)$-current.
\end{definition}

\begin{lemma}
Let $D$ be a divisor on $X$, $L=\caO(D)$ the corresponding line bundle, and $s$ a section of $L$ with $\dv(s)=D$. Then, there is a bijective 
correspondence
\begin{displaymath}
\big\{\text{singular hermitian metrics $\Vert\cdot\Vert$ on $L$}\big\}\quad\longleftrightarrow\quad\big\{\text{Green functions $g$ for $D$}\big\},
\end{displaymath}
given by the assignment $\Vert\cdot\Vert\mapsto -\log\Vert s\Vert^{2}$. Moreover, this correspondence sends smooth, continuous, or semipositive 
hermitian metrics on $L$ to Green functions for $D$ of smooth, continuous, or plurisubharmonic type, respectively.
\end{lemma}
\begin{proof}
The proof of the lemma follows immediately from Definition~\ref{def:green} and the preceding discussion.
\end{proof}

\begin{remark}
Note that by Definition~\ref{def:green} any locally integrable function on $X$ is a Green function for any divisor. For example, we can start 
with a Green function $g$ for a divisor $D$ of smooth, continuous, or plurisubharmonic type. We can then think of $g$ as being a Green 
function for another divisor $D'$. But now the function $g$, viewed as a Green function for $D'$, may become singular or may no longer be 
of plurisubharmonic type. 
\end{remark}

The next lemma shows that positivity is always preserved when increasing the divisor.

\begin{lemma}
\label{lem:greeneff}
Let $D$ and $D'$ be divisors on $X$ such that the divisor $C\coloneqq D'-D$ is effective. If $g$ is a Green function for $D$ of plurisubharmonic
type, then $g$ also is a Green function for $D'$ of plurisubharmonic type.
\end{lemma}
\begin{proof}
Let $V\subseteq X$ be an open subset and let $f_{D}$, $f_{D'}$, and $f_{C}$ be local equations on $V$ for $D$, $D'$, and $C$, respectively, 
satisfying $f_{D'}=f_{D}\cdot f_{C}$. Since $C$ is effective, the function $f_{C}$ is holomorphic on $V$ and hence the function $\log\vert f_{C}
\vert^{2}$ is plurisubharmonic on $V$. Now, since the function $g+\log\vert f_{D}\vert^{2}$ is plurisubharmonic on $V$ by hypothesis, and since  
the sum of plurisubharmonic functions is again plurisubharmonic, we deduce that 
\begin{displaymath}
g+\log\vert f_{D'}\vert^{2}=g+\log\vert f_{D}\vert^{2}+\log\vert f_{C}\vert^{2}
\end{displaymath}
is also plurisubharmonic on $V$. This proves the claim.
\end{proof}

Let us next rephrase the result of Lemma~\ref{lem:greeneff} in terms of metrized line bundles. To do this, we need the following remark. 

\begin{remark} 
\label{rem:metrics}
Let $D$ be a divisor on $X$ and assume that the corresponding line bundle $\caO(D)$ is equipped with a semipositive hermitian metric $\Vert
\cdot\Vert$. Furthermore, let $D'$ be another divisor on $X$ such that the divisor $C\coloneqq D'-D$ is effective. Therefore, we have an inclusion 
$\caO(D)\subseteq\caO(D')$ of line bundles. By definition, we can view the line bundles $\caO(D)$ and $\caO(D')$ as subsheaves of the sheaf 
of rational functions $\caK_{X}$ of $X$. Thus, the unit function $1$ of $\caK_{X}$, which gives rise to a rational section of $\caO(D)$, can also 
be viewed as a rational section of $\caO(D')$. In this way, we can define by means of the semipositive hermitian metric $\Vert\cdot\Vert$ on 
$\caO(D)$, a hermitian metric $\Vert\cdot\Vert'$ on $\caO(D')$ through the formula
\begin{displaymath}
\Vert 1\Vert'\coloneqq \Vert 1\Vert,
\end{displaymath}
which will turn out to be again semipositive in the subsequent lemma.
\end{remark}

\begin{lemma}
\label{lem:metrics}
Let $D$ and $D'$ be divisors on $X$ such that the divisor $C\coloneqq D'-D$ is effective. If $\Vert\cdot\Vert$ is a semipositive hermitian metric 
on $\caO(D)$, then the hermitian metric $\Vert\cdot\Vert'$ on $\caO(D')$ defined as in Remark~\ref{rem:metrics} is also semipositive.
\end{lemma}
\begin{proof}
By our assumptions, we have
\begin{displaymath}
\mathcal{O}(D')=\mathcal{O}(D)\otimes\mathcal{O}(C).
\end{displaymath}
By now choosing $V$ small enough as an open subset of $X$, we find non-vanishing sections $s,s'$, and $s_{C}$ of the line bundles $\caO(D),
\caO(D')$, and $\caO(C)$ over $V$, respectively, such that we have
\begin{displaymath}
s'=s\cdot s_{C}.
\end{displaymath}
Observing next that the inclusion $\caO(C)\subseteq\caK_{X}$ is provided by the assignment $s_{C}\mapsto f_{C}^{-1}\cdot h$, where $f_{C}\in
\caK_{X}(V)$ gives rise to a local equation for the divisor $C$ on $V$ and $h\in\caK_{X}(V)^{\times}$ is a unit, we find
\begin{displaymath}
s'=s\cdot f_{C}^{-1}\cdot h.
\end{displaymath}
From this equality we compute
\begin{align*}
-\log\Vert s'\Vert'^{\,2}&=-\log\Vert s\Vert^{2}-\log\vert f_{C}^{-1}\vert^{2}-\log\vert h\vert^{2} \\[1mm]
&=-\log\Vert s\Vert^{2}+\log\vert f_{C}\vert^{2}-\log\vert h\vert^{2}.
\end{align*}
Arguing now as in the proof of Lemma~\ref{lem:greeneff}, the plurisubharmonicity of $-\log\Vert s\Vert^{2}$ on $V$ implies the plurisubharmonicity
of $-\log\Vert s'\Vert'^{\,2}$ on $V$, which proves the semipositivity of the hermitian metric $\Vert\cdot\Vert'$ on $\caO(D')$.
\end{proof}

In the next lemma we give a criterion when the change of divisor preserves positivity even if the change of divisor is not increasing.

\begin{lemma}
\label{lemm:4}
Let $D$ and $D'$ be divisors on $X$. If $g_{D}$ is a Green function for $D$ of plurisubharmonic type, and if for each point $x\in X$, there is a 
neighbourhood $V$ of $x$ and a local equation $f_{D'}$ of $D'$ such that $g_{D}+\log\vert f_{D'}\vert^{2}$ is locally bounded above on $V$, then 
$g_{D}$ also is a Green function for $D'$ of plurisubharmonic type. 
\end{lemma}
\begin{proof}
We begin by choosing a covering of $X$ by open subsets $V$ such that in each open subset there are equations $f_{D}$ and $f_{D'}$ for $D$ 
and $D'$, respectively. Moreover, $g_{D}+\log\vert f_{D}\vert^{2}$ extends to a plurisubharmonic function on $V$ and $g_{D}+\log\vert f_{D'}\vert^
{2}$ is bounded above on $V$. On $V\setminus(\vert D\vert\cup\vert D'\vert)$, the function
\begin{displaymath}
g_{D}+\log\vert f_{D'}\vert^{2}=g_{D}+\log\vert f_{D}\vert^{2}+\log\vert f_{D'}/f_{D}\vert^{2}
\end{displaymath}
is plurisubharmonic because the function $f_{D'}/f_{D}$ is holomorphic there. Since the function $g_{D}+\log\vert f_{D'}\vert^{2}$ is bounded 
above on $V$, by~\cite[Theorem 5.24]{demailly12:cadg}, it extends to a plurisubharmonic function on $V$, proving the lemma. 
\end{proof}

\subsection{Divisors with real coefficients}
\label{sec:real-coeffs}

Since later on we will be interested in different completions of spaces of divisors, it is natural to use real coefficients instead of rational or integral 
ones. We will write
\begin{displaymath}
\DivR(X)\coloneqq\Div(X)\otimes_{\Z}\R.
\end{displaymath}

The notion of Green function can be directly extended to divisors with real coefficients.

\begin{definition}
\label{def:4} 
Let $D$ be a divisor on $X$ with real coefficients. A \emph{Green function $g$ for $D$} is  the current associated to a locally integrable function 
$g\colon X\rightarrow\R\cup\{\pm\infty\}$. It is said to be \emph{of smooth, of continuous, or of locally bounded type}, if there are decompositions 
\begin{equation}
\label{eq:2}
D=\sum_{j=1}^{n}a_{j}D_{j}\qquad\text{and}\qquad g=\sum_{j=1}^{n}a_{j}g_{j},
\end{equation}
with $D_{j}\in\Div(X)$ and $a_{j}\in\R$ such that $g_{j}$ are Green functions for $D_{j}$ of smooth, of continuous, or of locally bounded type for 
$j=1,\dots,n$, respectively. If $D$ is a divisor on $X$ with real coefficients and $g$ is a Green function for $D$, we continue to write
\begin{displaymath}
\omega_{D}(g)=\ddc g+\delta_{D}.
\end{displaymath}
\end{definition}

The definition of a Green function of plurisubharmonic type does not rely on writing the divisor as a linear combination of ordinary divisors, but 
can be written directly. 

\begin{definition}
\label{def:6}  
Let $D$ be a divisor on $X$ with real coefficients. A Green function $g$ for $D$ is said to be \emph{of plurisubharmonic type}, if the closed 
$(1,1)$-current $\omega_{D}(g)$ is positive. 
\end{definition}

The next lemma gives a characterization of Green functions of plurisubharmonic type.

\begin{lemma}
\label{lemm:1} 
Let $D$ be a divisor on $X$ with real coefficients and $g_{0}$ be a Green function for $D$ of smooth type; put $\omega_{0}=\omega_{D}
(g_{0})$. A Green function $g$ for $D$ is of plurisubharmonic type, if and only if the function $g-g_{0}$ is $\omega_{0}$-plurisubharmonic. 
\end{lemma}
\begin{proof}
We first compute
\begin{displaymath}
\omega_{D}(g)=\ddc g+\delta_{D}=\ddc g+\omega_{0}-\omega_{0}+\delta_{D}=\ddc g+\omega_{0}-\ddc g_{0}=\ddc(g-g_{0})+\omega_{0}.
\end{displaymath}
Therefore, if $g-g_{0}$ is $\omega_{0}$-plurisubharmonic, then $\omega_{D}(g)$ is positive, and hence $g$ is of plurisubharmonic type.
For the converse, assume that $\omega_{D}(g)$ is positive. Then, $\ddc(g-g_{0})+\omega_{0}$ is positive. By Corollary~\ref{cor:1}, there 
is a unique $\omega_{0}$-plurisubharmonic function $\varphi$ that agrees with $g-g_{0}$ almost everywhere. Then, the function $g'
\coloneqq g_{0}+\varphi$ defines the same Green function as $g$ and $g'-g_{0}$ is $\omega_{0}$-plurisubharmonic.
\end{proof}

\begin{remark}
\label{rem:5}
If $D$ is a divisor with integral coefficients, then both definitions of Green functions of plurisubharmonic type, namely Definition~\ref{def:green} 
and Definition~\ref{def:6}, are equivalent, after identifying functions that agree almost everywhere. Indeed, if $g$ is a Green function of 
plurisubharmonic type for $D$ according to Definition~\ref{def:green}, then for any open subset $V\subseteq X$ and any local equation 
$f_{D}$ for $D$ on $V$, we have that $g+\log\vert f_{D}\vert^{2}$ is plurisubharmonic. In particular, we have that
\begin{displaymath}
\omega_{D}(g)\vert_{V}=\ddc(g+\log\vert f_{D}\vert^{2})\ge 0.
\end{displaymath}
Hence, $g$ is of plurisubharmonic type according to Definition~\ref{def:6}. Conversely, assume that $g$ satisfies Definition~\ref{def:6}. Let 
$g_{0}$ be a Green function of smooth type for $D$ and put $\omega_{0}=\omega_{D}(g_{0})$. Then, the proof of Lemma~\ref{lemm:1} 
shows that there exists an $\omega_{0}$-plurisubharmonic function $\varphi$ that agrees with $g-g_{0}$ almost everywhere; in other words, 
the function $g'=g_{0}+\varphi$ agrees with $g$ almost everywhere. Locally, on $V$, we now have
\begin{displaymath}
g'+\log\vert f_{D}\vert^{2}= g_{0}+\log\vert f_{D}\vert^{2}+\varphi.
\end{displaymath}
Since $g_{0}+\log\vert f_{D}\vert^{2}$ is smooth and $\varphi$ is $\omega_{0}$-plurisubharmonic, we deduce that $g'+\log\vert f_{D}\vert^{2}$ 
is upper semicontinuous, does not take the value $+\infty$, and that $\ddc(g'+\log\vert f_{D}\vert^{2})$ is positive on $V$. Thus, $g'+\log\vert 
f_{D}\vert^{2}$ is plurisubharmonic on $V$. 
\end{remark}

We can extend Lemma~\ref{lem:greeneff} to divisors with real coefficients. 

\begin{lemma}
\label{lem:greeneffreal}
Let $D$ and $D'$ be divisors on $X$ with real coefficients such that the divisor $C\coloneqq D'-D$ is effective. If $g$ is a Green function for 
$D$ of plurisubharmonic type, then $g$ also is a Green function for $D'$ of plurisubharmonic type.
\end{lemma}
\begin{proof}
Since the divisor $C=D'-D$ is effective, the current $\delta_{C}$ is positive. Assume now that $g$ is a Green function for $D$ of plurisubharmonic 
type. Then, we have
\begin{displaymath}
\omega_{D'}(g)=\ddc g+\delta_{D'}=\ddc g+\delta_{D}+\delta_{C}=\omega_{D}(g)+\delta_{C}\ge 0.
\end{displaymath}
Hence, $g$ is a Green function for $D'$ of plurisubharmonic type. 
\end{proof}

\subsection{Non-pluripolar product of currents}

In general, currents cannot be multiplied. Nevertheless, the Bedford--Taylor calculus allows us to define currents of the form
\begin{displaymath}
\ddc\varphi_{1}\wedge\ldots\wedge\ddc\varphi_{p}\qquad\text{and}\qquad\varphi_{0}\,\ddc\varphi_{1}\wedge\ldots\wedge\ddc\varphi_{p},
\end{displaymath}
whenever the functions $\varphi_{0},\varphi_{1},\ldots,\varphi_{p}$ are locally bounded plurisubharmonic functions on an open subset $V
\subseteq X$. In order to extend these products to plurisubharmonic functions that are not locally bounded on $V$, one can follow two directions. 

One direction consists in imposing conditions on the sets, where the functions $\varphi_{0},\varphi_{1},\ldots,\varphi_{p}$ take the value 
$-\infty$. For this approach, we refer for instance to~\cite[Chapter~III, \S~3]{demailly12:cadg}.

The other direction is to ``throw away'' the part of the currents that is supported on the so-called pluripolar sets. This point of view is introduced 
in~\cite{BEGZ}. We recall this construction next. 

\begin{definition}
Let $Z\subseteq X$ be a subset of $X$. Then, $Z$ is called a \emph{complete pluripolar set}, if for each $x\in Z$, there exists a connected open 
subset $V_{x}\subseteq X$ containing $x$ and a plurisubharmonic function $\varphi_{x}\colon V_{x}\rightarrow\mathbb{R}\cup\{-\infty\}$, not 
identically equal to $-\infty$, such that
\begin{displaymath}
Z\cap V_{x}=\{z\in V_{x}\,\vert\,\varphi_{x}(z)=-\infty\}.
\end{displaymath}
We note that it can be shown that any proper analytic subvariety $Z$ of $X$ is a complete pluripolar set. Any subset of a complete pluripolar set 
is called a \emph{pluripolar set}.  
\end{definition}

\begin{definition}
A function $\varphi\colon X\rightarrow\mathbb{R}\cup\{-\infty\}$ is said to \emph{have small unbounded locus}, if there exists a complete 
pluripolar set $Z$ such that $\varphi$ is locally bounded outside $Z$.
\end{definition}

\begin{definition}
\label{def:3}
For $j=1,\ldots,p$, let $\omega_{j}$ be closed smooth real $(1,1)$-forms on $X$ and let $\varphi_{j}$ be $\omega_{j}$-plurisubharmonic functions 
on $X$ having small unbounded locus. Then, the \emph{non-pluripolar product of $(1,1)$-currents}
\begin{equation}
\label{eq:npp}
\big\langle(\ddc\varphi_{1}+\omega_{1})\wedge\ldots\wedge(\ddc\varphi_{p}+\omega_{p})\big\rangle
\end{equation}
is the $(p,p)$-current defined as follows: For each $k\in\N_{>0}$, we introduce the set
\begin{displaymath}
U_{k}\coloneqq\{x\in X\,\vert\,\varphi_{j}(x)\ge -k\text{ for }j=1,\ldots,p\};
\end{displaymath}
we observe that the functions $\varphi_{j}$ are locally bounded on $U_{k}$ for $j=1,\ldots,p$. Then, letting $q\coloneqq d-p$ and $\eta$ be a 
smooth $(q,q)$-form on $X$, the non-pluripolar product~\eqref{eq:npp} is defined as
\begin{equation}
\label{eq:1}
\big\langle(\ddc\varphi_{1}+\omega_{1})\wedge\ldots\wedge(\ddc\varphi_{p}+\omega_{p})\big\rangle(\eta)\coloneqq\lim_{k\to\infty}\int_{U_{k}}
(\ddc\varphi_{1}+\omega_{1})\wedge\ldots\wedge(\ddc\varphi_{p}+\omega_{p})\wedge\eta,
\end{equation}
where the right-hand side of equation~\eqref{eq:1} is defined using Bedford--Taylor theory. 
\end{definition}

The basic properties of the non-pluripolar product of $(1,1)$-currents just defined are the following. They follow from~\cite[Proposition~1.6]
{BEGZ}, \cite[Proposition~1.4]{BEGZ}, and~\cite[Theorem~1.8]{BEGZ}, because we are assuming that $X$ is projective, and hence is K\"ahler. 

\begin{theorem}
\label{thm:2}
With the notations of Definition~\ref{def:3}, we set $T_{j}\coloneqq\ddc\varphi_{j}+\omega_{j}$ for $j=1,\ldots,p$. Then, we have the following 
statements:
\begin{enumerate}
\item 
\label{item:3}
The non-pluripolar product $\big\langle T_{1}\wedge\ldots\wedge T_{p}\big\rangle$ is well-defined and depends only on the currents $T_{j}$ and 
not on the particular choice of functions $\varphi_{j}$ for $j=1,\ldots,p$.
\item 
\label{item:4}
The non-pluripolar product $\big\langle T_{1}\wedge\ldots\wedge T_{p}\big\rangle$ is symmetric. Moreover, it is multilinear in the following sense: 
If, for example, $\omega_{1}'$ is another closed smooth real $(1,1)$-form on $X$, $\varphi_{1}'$ is an $\omega_{1}'$-plurisubharmonic function on 
$X$ having small unbounded locus, $T_{1}'\coloneqq\ddc\varphi_{1}'+\omega_{1}'$, and $\lambda_{1},\lambda_{1}'\in\R_{>0}$, then we have
\begin{displaymath}
\big\langle(\lambda_{1}T_{1}+\lambda_{1}' T_{1}')\wedge\ldots\wedge T_{p}\big\rangle=\lambda_{1}\big\langle T_{1}\wedge\ldots\wedge T_{p}
\big\rangle+\lambda_{1}'\big\langle T_{1}'\wedge\ldots\wedge T_{p}\big\rangle.
\end{displaymath}
\item 
\label{item:5}
The current $\big\langle T_{1}\wedge\ldots\wedge T_{p}\big\rangle$ is closed and positive. 
\end{enumerate}
\end{theorem}

Since $\big\langle T_{1}\wedge\ldots\wedge T_{p}\big\rangle$ is closed, we will denote by $\cl\big(\big\langle T_{1}\wedge\ldots\wedge T_{p}\big
\rangle\big)$ the associated cohomology class in $H^{p,p}(X,\C)$. 

The following monotonicity criterion is contained in~\cite[Theorem~1.16]{BEGZ}.

\begin{theorem}
\label{thm:1}
For $j=1,\ldots,p$, let $\omega_{j}$ be closed smooth real $(1,1)$-forms on $X$ and let $\varphi_{j},\varphi_{j}'$ be $\omega_{j}$-plurisubharmonic 
functions on $X$ having small unbounded locus. If there are constants $c_{j}\in\R_{>0}$ such that $\varphi_{j}'\le\varphi_{j}+c_{j}$, i.\,e., the functions 
$\varphi_{j}'$ are more singular than the functions $\varphi_{j}$ for $j=1,\ldots,p$, then the cohomology class  
\begin{displaymath}
\cl\big(\big\langle(\ddc\varphi_{1}+\omega_{1})\wedge\ldots\wedge(\ddc\varphi_{p}+\omega_{p})\big\rangle\big)-\cl\big(\big\langle(\ddc\varphi_{1}'+
\omega_{1})\wedge\ldots\wedge(\ddc\varphi_{p}'+\omega_{p})\big\rangle\big)
\end{displaymath}
is strongly positive. This means that for any closed smooth and weakly positive form $\eta$, the inequality 
\begin{displaymath}
\int_{X}\big\langle(\ddc\varphi_{1}+\omega_{1})\wedge\ldots\wedge(\ddc\varphi_{p}+\omega_{p})\big\rangle\wedge\eta-\int_{X}\big\langle(\ddc
\varphi_{1}'+\omega_{1})\wedge\ldots\wedge(\ddc\varphi_{p}'+\omega_{p})\big\rangle\wedge\eta\ge 0
\end{displaymath}
holds. In particular, if $p=d$, the inequality 
\begin{displaymath}
\int_{X}\big\langle(\ddc\varphi_{1}'+\omega_{1})\wedge\ldots\wedge(\ddc\varphi_{d}'+\omega_{d})\big\rangle\le\int_{X}\big\langle(\ddc\varphi_{1}+
\omega_{1})\wedge\ldots\wedge(\ddc\varphi_{d}+\omega_{d})\big\rangle
\end{displaymath}
holds.
\end{theorem}

The non-pluripolar product of $(1,1)$-currents given in Definition~\ref{def:3} can be used to define a non-pluripolar product of $(1,1)$-currents of 
the form $\omega_{D}(g)$ associated to Green functions $g$ for a divisor $D$ on $X$ which are of plurisubharmonic type.

\begin{definition}
\label{def:2}
For $j=1,\ldots,p$, let $D_{j}$ be divisors on $X$ and let $g_{j}$ be Green functions for $D_{j}$ of plurisubharmonic type having small unbounded 
locus. Furthermore, let $g_{j,0}$ be Green functions for $D_{j}$ of smooth type, and put $\omega_{j,0}\coloneqq \omega_{D_{j}}(g_{j,0})$. Then, 
the difference $\varphi_{j}\coloneqq g_{j}-g_{j,0}$ is an $\omega_{j,0}$-plurisubharmonic function on $X$. We thus define
\begin{displaymath}
\big\langle\omega_{D_{1}}(g_{1})\wedge\ldots\wedge\omega_{D_{p}}(g_{p})\big\rangle\coloneqq\big\langle(\ddc\varphi_{1}+\omega_{1,0})\wedge
\ldots\wedge(\ddc\varphi_{p}+\omega_{p,0})\big\rangle.
\end{displaymath}
\end{definition}

We note that by part~\ref{item:3} of Theorem~\ref{thm:2}, the non-pluripolar product introduced above is well-defined and does not depend on the 
choice of Green functions $g_{j,0}$ for $D_{j}$ of smooth type ($j=1,\ldots,p$). What is surprising is that it only depends on the Green functions 
$g_{j}$, but not on the actual divisors $D_{j}$ as long as $g_{j}$ are Green functions for $D_{j}$ of plurisubharmonic type ($j=1,\ldots,p$).

\begin{proposition}
\label{prop:5}
For $j=1,\ldots,p$, let $D_{j}$, $D_{j}'$ be divisors on $X$ and let $g_{j}$ be Green functions for $D_{j}$ of plurisubharmonic type such that they 
are also Green functions for $D_{j}'$ of plurisubharmonic type. Then, we have the equality
\begin{displaymath}
\big\langle\omega_{D_{1}'}(g_{1})\wedge\ldots\wedge\omega_{D_{p}'}(g_{p})\big\rangle=\big\langle\omega_{D_{1}}(g_{1})\wedge\ldots\wedge
\omega_{D_{p}}(g_{p})\big\rangle.
\end{displaymath}
\end{proposition}
\begin{proof}
We can find divisors $D''_{j}$ such that $D''_{j}-D_{j}$ and $D''_{j}-D'_{j}$ are effective. It is then enough to prove the two equalities
\begin{displaymath}
\big\langle\omega_{D_{1}'}(g_{1})\wedge\ldots\wedge\omega_{D_{p}'}(g_{p})\big\rangle=\big\langle\omega_{D_{1}''}(g_{1})\wedge\ldots\wedge
\omega_{D_{p}''}(g_{p})\big\rangle=\big\langle\omega_{D_{1}}(g_{1})\wedge\ldots\wedge\omega_{D_{p}}(g_{p})\big\rangle.
\end{displaymath}
Hence, without loss of generality, we can assume that we are in the case when $C_{j}\coloneqq D'_{j}-D_{j}$ is effective for $j=1,\dots,p$.

By construction, we have that $\omega_{D_{j}'}(g_{j})=\omega_{D_{j}}(g_{j})+\delta_{C_{j}}$ for $j=1,\ldots,p$. By part~\ref{item:4} of Theorem~\ref
{thm:2}, we are thus reduced to prove that
\begin{displaymath}
\big\langle\delta_{C_{1}}\wedge\omega_{D_{2}}(g_{2})\wedge\dots\wedge\omega_{D_{p}}(g_{p})\big\rangle=0.
\end{displaymath}
However, this follows directly from the definition of the non-pluripolar product of $(1,1)$-currents: For this, let $g$ be a Green function for $C_{1}$ 
of smooth type and let $\omega\coloneqq\omega_{C_{1}}(g)$, so that $\omega$ is a smooth $(1,1)$-form. Now, put $\varphi=-g$, so that $\varphi$ 
becomes an $\omega$-plurisubharmonic function on $X$ satisfying $\ddc\varphi+\omega=\delta_{C_{1}}$. On any of the sets $U_{k}$ introduced 
in Definition~\ref{def:3}, the function $\varphi$ will be bounded below. Therefore, we have $U_{k}\cap C_{1}=\emptyset$, which implies that after 
replacing $\varphi_{1}$ and $\omega_{1}$ in equation~\eqref{eq:1} by $\varphi$ and $\omega$, respectively, all the integrals on the right-hand 
side of~\eqref{eq:1} will vanish. This concludes the proof of the proposition.
\end{proof}

\begin{example}
Let $D$ and $D'$ be divisors on $X$ such that the divisor $C\coloneqq D'-D$ is effective. Let $\Vert\cdot\Vert$ be a semipositive hermitian metric 
on $\caO(D)$ and define the hermitian metric $\Vert\cdot\Vert'$ on $\caO(D')$ as in Remark~\ref{rem:metrics}; by Lemma~\ref{lem:metrics}, we 
know that $\Vert\cdot\Vert'$ is also semipositive. Choosing a non-trivial rational section $s$ of $\caO(D)$, which can also be viewed as a non-trivial 
rational section of $\caO(D')$, we obtain the relation (see the proof of Lemma~\ref{lem:metrics})
\begin{displaymath}
\ddc(-\log\Vert s\Vert'^{\,2})=\ddc(-\log\Vert s\Vert^{2})+\delta_{C},
\end{displaymath}
which shows that we have the following equality of non-pluripolar parts
\begin{displaymath}
\big\langle\ddc(-\log\Vert s\Vert'^{\,2})\big\rangle=\big\langle\ddc(-\log\Vert s\Vert^{2})\big\rangle.
\end{displaymath}
Letting $\varphi$ and $\varphi'$ denote the $\omega_{0}$-plurisubharmonic functions on $X$ corresponding to the semipositive hermitian metrics 
$\Vert\cdot\Vert$ and $\Vert\cdot\Vert'$, respectively, we can rewrite the above relation as
\begin{displaymath}
\ddc\varphi'+\omega_{0}=\ddc\varphi+\omega_{0}+\delta_{C},
\end{displaymath}
which shows that the $\omega_{0}$-plurisubharmonic function $\varphi'$ is ``more singular'' than  the $\omega_{0}$-plurisub\-harmonic function 
$\varphi$; the equality of non-pluripolar parts translates as
\begin{displaymath}
\big\langle\ddc\varphi'+\omega_{0}\big\rangle=\big\langle\ddc\varphi+\omega_{0}\big\rangle.
\end{displaymath}
\end{example}

\subsection{Capacity and monotonicity of non-pluripolar products}
\label{sec:capac-monot-non}

The monotonicity result given in Theorem~\ref{thm:1} has been strengthened in~\cite{darvas18:mnpp} to obtain weak convergence of measures. 
We need some preliminary definitions (see~\cite[Chapter~IX]{guedj17:_degen_monge}).

\begin{definition}
Let $\omega $ be a K\"ahler form on  $X$. Write $V=\int_{S}\omega^{d}$. For every Borel subset $K\subseteq X$, the \emph{$\omega$-capacity} 
of $K$ is defined as
\begin{displaymath}
\Capa_{\omega}(K)\coloneqq\sup\bigg\{\frac{1}{V}\int_{K}(\ddc\varphi+\omega)^{\wedge d}\,\bigg\vert\,\varphi\in\PSH(X,\omega),\,0\le\varphi\le 
1\bigg\}. 
\end{displaymath}
If $K \subseteq X$ is any subset, the \emph{outer $\omega$-capacity} of $K$ is defined as
\begin{displaymath}
\Capa^{\ast}_{\omega}(K)\coloneqq\inf\big\{\Capa_{\omega}(W)\,\big\vert\,W\text{ open},\,K\subseteq W\big\}.
\end{displaymath}
\end{definition}

\begin{definition}
A sequence of functions $(\varphi_{n})_{n\in\N}$ is said to \emph{converge in capacity} to a function $\varphi$, if for any $\delta>0$, we have
\begin{displaymath}
\lim_{n\to\infty}\Capa_{\omega}^{\ast}\big\{x\in X\,\big\vert\,\vert\varphi_{n}(x)-\varphi(x)\vert>\delta\big\}=0.
\end{displaymath}
\end{definition}

\begin{example}
Let $(\varphi_{n})_{n\in\N}$ be a sequence of plurisubharmonic functions that converges monotonically almost everywhere to a plurisubharmonic 
function $\varphi$. Then, the sequence $(\varphi_{n})_{n\in\N}$ converges in capacity to $\varphi$. See~\cite[Proposition~4.25]
{guedj17:_degen_monge}. 
\end{example}

The following result is~\cite[Theorem~1.2]{darvas18:mnpp}.

\begin{theorem} 
\label{thm:5}
For $j=1,\ldots,p$, let $\omega_{j}$ be closed smooth real $(1,1)$-forms on $X$ and let $\varphi_{j}$ be $\omega_{j}$-plurisubharmonic functions 
on $X$. Furthermore, let $(\varphi_{j,n})_{n\in\N}$ be a sequence of $\omega_{j}$-plurisubharmon\-ic functions on $X$ converging in capacitiy to 
$\varphi_{j}$ and satisfying
\begin{displaymath}
\int_{X}\big\langle(\ddc\varphi_{1}+\omega_{1})\wedge\ldots\wedge(\ddc\varphi_{d}+\omega_{d})\big\rangle\ge\limsup_{n\to\infty}\int_{X}\big
\langle(\ddc\varphi_{1,n}+\omega _{1})\wedge\ldots\wedge(\ddc\varphi_{d,n}+\omega _{d})\big\rangle.
\end{displaymath}
Then, the sequence of measures $\big(\big\langle(\ddc\varphi_{1,n}+\omega_{1})\wedge\ldots\wedge(\ddc\varphi_{d,n}+\omega_{d})\big\rangle
\big)_{n\in\N}$ converges weakly to the measure $\big\langle(\ddc\varphi_{1}+\omega_{1})\wedge\ldots\wedge(\ddc\varphi_{d}+\omega_{d})\big
\rangle$.
\end{theorem}

\subsection{Relative energy}
\label{sec:relative-energy}

In this subsection, we assume that $\omega$ is a K\"ahler form on $X$ and that $\omega_{0}$ is a smooth closed positive $(1,1)$-form on  $X$. 
We start by recalling the definitions of relative full mass and relative energy following~\cite{darvas18:mnpp}.

\begin{definition} 
\label{def:rfm}
Let $\psi\in\PSH(X,\omega_{0})$ having small unbounded locus. The space $\caE(X,\omega_{0},\psi)$ of \emph{functions with relative full mass 
(with respect to $\psi $)} is the subspace of $\PSH(X,\omega_{0})$ consisting of those functions $\varphi$, which are more singular than $\psi$ 
and such that
\begin{displaymath}
\int_{X}\big\langle(\ddc\varphi+\omega_{0})^{\wedge d}\big\rangle=\int_{X}\big\langle(\ddc\psi+\omega_{0})^{\wedge d}\big\rangle.
\end{displaymath}
\end{definition}

\begin{definition}
\label{def:dar}
Let $\psi\in\PSH(X,\omega_{0})$ and $\varphi\in\caE(X,\omega_{0},\psi)$ with $[\varphi]=[\psi]$, i.\,e., $\varphi-\psi$ is bounded, then the \emph
{relative energy of $\varphi$ (with respect to $\psi $)} is defined as
\begin{displaymath}
I_{\psi}(\varphi)\coloneqq\frac{1}{d+1}\sum_{k=0}^{d}\,\int_{X}(\varphi-\psi)\big\langle(\ddc\varphi+\omega_{0})^{\wedge k}\wedge(\ddc\psi+
\omega_{0})^{\wedge(d-k)}\big\rangle.
\end{displaymath}
Moreover, if $\varphi\in\PSH(X,\omega_{0})$ does not satisfy $[\varphi]=[\psi]$, its relative energy is defined as 
\begin{displaymath} 
I_{\psi}(\varphi)\coloneqq\inf\big\{I_{\psi}(u)\,\big\vert\,u\in\caE(X,\omega_{0},\psi),\,[u]=[\psi],\,\varphi\le u\big\}. 
\end{displaymath}
With this notion at hand, the space $\caE^{1}(X,\omega_{0},\psi)$ of \emph{functions of finite relative energy} is then defined as the subspace 
of $\caE(X,\omega_{0},\psi)$ consisting of those functions whose relative energy is finite.  
\end{definition}

We have the following characterization of functions with finite relative energy, whose proof is given in~\cite[Lemma 5.6]{darvas23:_relat_kaehl}.
\begin{lemma}
\label{lemm:10}
Let $\varphi\in\PSH(X,\omega_{0})$ with $[\varphi]\le[\psi]$. Then, $\varphi\in\caE^{1}(X,\omega_{0},\psi )$ if and only if $\varphi\in\caE(X,\omega_
{0},\psi )$ and
\begin{displaymath}
\int_{X}(\varphi-\psi)\langle(\ddc\varphi+\omega_{0})^{\wedge d}\rangle>-\infty.
\end{displaymath}
\end{lemma}

\begin{remark}
\label{rem:7}
The spaces $\caE(X,\omega_{0},\psi)$ and $\caE^{1}(X,\omega_{0},\psi)$ are convex. Moreover, the spaces \linebreak $\caE(X,\omega_{0},\psi)$ 
also satisfy a convexity condition when we change the reference functions $\psi$. Namely, assume that $\psi\in\PSH(X,\omega_{0})$ and $\psi'\in 
\PSH(X,\omega'_{0})$. Under the assumption that
\begin{displaymath}
\int_{X}\big\langle(\ddc\psi+\omega_{0})^{\wedge d}\big\rangle>0\qquad\text{and}\qquad\int_{X}\big\langle(\ddc\psi'+\omega'_{0})^{\wedge d}\big
\rangle>0,
\end{displaymath}
one can prove that given $\varphi\in\caE(X,\omega_{0},\psi)$ and $\varphi'\in\caE(X,\omega'_{0},\psi')$, then the sum $\varphi+\varphi'\in\linebreak
\caE(X,\omega_{0}+\omega'_{0},\psi+\psi')$. Nevertheless, this convexity property is no longer true for the spaces $\caE^{1}(X,\omega_{0},\psi)$. 
That is, given $\varphi\in\caE^{1}(X,\omega_{0},\psi)$ and $\varphi'\in\caE^{1}(X,\omega'_{0},\psi')$, then the sum $\varphi+\varphi'$ need not 
necessarily belong to $\caE^{1}(X,\omega_{0}+\omega'_{0},\psi+\psi')$. See~\cite[Example~4.5]{di15:_stabil_monge} for an example.  
\end{remark}

Before we introduce next the mixed version of the relative energy, we need some additional notation.

\begin{notation}
Given smooth closed positive $(1,1)$-forms $\omega_{j}$ on $X$ for $j=0,\ldots,d$, we write $\pmb{\omega}=(\omega_{0},\ldots,\omega_{d})$ and
$\PSH(X,\pmb{\omega})=\PSH(X,\omega_{0})\times\ldots\times\PSH(X,\omega_{d})$. Next, given functions $\psi_{j}\in\PSH(X,\omega_{j})$ of small
unbounded locus, we write $\pmb{\psi}=(\psi_{0},\ldots,\psi_{d})$ to obtain $\pmb{\psi}\in\PSH(X,\pmb{\omega})$ of small unbounded locus; we also
write $\caE(X,\pmb{\omega},\pmb{\psi})=\caE(X,\omega_{0},\psi_{0})\times\ldots\times\caE(X,\omega_{d},\psi_{d})$. Finally, given functions $\varphi_
{j}\in\caE(X,\omega_{j},\psi_{j})$, we write  $\pmb{\varphi}=(\varphi_{0},\ldots,\varphi_{d})$ to obtain $\pmb{\varphi}\in\caE(X,\pmb{\omega},\pmb{\psi})$.
The notation $\pmb{\varphi}\le\pmb{\psi}$ means $\varphi_{j}\le\psi_{j}$; similarly, the notation $[\pmb{\varphi}]=[\pmb{\psi}]$ means $[\varphi_{j}]=
[\psi_{j}]$.
\end{notation}

\begin{definition}
\label{def:mre}
Let $\pmb{\omega}=(\omega_{0},\ldots,\omega_{d})$ be smooth closed positive $(1,1)$-forms on $X$ and let $\pmb{\psi}\in\PSH(X,\pmb{\omega})$ 
be of small unbounded locus. For $\pmb{\varphi}\in\caE(X,\pmb{\omega},\pmb{\psi})$ with $[\pmb{\varphi}]=[\pmb{\psi}]$, we define the \emph{mixed
relative energy of $\pmb{\varphi}$ (with respect to $\pmb{\psi}$)} by
\begin{align}
\notag
I_{\pmb{\psi}}(\pmb{\varphi})\coloneqq\sum_{k=0}^{d}\,\int_{X}(\varphi_{k}-\psi_{k})\big\langle&(\ddc\varphi_{0}+\omega_{0}\wedge\ldots\wedge(\ddc
\varphi_{k-1}+\omega_{k-1})\wedge \\
\label{keyfor} 
&\wedge(\ddc\psi_{k+1}+\omega_{k+1})\wedge\ldots\wedge(\ddc\psi_{d}+\omega_{d})\big\rangle.
\end{align}
When $\pmb{\psi}$ is clear from the context, we will call $I_{\pmb{\psi}}(\pmb{\varphi})$ simply the mixed relative energy of $\pmb{\varphi}$.
\end{definition}

\begin{lemma}
\label{lem:tech}
With the notations and assumptions of Definition~\ref{def:mre}, the following statements hold:
\begin{enumerate}
\item
\label{item:21}
The mixed relative energy $I_{\pmb{\psi}}(\pmb{\varphi})$ is symmetric with respect to the simultaneous permutation of $\varphi_{0},\ldots,\varphi_{d}$ 
and $\psi_{0},\ldots,\psi_{d}$ as well as of $\omega_{0},\ldots,\omega_{d}$, i.\,e., we have 
\begin{equation}
\label{perm}
I_{\sigma(\pmb{\psi})}(\sigma(\pmb{\varphi}))=I_{\pmb{\psi}}(\pmb{\varphi})
\end{equation}
for every permutation $\sigma$ of the symmetric group $S_{d+1}$ of the set $\{0,\ldots,d\}$, where $\sigma(\pmb{\varphi})$ means $(\varphi_{\sigma(0)},
\ldots,\varphi_{\sigma(d)})$.
\item
\label{item:22}
The mixed relative energy $I_{\pmb{\psi}}(\pmb{\varphi})$ is multilinear in the following sense: If, for example, $\omega'_{0}$ is another smooth closed 
positive $(1,1)$-form on $X$, $\psi'_{0}\in\PSH(X,\omega'_{0})$ is of small unbounded locus, $\varphi'_{0}\in\caE(X,\omega'_{0},\psi'_{0})$ with $[\varphi'_
{0}]=[\psi'_{0}]$, and $\lambda_{0},\lambda'_{0}\in\R_{>0}$, we have
\begin{displaymath}
I_{\lambda_{0}\psi_{0}+\lambda'_{0}\psi'_{0},\ldots,\psi_{d}}(\lambda_{0}\varphi_{0}+\lambda'_{0}\varphi'_{0},\ldots,\varphi_{d})=\lambda_{0}I_{\psi_{0},
\ldots,\psi_{d}}(\varphi_{0},\ldots,\varphi_{d})+\lambda'_{0}I_{\psi'_{0},\ldots,\psi_{d}}(\varphi'_{0},\ldots,\varphi_{d}).
\end{displaymath}
\item 
\label{item:24} 
The mixed relative energy $I_{\pmb{\psi}}(\pmb{\varphi})$ is monotonous with respect to the functions $\pmb{\varphi}$. That is, if $\pmb{\varphi}'\in\caE
(X,\pmb{\omega},\pmb{\psi})$ with $[\pmb{\varphi}']=[\pmb{\psi}]$and $\pmb{\varphi}'\le\pmb{\varphi}$, then
\begin{displaymath}
I_{\pmb{\psi}}(\pmb{\varphi}')\le I_{\pmb{\psi}}(\pmb{\varphi}). 
\end{displaymath}
\item
\label{item:23}
Given $\pmb{\varphi}\in\caE(X,\pmb{\omega},\pmb{\psi})$ with $[\pmb{\varphi}]=[\pmb{\psi}]$, let $(\pmb{\varphi_{n}})_{\pmb{n}\in\N^{d+1}}$ be 
sequences of $(d+1)$-tuples of functions in $\caE(X,\pmb{\omega},\pmb{\psi})$ with $[\pmb{\varphi_{n}}]=[\pmb{\psi}]$, converging componentwise
from above to $\pmb{\varphi}$. Then, the net of mixed relative energies $(I_{\pmb{\psi}}(\pmb{\varphi_{n}}))_{\pmb{n}\in\N^{d+1}}$ converges from 
above to the mixed relative energy $I_{\pmb{\psi}}(\pmb{\varphi})$, i.\,e., we have
\begin{displaymath}
\lim_{\pmb{n}\to\pmb{\infty}}I_{\pmb{\psi}}(\pmb{\varphi_{n}})=I_{\pmb{\psi}}(\pmb{\varphi}).
\end{displaymath}
\end{enumerate}
\end{lemma}
\begin{proof}
We start with the proof of~\ref{item:21}. We note that it is enough to prove that the claimed equality~\eqref{perm} holds for the transpositions $\tau_
{j,j+1}\in S_{d+1}$ interchanging $\varphi_{j},\psi_{j},\omega_{j}$ with $\varphi_{j+1},\psi_{j+1},\omega_{j+1}$ for $j=0,\ldots,d-1$ and leaving the 
remaining entries fixed, since these transpositions generate the whole symmetric group. For this, we recall from~\cite[Theorem~1.14]{BEGZ} applied 
to the bounded functions $\varphi_{j}-\psi_{j}$ and $\varphi_{j+1}-\psi_{j+1}$, and the symmetry of the non-pluripolar product, the equality
\begin{displaymath}
\int_{X}(\varphi_{j}-\psi_{j})\big\langle\ddc(\varphi_{j+1}-\psi_{j+1})\wedge\Theta\big\rangle=\int_{X}(\varphi_{j+1}-\psi_{j+1})\big\langle\ddc(\varphi_{j}-
\psi_{j})\wedge\Theta\big\rangle,
\end{displaymath}
where $\Theta$ stands for the product of currents
\begin{displaymath}
\Theta=(\ddc\varphi_{0}+\omega_{0})\wedge\ldots\wedge(\ddc\varphi_{j-1}+\omega_{j-1})\wedge(\ddc\psi_{j+2}+\omega_{j+2})\wedge\ldots\wedge
(\ddc\psi_{d}+\omega_{d}).
\end{displaymath}
Rearranging terms, this leads to
\begin{multline*}
\int_{X}(\varphi_{j}-\psi_{j})\big\langle(\ddc\psi_{j+1}+\omega_{j+1})\wedge\Theta\big\rangle+(\varphi_{j+1}-\psi_{j+1})\big\langle(\ddc\varphi_{j}+
\omega_{j})\wedge\Theta\big\rangle \\[1mm]
=\int_{X}(\varphi_{j+1}-\psi_{j+1})\big\langle(\ddc\psi_{j}+\omega_{j})\wedge\Theta\big\rangle+(\varphi_{j}-\psi_{j})\big\langle(\ddc\varphi_{j+1}+
\omega_{j+1})\wedge\Theta\big\rangle.
\end{multline*}
The latter equality immediately implies the claim for the transposition $\tau_{j,j+1}$. 

The proof of~\ref{item:22} follows immediately from the multilinearity of the non-pluripolar product of $(1,1)$-currents stated in part~\ref{item:4} of 
Theorem~\ref{thm:2}. 

We next prove~\ref{item:24}. By the symmetry of the mixed relative energy, it is enough to prove that
\begin{displaymath}
I_{\psi_{0},\ldots,\psi_{d}}(\varphi_{0},\ldots,\varphi_{d})-I_{\psi_{0},\ldots,\psi_{d}}(\varphi'_{0},\ldots,\varphi_{d})\ge 0.
\end{displaymath}
A repeated use of the integration by parts~\cite[Theorem~1.14]{BEGZ} shows that
\begin{align}
\notag
I_{\psi_{0},\ldots,\psi_{d}}(\varphi_{0},\ldots,\varphi_{d})&-I_{\psi_{0},\ldots,\psi_{d}}(\varphi'_{0},\ldots,\varphi_{d}) \\[1mm]
\label{eq:16}
&=\int_{X}(\varphi_{0}-\varphi_{0}')\big\langle(\ddc\varphi_{1}+\omega_{1})\wedge\ldots\wedge(\ddc\varphi_{d}+\omega_{d})\big\rangle.
\end{align}
Since $\varphi_{0}-\varphi_{0}'\ge 0$ and the non-pluripolar product defines a positive measure, we deduce the claim. 

We finally prove~\ref{item:23}. In order to prove convergence in the sense of nets, by the monotonicity property~\ref{item:24}, it is enough to prove 
iterated convergence. This can be done using one function at the time. Again by the symmetry property~\ref{item:21}, it is thus enough to show that 
\begin{equation}
\label{eq:7}
\lim_{n_{0}\to\infty}I_{\psi_{0},\ldots,\psi_{d}}(\varphi_{0,n_{0}},\ldots,\varphi_{d})=I_{\psi_{0},\ldots,\psi_{d}}(\varphi_{0},\ldots,\varphi_{d}).
\end{equation}
As in~\eqref{eq:16}, we have that
\begin{align*}
I_{\psi_{0},\ldots,\psi_{d}}(\varphi_{0,n_{0}},\ldots,\varphi_{d})&-I_{\psi_{0},\ldots,\psi_{d}}(\varphi_{0},\ldots,\varphi_{d}) \\[1mm]
&=\int_{X}(\varphi_{0,n_{0}}-\varphi_{0})\big\langle(\ddc\varphi_{1}+\omega_{1})\wedge\ldots\wedge(\ddc\varphi_{d}+\omega_{d})\big\rangle.
\end{align*}
The claimed equality~\eqref{eq:7} now follows from this and the theorem of dominated convergence. 
\end{proof}

\begin{definition}
\label{def:finen}
Let $\pmb{\omega}=(\omega_{0},\ldots,\omega_{d})$ be smooth closed positive $(1,1)$-forms on $X$ and let $\pmb{\psi}\in\PSH(X,\pmb{\omega})$ 
be of small unbounded locus. For arbitray $\pmb{\varphi}\in\caE(X,\pmb{\omega},\pmb{\psi})$, we define the \emph{mixed relative energy of 
$\pmb{\varphi}$ (with respect to $\pmb{\psi}$)} by
\begin{displaymath}
I_{\pmb{\psi}}(\pmb{\varphi})\coloneqq\inf\big\{I_{\pmb{\psi}}(\pmb{u})\,\big\vert\,\pmb{u}\in\caE(X,\pmb{\omega},\pmb{\psi}),\,[\pmb{u}]=[\pmb{\psi}],
\,\pmb{\varphi}\le\pmb{u}\big\}. 
\end{displaymath}
The space $\caE^{1}(X,\pmb{\omega},\pmb{\psi})$ of \emph{functions of finite mixed relative energy} is then defined as the subspace of $\caE(X,
\pmb{\omega},\pmb{\psi})$ consisting of those tuples of functions $\pmb{\varphi}$ whose mixed relative energy is finite.
\end{definition}

\begin{remark}
Given functions $\varphi_{j}\in\caE^{1}(X,\omega_{j},\psi_{j})$ for $j=0,\ldots,d$, Remark~\ref{rem:7} then shows that it is not necessarily true that 
$\pmb{\varphi}=(\varphi_{0},\ldots,\varphi_{d})$ belongs to $\caE^{1}(X,\pmb{\omega},\pmb{\psi})$. Therefore, the properties of the mixed relative 
energy have to be proven anew. Anyhow, the proofs are based on the corresponding properties of the usual relative energy. 
\end{remark}

For every $\varphi\in\caE(X,\omega,\psi )$ and real number $C\ge 0$, the \emph{canonical approximants} are defined as 
\begin{displaymath}
\varphi^{(C)}\coloneqq\max(\varphi,\psi-C).
\end{displaymath}

We claim that $\varphi^{(C)}\in\caE(X,\omega,\psi)$ and that $[\varphi^{(C)}]=[\psi]$. Indeed, since $\varphi^{(C)}$ is a maximum of two plurisubharmonic
functions, it is also plurisubharmonic, whence $\varphi^{(C)}\in\PSH(X,\omega)$. Moreover, since $\varphi\in \caE(X,\omega,\psi)$, it is more singular 
than $\psi$, from which we derive that $\varphi^{(C)}$ has also to be more singular than $\psi$. Since, by definition $\psi\le\varphi^{(C)}+C$, we deduce
that $\psi$ and $\varphi^{(C)}$ have equivalent singularities, that is, $\varphi^{(C)}-\psi$ is bounded. Whence, taking also into account Theorem~\ref
{thm:1}, we deduce that $\varphi^{(C)}\in\caE(X,\omega,\psi)$ and that $[\varphi^{(C)}]=[\psi]$.  

\begin{proposition}
\label{prop:10}
Let $\pmb{\omega}=(\omega_{0},\ldots,\omega_{d})$ be smooth closed positive $(1,1)$-forms on $X$ and let $\pmb{\psi}\in\PSH(X,\pmb{\omega})$ 
be of small unbounded locus. Then, all the properties of Lemma~\ref{lem:tech} can be extended to functions of finite mixed relative energy. In
particular, if $\pmb{\varphi}\in\caE^{1}(X,\pmb{\omega},\pmb{\psi})$, then $I_{\pmb{\psi}}(\pmb{\varphi})$ is symmetric with respect to the simultaneous
permutation of $\varphi_{0},\ldots,\varphi_{d}$ and $\psi_{0},\ldots,\psi_{d}$ as well as of $\omega_{0},\ldots,\omega_{d}$. Moreover,  if $(\pmb{\varphi_
{n}})_{\pmb{n}\in\N^{d+1}}$ is a sequence of $(d+1)$-tuples of functions in $\caE^{1}(X,\pmb{\omega},\pmb{\psi})$, converging componentwise
from above to $\pmb{\varphi}$, then, the net of mixed relative energies $(I_{\pmb{\psi}}(\pmb{\varphi_{n}}))_{\pmb{n}\in\N^{d+1}}$ converges from 
above to the mixed relative energy $I_{\pmb{\psi}}(\pmb{\varphi})$, i.\,e., we have
\begin{displaymath}
\lim_{\pmb{n}\to\pmb{\infty}}I_{\pmb{\psi}}(\pmb{\varphi_{n}})=I_{\pmb{\psi}}(\pmb{\varphi}).
\end{displaymath}
\end{proposition}
\begin{proof}
The proof is based on~\cite[Proposition~2.10]{BEGZ}. We start by
proving that
\begin{equation}
\label{eq:17}
I_{\pmb{\psi}}(\pmb{\varphi})=\lim_{C\to\infty}I_{\pmb{\psi}}(\pmb{\varphi}^{(C)}),
\end{equation}
where $\pmb{\varphi}^{(C)}=(\varphi_{0}^{(C)},\ldots,\varphi_{d}^{(C)})$. For this, let $u_{j}\in\caE(X,\omega_{j},\psi_{j})$ satisfying $[u_{j}]=[\psi_{j}]$ 
and $\varphi_{j}\le u_{j}$ for $j=0,\ldots,d$. Introducing the functions $\widetilde{u}_{j,C}=\max(\varphi_{j}^{(C)},u_{j})$, we observe that $\widetilde{u}_
{j,C}$ converge from above to $u_{j}$ as $C\rightarrow\infty$. Then, by Lemma~\ref{lem:tech}, parts~\ref{item:23} and~\ref{item:24}, we find that
\begin{displaymath}
I_{\pmb{\psi}}(\pmb{u})=\lim_{C\to\infty}I_{\pmb{\psi}}(\widetilde{\pmb{u}}_{C})\ge\lim_{C\to\infty}I_{\pmb{\psi}}(\pmb{\varphi}^{(C)}),
\end{displaymath}
where $\widetilde{\pmb{u}}_{C}=(\widetilde{u}_{0,C},\ldots,\widetilde{u}_{d,C})$. Applying the definition of the mixed relative energy as an infimum 
and using Lemma~\ref{lem:tech}~\ref{item:24} once again, we deduce
\begin{displaymath}
I_{\pmb{\psi}}(\pmb{\varphi})\ge\lim_{C\to\infty}I_{\pmb{\psi}}(\pmb{\varphi}^{(C)})\ge I_{\pmb{\psi}}(\pmb{\varphi}),
\end{displaymath}
from which equation~\eqref{eq:17} follows. This convergence result allows us to easily transfer the properties of the mixed relative energy $I_{\pmb
{\psi}}(\pmb{\varphi})$ from functions $\pmb{\varphi}\in\caE(X,\pmb{\omega},\pmb{\psi})$ with $[\pmb{\varphi}]=[\pmb{\psi}]$ to arbitrary functions 
$\pmb{\varphi}\in\caE^{1}(X,\pmb{\omega},\pmb{\psi})$. For instance, we show the continuity of the mixed relative energy with respect to monotone 
sequences. Let $(\pmb{\varphi_{n}})_{\pmb{n}\in\N^{d+1}}$ be a sequence of $(d+1)$-tuples of functions in $\caE^{1}(X,\pmb{\omega},\pmb{\psi})$, 
converging componentwise from above to $\pmb{\varphi}$. Then, we have by Lemma \ref{lem:tech}~\ref{item:23}
\begin{displaymath} 
\lim_{\pmb{n}\to\pmb{\infty}}I_{\pmb{\psi}}(\pmb{\varphi}^{(C)}_{\pmb{n}})=I_{\pmb{\psi}}(\pmb{\varphi}^{(C)}),  
\end{displaymath}
while the previous discussion gives
\begin{displaymath}
\lim_{C\to\infty}I_{\pmb{\psi}}(\pmb{\varphi}^{(C)}_{\pmb{n}})=I_{\pmb{\psi}}(\pmb{\varphi}_{\pmb{n}})\qquad\text{and}\qquad\lim_{C\to\infty}I_
{\pmb{\psi}}(\pmb{\varphi}^{(C)})=I_{\pmb{\psi}}(\pmb{\varphi}).
\end{displaymath}
By monotone convergence, we deduce
\begin{displaymath}
\lim_{\pmb{n}\to\pmb{\infty}}I_{\pmb{\psi}}(\pmb{\varphi}_{\pmb{n}})=I_{\pmb{\psi}}(\pmb{\varphi}).
\end{displaymath}
\end{proof}

\begin{lemma}
\label{lemm:2}
Let $\pmb{\omega}=(\omega_{0},\ldots,\omega_{d})$ be smooth closed positive $(1,1)$-forms on $X$, let $\pmb{\psi}\in\PSH(X,\pmb{\omega})$ 
be of small unbounded locus, and let $\pmb{\varphi}\in\caE^{1}(X,\pmb{\omega},\pmb{\psi})$. Then, for $j=0,\ldots,d$, the integral
\begin{displaymath}
\int_{X}(\varphi_{j}-\psi_{j})\big\langle(\ddc\varphi_{0}+\omega_{0})\wedge\ldots\wedge(\ddc\varphi_{j-1}+\omega_{j-1})\wedge(\ddc\psi_{j+1}+
\omega_{j+1})\wedge\ldots\wedge(\ddc \psi_{d}+\omega_{d})\big\rangle
\end{displaymath}
is finite.
\end{lemma}
\begin{proof}
The proof is based on the proof of~\cite[Proposition~2.11~(ii)]{BEGZ}. Since each function $\varphi_{j}$ is more singular than the corresponding
$\psi_{j}$, i.\,e., $\varphi_{j}\le\psi_{j}+O(1)$, we can assume, after adding a constant to each $\varphi_{j}$, that $\varphi_{j}\le\psi_{j}$. Then,
for any $k\in\N$, we consider the sets
\begin{displaymath}
U_{k}=\{x\in X\,\vert\,\varphi_{j}>\psi_{j}-k,\text{ for }j=0,\ldots,d\},
\end{displaymath}
and the canonical approximants $\varphi_{j}^{(k)}=\max(\varphi_{j},\psi_{j}-k)$ giving rise to the measure
\begin{displaymath}
\widetilde{\Theta}_{k}=\big\langle(\ddc\varphi^{(k)}_{0}+\omega_{0})\wedge\ldots\wedge(\ddc\varphi^{(k)}_{j-1}+\omega_{j-1})\wedge(\ddc\psi_
{j+1}+\omega_{j+1})\wedge\ldots\wedge(\ddc\psi_{d}+\omega_{d})\big\rangle
\end{displaymath}
on $X$. By the very construction, it turns out that the measure
\begin{displaymath}
\Theta=\big\langle(\ddc\varphi_{0}+\omega_{0})\wedge\ldots\wedge(\ddc\varphi_{j-1}+\omega_{j-1})\wedge(\ddc\psi_{j+1}+\omega_{j+1})
\wedge\ldots\wedge(\ddc\psi_{d}+\omega_{d})\big\rangle
\end{displaymath}
is the increasing limit of the measures $\Theta_{k}=\mathbf{1}_{U_{k}}\,\widetilde{\Theta}_{k}$, where $\mathbf{1}_{U_{k}}$ denotes the 
characteristic function for $U_{k}$, in other words
\begin{displaymath}
\Theta=\lim_{k\to\infty}\Theta_{k}.
\end{displaymath}
Furthermore, we note that the measure $\widetilde{\Theta}_{k}-\Theta_{k}$ is a positive measure on $X$. For each $\ell\in\N$, the functions 
$\varphi_{j}^{(\ell)}-\psi_{j}$ are bounded, measurable, and non-positive on $X$. Therefore, we arrive at the estimates
\begin{align*}
\int_{X}(\varphi^{(\ell)}_{j}-\psi_{j})\Theta&=\lim_{k\to\infty}\int_{X}(\varphi^{(\ell)}_{j}-\psi_{j})\Theta_{k}\ge\lim_{k\to\infty}\int_{X}(\varphi^{(k)}_
{j}-\psi_{j})\Theta_{k} \\[1mm]
&\ge\lim_{k\to\infty}\int_{X}(\varphi^{(k)}_{j}-\psi_{j})\widetilde{\Theta}_{k}\ge\lim_{k\to\infty}I_{\pmb{\psi}}(\pmb{\varphi}^{(k)})=I_{\pmb{\psi}}
(\pmb{\varphi}).
\end{align*}
Here, we note that we have used in the third inequality above that the mixed relative energy $I_{\pmb{\psi}}(\pmb{\varphi}^{(k)})$ is a sum of 
non-positive terms, one of them being the one in the preceding inequality. By monotone convergence, we finally deduce that
\begin{displaymath}
\int_{X}(\varphi_{j}-\psi_{j})\Theta\ge I_{\pmb{\psi}}(\pmb{\varphi})>-\infty,
\end{displaymath}
which proves the claim of the lemma.
\end{proof}

\begin{theorem}
\label{thm:8}
Let $\pmb{\omega}=(\omega_{0},\ldots,\omega_{d})$ be smooth closed positive $(1,1)$-forms on $X$, let $\pmb{\psi}\in\PSH(X,\pmb{\omega})$ 
be of small unbounded locus, and let $\pmb{\varphi}\in\caE^{1}(X,\pmb{\omega},\pmb{\psi})$. Then, formula~\eqref{keyfor} holds, that is
\begin{align*}
I_{\pmb{\psi}}(\pmb{\varphi})=\sum_{j=0}^{d}\,\int_{X}(\varphi_{j}-\psi_{j})\big\langle&(\ddc\varphi_{0}+\omega_{0})\wedge\ldots\wedge(\ddc\varphi_
{j-1}+\omega_{j-1})\wedge \\
&\wedge(\ddc\psi_{j+1}+\omega_{j+1})\wedge\ldots\wedge(\ddc\psi_{d}+\omega_{d})\big\rangle.
\end{align*}
\end{theorem}
\begin{proof}
The proof of this result is based on the proof of~\cite[Proposition~2.5]{DNK:L1metricgeometry}. As in the proof of Lemma~\ref{lemm:2}, we can 
assume that $\varphi_{j}\le\psi_{j}$ for $j=0,\dots,d$. We note that, in the notation of Lemma~\ref{lemm:2}, it is enough to prove that
\begin{multline}
\label{eq:25}
\lim_{k\to\infty}\int_{X}(\varphi^{(k)}_{0}-\psi_{0})\big\langle(\ddc\varphi^{(k)}_{1}+\omega_{1})\wedge\ldots\wedge(\ddc\varphi^{(k)}_{d}+\omega_
{d})\big\rangle \\
=\int_{X}(\varphi_{0}-\psi_{0})\big\langle(\ddc\varphi_{1}+\omega_{1})\wedge\ldots\wedge(\ddc\varphi_{d}+\omega_{d})\big\rangle,
\end{multline}
as the mixed relative energy of the canonical approximants $I_{\pmb{\psi}}(\pmb{\varphi}^{(k)})$ converges to the mixed relative energy $I_{\pmb
{\psi}}(\pmb{\varphi})$, and the mixed relative energy is a sum of integrals of the above shape, after reordering and replacing some of the $\varphi_
{j}$'s by the respective $\psi_{j}$'s. For each $k\in\N$ and each subset $I\subseteq\{0,\ldots,d\}$, we now write
\begin{align*}
V_{I,k}=\left\{x\in X\,\middle\vert\,
\begin{aligned}
\varphi_{j}&>\psi_{j}-k,\quad\text{if }j\not\in I \\
\varphi_{j}&\le\psi_{j}-k,\quad\text{if }j\in I
\end{aligned}
\right\}.
\end{align*}
In particular, we note that the set $V_{\emptyset,k}$ is the open set $U_{k}$ introduced in the proof of Lemma~\ref{lemm:2}. Since we trivially
have that
\begin{align*}
\lim_{k\to\infty}\int_{V_{\emptyset,k}}(\varphi^{(k)}_{0}-\psi_{0})\big\langle&(\ddc\varphi^{(k)}_{1}+\omega_{1})\wedge\ldots\wedge(\ddc\varphi^
{(k)}_{d}+\omega_{d})\big\rangle \\[-2mm]
&=\lim_{k\to\infty}\int_{V_{\emptyset,k}}(\varphi_{0}-\psi_{0})\big\langle(\ddc\varphi_{1}+\omega_{1})\wedge\ldots\wedge(\ddc\varphi_{d}+
\omega_{d})\big\rangle \\ 
&=\int_{X}(\varphi_{0}-\psi_{0})\big\langle(\ddc\varphi_{1}+\omega_{1})\wedge\ldots\wedge(\ddc\varphi_{d}+\omega_{d})\big\rangle,
\end{align*}
we are reduced to show that, for $I\neq\emptyset$, we have
\begin{equation}
\label{eq:18}
\lim_{k\to\infty}\int_{V_{I,k}}(\varphi^{(k)}_{0}-\psi_{0})\big\langle(\ddc\varphi^{(k)}_{1}+\omega_{1})\wedge\ldots\wedge(\ddc\varphi^{(k)}_{d}+
\omega_{d})\big\rangle=0.
\end{equation}
To do so, we note that on $V_{I,k}$, the function $\varphi_{j}^{(k)}$ is equal to $\varphi_{j}$ if $j\not\in I$, and agrees with $\psi_{j}-k$ if $j\in I$. 
Since $\ddc(\psi_{j}-k)=\ddc\psi_{j}$ and $0\ge\varphi_{0}^{(k)}-\psi_{0}\ge\varphi_{0}-\psi_{0}$, we can assume that the integrand is independent 
of $k$ and thus corresponds to a non-positive function multiplied by a measure that does not charge any pluripolar set of the shape considered 
in Lemma~\ref{lemm:2}. Thus, the integral over the whole of $X$ is finite. Since the sets $V_{I,k}$ for $I\neq\emptyset$ converge to a pluripolar
set, we deduce that the limit~\eqref{eq:18} is zero, completing the proof of the theorem.  
\end{proof}

The following criterion together with Lemma~\ref{lemm:10} will be useful to check that a $(d+1)$-tuple of functions has finite mixed relative 
energy. To state the result we use the following notation: If $I\subseteq\{0,\ldots,d\}$ and $\pmb{\varphi}=(\varphi_{0},\ldots,\varphi_{d})$, we 
write $\pmb{\varphi}_{I}=\sum_{j\in I}\varphi_{j}$ and similarly $\pmb{\omega}_{I}=\sum_{j\in I}\omega_{j}$ as well as $\pmb{\psi}_{I}=\sum_
{j\in I}\psi_{j}$. 

\begin{proposition}
\label{prop:8} 
Let $\pmb{\omega}=(\omega_{0},\ldots,\omega_{d})$ be smooth closed positive $(1,1)$-forms on $X$, let $\pmb{\psi}\in\PSH(X,\pmb{\omega})$ 
be of small unbounded locus, and let $\pmb{\varphi}\in\caE(X,\pmb{\omega},\pmb{\psi})$. If for all subsets $I\subseteq\{0,\ldots,d\}$ the condition 
$\pmb{\varphi}_{I}\in\caE^{1}(X,\pmb{\omega}_{I},\pmb{\psi}_{I})$ holds, then $\pmb{\varphi}\in\caE^{1}(X,\pmb{\omega},\pmb{\psi })$ and the 
usual polarization formula
\begin{equation}
\label{eq:29}
(d+1)!\,I_{\pmb{\psi}}(\pmb{\varphi})=\sum_{I\subseteq\{0,\ldots,d\}}(-1)^{d+1-\vert I\vert}I_{\pmb{\psi}_{I}}(\pmb{\varphi}_{I})
\end{equation}
is satisfied.
\end{proposition}
\begin{proof}
By the multilinearity of the non-pluripolar product and the definition of the mixed relative energy, formula~\eqref{eq:29} is satisfied when we 
have minimal relative singularities, that is, when $[\varphi_{j}]=[\psi_{j}]$ for $j=0,\ldots,d$. In particular, we find
\begin{displaymath}
(d+1)!\,I_{\pmb{\psi}}(\pmb{\varphi}^{(C)})=\sum_{I\subseteq\{0,\ldots,d\}}(-1)^{d+1-\vert I\vert}I_{\pmb{\psi}_{I}}((\pmb{\varphi}^{(C)})_{I}).
\end{displaymath}
If $\pmb{\varphi}_{I}\in\caE^{1}(X,\pmb{\omega}_{I},\pmb{\psi}_{I})$ for all $I\subseteq\{0,\dots,d\}$, then the right-hand side has a finite limit 
as $C\rightarrow\infty$. Therefore, the left-hand side also has a finite limit, so $\pmb{\varphi}\in\caE^{1}(X,\pmb{\omega},\pmb{\psi })$ and 
equation~\eqref{eq:29} is satisfied. 
\end{proof}

\subsection{Algebraic and almost algebraic singularities}
\label{sec:agebr-almost-algebr}

In this subsection, we fix a smooth closed positive $(1,1)$-form
$\omega_{0}$.

\begin{definition}
\label{def:17}
An $\omega_{0}$-plurisubharmonic function $\varphi$ on $X$ is said \emph{to have algebraic singularities}, if the function $\varphi$ can locally be 
written as  
\begin{equation}
\label{eq:9}
\varphi=\frac{c}{2}\log\big(\vert f_{1}\vert^2+\ldots+\vert f_{N}\vert^2\big)+\lambda,
\end{equation}
where $c\in\Q_{\ge 0}$ is a constant, $f_{1},\ldots,f_{N}$ are non-zero algebraic rational functions, and $\lambda$ is a bounded function.
\end{definition}

\begin{definition}
\label{def:19}
An $\omega_{0}$-plurisubharmonic function $\varphi$ on $X$ is said \emph{to have almost asymptotically algebraic singularities}, if there is a 
sequence of $\omega_{0}$-plurisubharmonic functions $(\psi_{k})_{k\ge 1}$, an $\omega_{0}$-plurisubharmonic function $f$, having algebraic 
singularities, and a constant $C>0$ such that the inequalities
\begin{displaymath}
\psi_{k}+\frac{1}{k}f-C\le\varphi\le\psi_{k}+C
\end{displaymath}
hold.
\end{definition}

It is easy to see that, if a function satisfies this definition, then it also satisfies~\cite[Definition~3.2]{botero21:_chern_weil_hilber_samuel}.

\section{Review of the theory of adelic arithmetic line bundles} 
\label{third-section}

In this section we summarize some concepts and results from~\cite{YuanZhang:adelic}, although the presentation is slightly different. For 
instance, instead of working with mixed $(\Q,\Z)$-divisors, we will work with divisors with real coefficients.

\subsection{The geometric case, adelic divisors and b-divisors}
\label{sec:geom-case-adel}

Let $X$ be either a projective normal variety over a field $k$ or a projective flat normal scheme, separated and of finite type over a Dedekind 
domain, $B$ an effective Cartier divisor on $X$ (the ``boundary'' divisor) with support $\vert B\vert$,  and $U\coloneqq X\setminus\vert B\vert$. 
We denote by $R(X,U)$ the category of all proper normal modifications of $X$ which are isomorphisms over $U$. Given $\pi\in R(X,U)$, i.\,e., 
$\pi\colon X_{\pi}\rightarrow X$, we denote by $\Div(X_{\pi })$ the group of Cartier divisors on $X_{\pi}$ and we write $\DivR(X_{\pi})\coloneqq
\Div(X_{\pi})\otimes_{\Z}\R$.

\begin{definition} 
\label{def:7} 
The space of \emph{model divisors on $U$} is defined as the limit
\begin{displaymath}
\DivR(U)^{\model}\coloneqq\varinjlim_{\pi\in R(X,U)}\Div_{\R}(X_{\pi}).
\end{displaymath}
\end{definition}

On $\DivR(U)^{\model}$ there is a $B$-adic norm (possibly taking the value $\infty$) defined by
\begin{equation}
\label{eq:4}
\Vert D\Vert\coloneqq\inf\{\varepsilon\in\R_{>0}\,\vert\,-\varepsilon B\le D\le\varepsilon B\};
\end{equation}
here $D\in\DivR(U)^{\model}$, which means that there is a $\pi\in R(X,U)$ such that $D\in\DivR(X_{\pi})$. The condition $D\le\varepsilon B$, 
means that the divisor $\varepsilon\,\pi^{\ast}B-D$ is effective on $X_{\pi}$. Finally, the infimum of an empty set is assumed to be 
$\infty$.

\begin{definition}
\label{def:8}
The space $\DivR(U)^{\adel}$ of \emph{adelic divisors on $U$} is the completion of $\DivR(U)^{\model}$ with respect to the $B$-adic topology 
induced by the  $B$-adic norm defined above.
\end{definition}

In other words, an element $D\in\DivR(U)^{\adel}$ is represented by a sequence of Cartier divisors $(D_{n})_{n\in\N}$ that is Cauchy with
respect to the $B$-adic norm. Two Cauchy sequences $(D_{n})_{n\in\N}$ and $(D_{n}')_{n\in\N}$ are equivalent if the sequence $(D_{n}'')_
{n\in\N}$ with $D_{2n}''=D_{n}$ and $D_{2n+1}''=D_{n}'$ is still Cauchy.

\begin{remark}
\label{rem:11}
If $B'$ is another boundary divisor with $\vert B\vert=\vert B'\vert$, then there is a positive constant $c>0$ such that $(1/c)B\le B' \le cB$. 
Therefore, the $B$-adic norm and the $B'$-adic norm are equivalent. Thus, the definition of $\DivR(U)^{\adel}$ only depends on the open 
set $U$ and not on a particular choice of boundary divisor $B$.
\end{remark}

\begin{definition}
\label{def:9}
The cone $\DivR(U)^{\nef}$ of \emph{nef adelic divisors on $U$} is the closure of the cone of nef model divisors on $U$ inside of $\DivR
(U)^{\adel}$. The space of \emph{integrable adelic divisors on $U$} is the vector space
\begin{displaymath}
\DivR(U)^{\inte}\coloneqq\DivR(U)^{\nef}-\DivR(U)^{\nef}.
\end{displaymath}
\end{definition}

\begin{remark}
\label{rem:8}
If $V\subseteq U\subseteq X$ is an inclusion of open subsets, then there are injective maps
\begin{displaymath}
\DivR(U)^{\adel}\longrightarrow\DivR(V)^{\adel},\quad\DivR(U)^{\nef}\longrightarrow\DivR(V)^{\nef}.
\end{displaymath}
The notion of nef adelic divisor we give here corresponds more precisely to the notion of strongly nef adelic divisor in~\cite{YuanZhang:adelic}. 
Nevertheless, if one is willing to go to a smaller open subset, there is no essential difference between nef and strongly nef adelic divisors 
in~\cite{YuanZhang:adelic}. Namely, if $D$ is a nef adelic divisor in an open subset $U$, then there is a subset $V\subseteq U$ such that 
the image of $D$ in $\DivR(V)^{\adel}$ is strongly nef in the sense of~\cite{YuanZhang:adelic}.  
\end{remark}

From now on we assume that we are in the case when $k$ is a field. The following results are proven in~\cite{YuanZhang:adelic}.

\begin{proposition}
\label{prop:2}
If $\pi\in R(X,U)$ and $B'$ is an effective divisor on $X_{\pi}$ with $\vert B'\vert=X_{\pi}\setminus U$, then the $B'$-adic norm is equivalent to 
the $B$-adic norm. Therefore, the definitions of the spaces $\DivR(U)^{\adel}$, $\DivR(U)^{\nef}$, and $\DivR(U)^{\inte}$ only depend on $U$ 
and not on a particular choice of compactification $X$ and divisor $B$.  
\end{proposition}

\begin{theorem}
\label{thm:3} 
The intersection pairing of nef Cartier divisors extends by continuity to a unique symmetric multilinear pairing
\begin{displaymath}
\underbrace{\DivR(U)^{\nef}\otimes\ldots\otimes\DivR(U)^{\nef}}_{\text{$\dim(U)$-\emph{times}}}\longrightarrow\R. 
\end{displaymath}
This pairing can be extended by linearity to a pairing 
\begin{displaymath}
\underbrace{\DivR(U)^{\inte}\otimes\ldots\otimes\DivR(U)^{\inte}}_{\text{$\dim(U)$-\emph{times}}}\longrightarrow\R.
\end{displaymath}
\end{theorem}

Theorem~\ref{thm:3} means that, if $D_{j}\in\DivR(U)^{\nef}$ with $j=1,\ldots,d=\dim(U)$ are nef adelic divisors on $U$ represented by the 
Cauchy sequences $(D_{j,n})_{n\in\N}$ of nef model divisors on $U$, then the limit
\begin{equation}
\label{eq:3}
D_{1}\cdot\ldots\cdot D_{d}\coloneqq\lim_{(n_{1},\ldots,n_{d})\to(\infty,\ldots,\infty)}D_{1,n_{1}}\cdot\ldots\cdot D_{d,n_{d}}
\end{equation}
exists and is independent of the choice of approximating sequences.

\begin{remark}
\label{rem:3} 
In~\cite{YuanZhang:adelic} it is only stated that the diagonal limit
\begin{displaymath}
\lim_{n\to\infty}D_{1,n}\cdot\ldots\cdot D_{d,n}
\end{displaymath}
exists and is independent of the chosen Cauchy sequences. But their proof actually shows that the multi-limit exists. In particular, we can 
compute it using iterated limits. 
\end{remark}

We next want to relate the notions of adelic divisors and of b-divisors. Recall that the spaces of Cartier b-divisors and Weil b-divisors on 
$X$ are defined as
\begin{displaymath}
\CbDiv(X)=\varinjlim_{\pi\in R(X)}\DivR(X_{\pi})\qquad\text{and}\qquad\WbDiv(X)=\varprojlim_{\pi\in R(X)}\DivR(X_{\pi}),
\end{displaymath}
respectively, where the limits are taken with respect to all possible proper modifications of $X$, i.\,e., $R(X)=\{\pi\colon X_{\pi}\rightarrow 
X\,\vert\,\text{$\pi$ is a proper modification}\}$. Both limits are defined in the category of topological vector spaces and thus they have a 
topology.

There is a map $b\colon\DivR(U)^{\model}\rightarrow\CbDiv(X)$, since the directed set used to define the first one is a subset of the one 
used to define the second one. Since there is a canonical inclusion $\CbDiv(X)\rightarrow\WbDiv(X)$, we obtain a map, also denoted by 
$b$, from $\DivR(U)^{\model}$ to $\WbDiv(X)$.

\begin{lemma}
\label{lemm:8}
Let $(D_{n})_{n\in \N}$ be a Cauchy sequence in $\DivR(U)^{\model}$ with respect to the $B$-adic topology. Then, the sequence $(b(D_
{n}))_{n\in\N}$ is convergent in $\WbDiv(X)$.   
\end{lemma}
\begin{proof}
We start by recalling the topology of $\WbDiv(X)$. Since it is defined as an inverse limit, an element of $\WbDiv(X)$ is a collection $(D_
{\pi})_{\pi\in R(X)}$ of divisors, one on each proper modification $X_{\pi}$ of $X$, that are compatible by push-forward. A sequence $((D_
{n,\pi})_{\pi\in R(X)})_{n\in\N}$ in $\WbDiv(X)$ converges to $(D_{\pi })_{\pi\in R(X)}$, if and only if the sequence $(D_{n,\pi})_{n\in\N}$ 
converges to $D_{\pi}$ for all $\pi\in R(X)$. This means that there is a divisor $E_{\pi}$ on $X_{\pi}$ such that $\vert D_{n,\pi}\vert\subseteq
\vert E_{\pi}\vert$, that is, the supports of all the divisors $D_{n,\pi}$ are contained in a common divisor for all $n\in\N$, and the divisors 
$D_{n,\pi}$ converge to $D_{\pi}$ componentwise.

Now assume that $(D_{n})_{n\in\N}$ is a Cauchy sequence of model divisors on $U$. This means that, for all $\varepsilon>0$, there is 
$n_{\varepsilon}\in\N$ such that the inequality
\begin{equation}
\label{eq:28}
-\varepsilon B \le D_{n}-D_{m}\le\varepsilon B
\end{equation}
holds for all $n,m\ge n_{\varepsilon}$. Here, $D_{n}$ lives on some modification $X_{n}$, $D_{m}$ on some modification $X_{m}$, $B$
is the boundary divisor on $X$ defining $U$, and the inequality is to be understood on a suitable modification of $X$ that dominates both. 
For each modification $\pi\colon X_{\pi}\rightarrow X$, we define $B_{\pi}=\pi^{\ast}B$, and we choose a proper modifications $X_{\pi'}$ 
that dominates at the same time $X_{\pi}$, $X_{n}$, and $X_{m}$. Then, we define $D_{n,\pi}$ and $D_{m,\pi}$ by first pulling back $D_{n}$ 
and $D_{m}$ to $X_{\pi'}$ and then pushing forward the resulting divisors to $X_{\pi}$, respectively. The obtained divisors do not depend on 
the choice of $\pi'$. Condition~\eqref{eq:28} implies that, for all $n\ge n_{\varepsilon}$, the support of $D_{n,\pi}$ is contained in $\vert D_
{n_{\varepsilon}}\vert\cup\vert B\vert$ and that the multiplicities of $(D_{n,\pi})_{n\in\N}$ form a Cauchy sequence. Therefore, the sequence 
$(D_{n,\pi})_{n\in\N}$ converges to a divisor $D_{\pi}$ in $X_{\pi}$. It is clear that the various $D_{\pi}$'s obtained in this way are compatible 
with respect to push-forward and hence define an element of $\WbDiv(X)$. The sequence $(b(D_{n}))_{n\in\N}$ by construction converges 
to this element.  
\end{proof}

Lemma~\ref{lemm:8} implies that the map $b$ can be extended to a continuous map
\begin{displaymath}
b\colon\DivR(U)^{\adel}\longrightarrow\WbDiv(X). 
\end{displaymath}

In order to link adelic nef divisors with approximable nef b-divisors in the sense of~\cite[Definition~4.8]{botero21:_chern_weil_hilber_samuel}, 
it is convenient to have a monotonous approximation to a nef adelic divisor. This can always be achieved as soon as the boundary divisor is 
nef.

\begin{lemma}
\label{lemm:9}
Let $X$ be a complex projective manifold with boundary divisor $B$ and $U=X\setminus\vert B\vert$ such that $B$ is a nef divisor. Let 
$D$ be a nef adelic divisor on $U$. Then, there is a sequence $(D_{n})_{n\in\N}$ of nef model divisors on $U$ converging monotonically 
decreasingly to $D$ in the $B$-adic topology. That is, for all $n\in\N$, we have $D_{n+1}\le D_{n}$.
\end{lemma}
\begin{proof}
Since $D$ is a nef adelic divisor on $U$, there is a sequence of nef model divisors $(E_{n})_{n\in\N}$ on $U$ converging to $D$ in the 
$B$-adic topology. After extracting a subsequence, we can assume that
\begin{displaymath}
D-\frac{1}{2^{n}}\,B\le E_{n}\le D+\frac{1}{2^{n}}\,B.
\end{displaymath}
Writing $D_{n}=E_{n}+\frac{4}{2^{n}}\,B$, we find
\begin{displaymath}
D_{n+1}=E_{n+1}+\frac{4}{2^{n+1}}\,B\le D+\frac{5}{2^{n+1}}\,B\le E_{n}+\frac{7}{2^{n+1}}\,B\le E_{n}+\frac{4}{2^{n}}B=D_{n}. 
\end{displaymath}
This complets the proof of the lemma.
\end{proof}

\begin{corollary}
\label{cor:6}
If $D$ is a nef adelic divisor on $U$, then $b(D)$ is an approximable nef b-divisor in the sense of~\cite[Definition~4.8]{botero21:_chern_weil_hilber_samuel}. 
Moreover, the intersection product of nef adelic divisors agrees through the map $b$ with the intersection product of nef b-divisors introduced 
in~\cite{DangFavre:inter_b_div}. 
\end{corollary}
\begin{proof}
If $D$ is a nef adelic divisor, after shrinking $U$ if neccesary, we can assume that the boundary divisor $B$ is nef. Therefore, by Lemma
\ref{lemm:9}, it can be aproximated monotonically by nef Cartier b-divisors, so $b(D)$ is approximable nef. The intersection products of~\cite
{YuanZhang:adelic} and of~\cite{DangFavre:inter_b_div} now agree because both are continuous extensions of the intersection product of
nef Cartier b-divisors.  
\end{proof}

\subsection{The local arithmetic case over an archimedean place, Green functions} 
\label{sec:local-arithm-case}

In this subsection, let $X$ be a complex projective manifold of dimension $d$. Let $B$ be an effective Cartier divisor on $X$ with support 
$\vert B\vert$ and let $U\coloneqq X\setminus\vert B\vert$. Let $g_{B}$ be a Green function of continuous type for $B$ such that there is an 
$\eta>0$ with $g_{B}(x)>\eta$ for all $x\in U$. Such Green functions exist because $B$ being effective, any Green function of continuous 
type for $B$ is bounded below, and we can add to it a big enough constant to make it bigger than $\eta$. We call $\overline{B}=(B,g_{B})$ an
\emph{arithmetic boundary divisor for $X$}.

We now repeat the steps of Subsection~\ref{sec:geom-case-adel} for the local arithmetic case. For this, let $R(X,U)$ be as in the previous 
subsection.

\begin{definition}
\label{def:10} 
An \emph{arithmetic divisor $\overline{D}$ on $X_{\pi}$ of smooth type} is a pair $(D,g_{D})$, where $D\in\DivR(X_{\pi})$ is a Cartier 
divisor on $X_{\pi}$ and $g_{D}$ is a Green function for $D$ of smooth type; here $\pi\in R(X,U)$. The vector space of arithmetic 
divisors on $X_{\pi}$ of smooth type is denoted by $\DivhR(X_{\pi})$.

The space of \emph{model  arithmetic divisors on $U$ of smooth type} is defined as the limit
\begin{displaymath}
\DivhR(U)^{\model}\coloneqq\varinjlim_{\pi\in R(X,U)}\DivhR(X_{\pi}).
\end{displaymath}
Similarly we can define \emph{model arithmetic divisors on $U$ of continuous type} by asking that the Green functions are of continuous 
type. To shorten the notation, by a model arithmetic divisor on $U$, we will mean a model arithmetic divisor on $U$ of smooth type, and 
when we consider a model arithmetic divisor on $U$ of continuous type, we will say so explicitely. 
\end{definition}

On $\DivhR(U)^{\model}$ there is a $\overline{B}$-adic norm (possibly again taking the value $\infty$) defined by
\begin{displaymath}
\Vert\overline{D}\Vert=\Vert(D,g_{D})\Vert\coloneqq\inf\{\varepsilon\in\R_{>0}\,\vert\,-\varepsilon B\le D\le\varepsilon B,\,-\varepsilon
g_{B}\le g_{D}\le\varepsilon g_{B}\};
\end{displaymath}
here $D\in\DivR(U)^{\model}$, which means that there is a $\pi\in R(X,U)$ such that $D\in\DivR(X_{\pi})$, and $g_{D}$ is a Green
function for $D$ of smooth type.

\begin{definition}
\label{def:11} 
The space $\DivhR(U)^{\adel}$ of \emph{adelic arithmetic divisors on $U$} is the completion of $\DivhR(U)^{\model}$ with respect to 
the $\overline{B}$-adic topology induced by the $\overline{B}$-adic norm defined above. As before, elements of $\DivhR(U)^{\adel}$ 
are represented by Cauchy sequences of model arithmetic divisors on $U$.   
\end{definition}

\begin{remark} 
\label{rem:9}
Since the set of smooth functions is dense inside of the set of continuous functions with respect to the topology of uniform convergence 
on compacta, the space $\DivhR(U)^{\adel}$ can also be seen as the completion of the space of model arithmetic divisors on $U$ of 
continuous type with respect to the $\overline{B}$-adic topology. 
\end{remark}

\begin{definition}
\label{def:16}
A model arithmetic divisor $\overline{D}=(D,g_{D})\in\DivhR(U)^{\model}$ is called \emph{nef}, if $D$ is a nef divisor on $X_{\pi}$ for 
some $\pi\in R(X,U)$ and $g_{D}$ is a Green function for $D$ that is at the same time of smooth type and of plurisubharmonic type.

The cone $\DivhR(U)^{\nef}$ of \emph{nef adelic arithmetic divisors on $U$} is the closure of the cone of nef model arithmetic divisors 
on $U$ inside of $\DivhR(U)^{\adel}$. That is, the divisors that can be represented by a Cauchy sequence of nef model arithmetic 
divisors on $U$. An adelic arithmetic divisor on $U$ is called \emph{integrable}, if it can be written as the difference of two nef adelic 
arithmetic divisors on $U$. The space of integrable adelic arithmetic divisors on $U$ is denoted by $\DivhR(U)^{\inte}$.
\end{definition}

\begin{remark}
The discussion of Remark~\ref{rem:8} is also valid in this case. 
\end{remark}

\begin{remark}
\label{rem:6}
By the global Richberg's regularization of continuous plurisubharmonic functions (see~\cite[Lemma~2.15]{dem-reg}), an adelic arithmetic
divisor on $U$ is nef if and only if it can be written as the limit of a sequence of model arithmetic divisors on $U$ with Green functions 
that are of continuous and of plurisubharmonic type, and which is Cauchy with respect to the $\overline{B}$-adic norm.     
\end{remark}

The analogue of the following result in the non-archimedean case is contained in~\cite[Theorem~3.6.4]{YuanZhang:adelic}. The proof 
in the archimedean case is much simpler. 

\begin{proposition}
\label{prop:3}
Let $(D_{n},g_{D_{n}})_{n\in\N}$ be a Cauchy sequence of model arithmetic divisors on $U$ representing an element in $\DivhR(U)^
{\adel}$. Then, the sequence $(D_{n})_{n\in\N}$ converges to an element $D\in\DivR(U)^{\adel}$ and the sequence of functions 
$(g_{D_{n}})_{n\in\N}$ converges pointwise on $U$ to a function $g_{D}$. Moreover, the convergence is uniform on compact subsets 
$K\subseteq U$. In particular, there is a short exact sequence
\begin{displaymath}
0\longrightarrow C^{0}(U)_{0}\longrightarrow\DivhR(U)^{\adel}\longrightarrow\DivR(U)^{\adel}\longrightarrow 0,
\end{displaymath}
where $C^{0}(U)_{0}$ denotes the set of continuous functions $f$ on $U$ satisfying that the quotient $f/g_{B}$ can be extended to
a continuous function in the whole of $X$, taking the value zero in $\vert B\vert$. In other words, 
\begin{displaymath}
\vert f(x)\vert=o(g_{B}(x)),
\end{displaymath}
as $x\rightarrow\vert B\vert$.
\end{proposition}

As a consequence of Proposition~\ref{prop:3}, every element of $\DivhR(U)^{\adel}$ will subsequently be written as $(D,g_{D})$ with 
$D\in\DivR(U)^{\adel}$ and $g_{D}$ a function on $U$ giving rise to a function on $X$ up to a set of measure zero.

\begin{definition}
Let $D\in\DivR(U)^{\adel}$. A function $g_{D}$ on $U$ is called an \emph{admissible Green function for $D$}, if the pair $(D,g_{D})$ 
belongs to $\DivhR(U)^{\nef}$. In particular, if such functions exist, $D\in\DivR(U)^{\nef}$. 
\end{definition}

We next show that, after choosing a sufficiently large arithmetic reference divisor, admissible Green functions can be related to 
plurisubharmonic functions whose singularities are at most asymptotically algebraic.  

\begin{proposition} 
\label{prop:7}
Let $D\in\DivR(U)^{\adel}$ and let $g_{D}$ be an admissible Green function for $D$. Choose a sufficiently large arithmetic reference 
divisor $\overline{E}=(E,g_{E})$ on $X$ with $g_{E}$ a Green function of smooth type for $E$ such that $\overline{E}\ge\overline{B}$
and $\overline{E}\ge\overline{D}+2\overline{B}$. Letting $\omega_{0}=\omega_{E}(g_{E})$, the following 
statements hold:
\begin{enumerate}
\item 
\label{item:13}
The function $\varphi=g_{D}-g_{E}$ is $\omega_{0}$-plurisubharmonic. In particular, $g_{D}$ belongs to $L_{\loc}^{1}(X)$ and is a 
Green function for $E$ of plurisubharmonic type.
\item  
\label{item:14}
The function $\varphi$ has almost asymptotically algebraic singularities in the sense of Definition~\ref{def:19}.
\end{enumerate}
\end{proposition}
\begin{proof}
To prove~\ref{item:13}, let $(D_{n},g_{D_{n}})_{n\in\N}$ be a sequence of nef model arithmetic divisors on $U$ converging to $(D,g_
{D})$ in the $\overline{B}$-adic topology. Since $\overline{E}\ge\overline{D}+2\overline{B}$, there is an $n_{0}\in\N$ such that $E\ge 
D_{n}$ and $g_{E}\ge g_{D_{n}}$ for all $n\ge n_{0}$. By taking a suitable subsequence, we can assume that $E\ge D_{n}$ and $g_
{E}\ge g_{D_{n}}$ for all $n\in\N$. Then, by Lemma~\ref{lem:greeneff}, the functions $g_{D_{n}}$ are Green functions for $E$ of 
plurisubharmonic type for all $n\in\N$. Therefore, the functions $\varphi_{n}=g_{D_{n}}-g_{E}\le 0$ are $\omega_{0}$-plurisubharmonic 
for all $n\in\N$. Thus, the sequence $(\varphi_{n})_{n\in\N}$ is bounded above by $0$ and converges to the function $\varphi$. 
By~\cite[Theorem~3.2.12]{hoermander94:_notionconvexity} (see the comment before Theorem~4.1.8), the function $\varphi$ is 
$\omega_{0}$-plurisubharmonic. This proves~\ref{item:13}.

To prove~\ref{item:14}, choose a sequence $(D_{n},g_{D_{n}})_{n\in\N}$ of nef model arithmetic divisors on $U$ converging to 
$(D,g_{D})$ formed by nef divisors with rational coefficients and a strictly increasing sequence $(n_{k})_{k\ge 1}$ such that
\begin{displaymath}
-\frac{1}{k}g_{E}\le-\frac{1}{k}g_{B}\le g_{D}-g_{D_{n_{k}}}\le\frac{1}{k}g_{B}\le\frac{1}{k}g_{E}.
\end{displaymath}
The functions $\psi_{k}=g_{D_{n_{k}}}-g_{E}$ are now $\omega_{0}$-plurisubharmonic functions with algebraic singularities. Similarly, 
the function $f=-g_{E}\le 0$ is also a $\omega_{0}$-plurisubharmonic function with algebraic singularities. Since we have
\begin{displaymath}
\psi_{k}+\frac{1}{k}f\le\varphi\le\psi_{k}-\frac{1}{k}f,
\end{displaymath}
we deduce that $\varphi$ has almost asymptotically algebraic singularities. 
\end{proof}

\begin{remark}
Note that the inequalities $\overline{E}\ge\overline{B}$ and $\overline{E}\ge\overline{D}+2\overline{B}$ chosen for the arithmetic reference 
divisor $\overline{E}$ in Proposition~\ref{prop:7} could be relaxed as the above proof shows. However, for later purposes and the sake
of uniformity, we have decided to stick to these inequalities from now on.
\end{remark}

We next show that the adelic convergence has strong consequences with respect to the convergence of the associated non-pluripolar 
Monge--Amp\`ere measures.  

\begin{proposition}
\label{prop:6}
For $j=1,\ldots,d$, let $\overline{D}_{j}=(D_{j},g_{j})$ be nef adelic arithmetic divisors on $U$ and let $(D_{j,n},g_{j,n})_{n\in\N}$ be sequences
of nef model arithmetic divisors on $U$ converging in the $\overline{B}$-adic topology to  $\overline{D}_{j}$. Choose a sufficiently large 
arithmetic reference divisor $\overline{E}=(E,g_{E})$ on $X$ with $g_{E}$ a Green function of smooth type for $E$ such that $\overline{E}
\ge\overline{B}$, $\overline{E}\ge\overline{D}_{j}+2\overline{B}$, and hence, without loss of generality, $\overline{E}\ge\overline{D}_{j,n}+
\overline{B}$ for all $j,n$. Letting $\omega_{0}=\omega_{E}(g_{E})$, the functions $\varphi_{j}=g_{j}-g_{E}$ and $\varphi_{j,n}=g_{j,n}-g_
{E}$ become $\omega_{0}$-plurisubharmonic on $X$ by Proposition~\ref{prop:7} and Lemma~\ref{lem:greeneff}, respectively. Then, the 
following statements hold:
\begin{enumerate}
\item 
\label{item:10}
For $j=1,\ldots,d$, the sequences of functions $(\varphi_{j,n})_{n\in\N}$ converge in capacity to the functions~$\varphi_{j}$.
\item 
\label{item:11}
There is an equality of integrals
\begin{displaymath}
\int_{X}\big\langle (\ddc\varphi_{1}+\omega_{0})\wedge\ldots\wedge(\ddc\varphi_{d}+\omega_{0})\big\rangle=\lim_{n\to\infty}\int_{X}\big
\langle (\ddc\varphi_{1,n}+\omega_{0})\wedge\ldots\wedge(\ddc\varphi_{d,n}+\omega_{0})\big\rangle.
\end{displaymath}
\item 
\label{item:12}
The sequence of measures $\big\langle(\ddc\varphi_{1,n}+\omega_{0})\wedge\ldots\wedge(\ddc\varphi_{d,n}+\omega_{0})\big\rangle$ 
converges weakly to the measure $\big\langle(\ddc\varphi_{1}+\omega_{0})\wedge\ldots\wedge(\ddc\varphi_{d}+\omega_{0})\big\rangle$.
\item
\label{item:16} 
Let $(0,f)\in\DivhR(U)^{\inte}$ with $\vert f\vert\le C$. Let $(D_{0,m},g_{0,m})_{m\in\N}$ and $(D_{0,m},g'_{0,m})_{m\in\N}$ be sequences 
of nef model arithmetic divisors on $U$ converging in the $\overline{B}$-adic topology to $(D_{0},g)$ and $(D_{0},g')$, respectively, such 
that $f=g-g'$ and $\vert g_{0,m}-g'_{0,m}\vert\le C$. Write $f_{m}=g_{0,m}-g'_{0,m}$. Then, the net of signed measures $f_{m}\big\langle
(\ddc\varphi_{1,n}+\omega_{0})\wedge\ldots\wedge(\ddc\varphi_{d,n}+\omega_{0})\big\rangle$ converges weakly to the measure $f\big
\langle(\ddc\varphi_{1}+\omega_{0})\wedge\ldots\wedge(\ddc\varphi_{d}+\omega_{0})\big\rangle$. In particular, we have
\begin{displaymath}
\lim_{m,n\to\infty}\int_{X}f_{m}\big\langle(\ddc\varphi_{1,n}+\omega_{0})\wedge\ldots\wedge(\ddc\varphi_{d,n}+\omega_{0})\big\rangle=
\int_{X}f\big\langle(\ddc\varphi_{1}+\omega_{0})\wedge\ldots\wedge(\ddc\varphi_{d}+\omega_{0})\big\rangle.
\end{displaymath}
\end{enumerate}
\end{proposition}
\begin{proof}
By means of Theorem~\ref{thm:5}, statement~\ref{item:12} follows immediately from statements~\ref{item:10} and~\ref{item:11}.

We start by proving~\ref{item:10}. For every open subset $W\subseteq X$ containing $B$, the set $X\setminus W$ is compact. Therefore,
the Green function $g_{B}$ is continuous on $X\setminus W$ and there is a positive constant $M$ such that $0<g_{B}\vert_{X\setminus 
W}<M$. Let now $\delta>0$. Since the sequences of nef model arithmetic divisors $(D_{j,n},g_{j,n})_{n\in\N}$ on $U$ converge in the 
$\overline{B}$-adic topology to the nef adelic arithmetic divisors $(D_{j},g_{j})$ for $j=1,\ldots,d$, there is an $n_{0}\in\N$ such that
\begin{displaymath}
\vert g_{j,n}-g_{j}\vert\le\frac{\delta}{M}g_{B}
\end{displaymath}
for $n\ge n_{0}$. Thus, the set
\begin{displaymath}
\big\{x\in X\,\big\vert\,\vert g_{j,n}-g_{j}\vert>\delta\big\}
\end{displaymath}
must be contained in $W$. Therefore, we find that
\begin{displaymath}
\limsup_{n\to\infty}\Capa_{\omega}^{\ast}\big\{\vert\varphi_{j,n}-\varphi_{j}\vert>\delta\big\}\le\Capa_{\omega }(W).
\end{displaymath}
Since $B$ is a proper analytic subset, its outer capacity is zero. Therefore, we can find open sets $W\supseteq B$ with arbitrarily small 
capacity. We thus deduce that
\begin{displaymath}
\lim_{n\to\infty}\Capa_{\omega}^{\ast}\big\{\vert\varphi_{j,n}-\varphi_{j}\vert>\delta\big\}=0,
\end{displaymath}
which proves part~\ref{item:10}.

We next prove~\ref{item:11}. For this, let $0<\varepsilon<1$. Then, the convergence in the $\overline{B}$-adic topology implies that there is 
an $n_{\varepsilon}\in\N$ such that
\begin{equation}
\label{eq:5}
\varphi_{j}\le\varphi_{j,n}+\varepsilon g_{B}\qquad\text{and}\qquad\varphi_{j,n}\le\varphi_{j}+\varepsilon g_{B}
\end{equation}
for all $n\ge n_{\varepsilon}$. Since we have
\begin{align*}
\ddc(\varphi_{j,n}+\varepsilon g_{B})+\omega_{0}&=\ddc(g_{j,n}-g_{E}+\varepsilon g_{B})+\omega_{0} \\
&=\omega_{D_{j,n}}(g_{j,n})-\delta_{D_{j,n}}-\omega_{0}+\delta_{E}+\varepsilon\omega_{B}(g_{B})-\varepsilon\delta_{B}+\omega_{0} \\
&=\omega_{D_{j,n}}(g_{j,n})+\varepsilon\omega_{B}(g_{B})+\delta_{E-D_{j,n}-\varepsilon B}\ge 0,
\end{align*}
the functions $\varphi_{j,n}+\varepsilon g_{B}$ belong to $\PSH(X,\omega_{0})$ by Corollary~\ref{cor:1}. In a similar way, one shows that 
the functions $\varphi_{j}+\varepsilon g_{B}$ also belong to $\PSH(X,\omega_{0})$. \\[1mm]
\indent
The inequalities~\eqref{eq:5} in conjunction with Theorem~\ref{thm:1} imply
\begin{align}
\notag
&\int_{X}\big\langle(\ddc\varphi_{1}+\omega_{0})\wedge\ldots\wedge(\ddc\varphi_{d}+\omega_{0})\big\rangle \\
\label{eq:10}
&\qquad\le\int_{X}\big\langle(\ddc(\varphi_{1,n}+\varepsilon g_{B})+\omega_{0})\wedge\ldots\wedge(\ddc(\varphi_{d,n}+\varepsilon g_{B})+
\omega_{0})\big\rangle, \\[1mm]
\notag
&\int_{X}\big\langle(\ddc\varphi_{1,n}+\omega_{0})\wedge\ldots\wedge(\ddc\varphi_{d,n}+\omega_{0})\big\rangle \\
\label{eq:12}
&\qquad\le\int_{X}\big\langle(\ddc(\varphi_{1}+\varepsilon g_{B})+\omega_{0})\wedge\ldots\wedge(\ddc(\varphi_{d}+\varepsilon g_{B})+
\omega_{0})\big\rangle.
\end{align}
Recalling that
\begin{displaymath}
\ddc(g_{B}-g_{E})+\omega_{0}=\ddc g_{B}+\delta_{E},
\end{displaymath}
we find from the effectivity of $E$ and the properties of the non-pluripolar product that the right-hand side of~\eqref{eq:10} can be rewritten 
as
\begin{align*}
&\int_{X}\big\langle(\ddc(\varphi_{1,n}+\varepsilon g_{B})+\omega_{0})\wedge\ldots\wedge(\ddc(\varphi_{d,n}+\varepsilon g_{B})+\omega_
{0})\big\rangle \\
&=\int_{X}\big\langle\big((\ddc\varphi_{1,n}+\omega_{0})+\varepsilon(\ddc g_{B}+\delta_{E})\big)\wedge\ldots\wedge\big((\ddc\varphi_{d,n}+
\omega_{0})+\varepsilon(\ddc g_{B}+\delta_{E})\big)\big\rangle \\
&=\int_{X}\big\langle\big((\ddc\varphi_{1,n}+\omega_{0})
  +\varepsilon(\ddc(g_{B}-g_{E})+\omega_{0})\big)\wedge\ldots
  \\
&\phantom{AAAAAAAAAAA}  \ldots\wedge\big((\ddc
\varphi_{d,n}+\omega_{0})+\varepsilon(\ddc(g_{B}-g_{E})+\omega_{0})\big)\big\rangle.
\end{align*}
Now, we can apply the multilinearity of the non-pluripolar product stated in Theorem~\ref{thm:2}~\ref{item:4} to the right-hand side of the
above inequality to deduce, with a constant $C>0$ independent of $\varepsilon$, from~\eqref{eq:10} that
\begin{equation}
\label{eq:13}
\int_{X}\big\langle(\ddc\varphi_{1}+\omega_{0})\wedge\ldots\wedge(\ddc\varphi_{d}+\omega_{0})\big\rangle\le\int_{X}\big\langle(\ddc\varphi_
{1,n}+\omega_{0})\wedge\ldots\wedge(\ddc\varphi_{d,n}+\omega_{0})\big\rangle+\varepsilon C
\end{equation}
for all $n\ge n_{\varepsilon }$. In a similar way, we derive from~\eqref{eq:12} the inequality
\begin{equation}
\label{eq:14}
\int_{X}\big\langle(\ddc\varphi_{1,n}+\omega_{0})\wedge\ldots\wedge(\ddc\varphi_{d,n}+\omega_{0})\big\rangle\le\int_{X}\big\langle(\ddc
\varphi_{1}+\omega_{0})\wedge\ldots\wedge(\ddc\varphi_{d}+\omega_{0})\big\rangle+\varepsilon C
\end{equation}
for all $n\ge n_{\varepsilon }$. It is now evident that the inequalities~\eqref{eq:13} and~\eqref{eq:14} imply the claimed statement~\ref{item:11}
of the proposition.

We are finally left to prove~\ref{item:16}. To do this, we follow the proof of~\cite[Lemma 4.1]{darvas18:mnpp}. For brevity, we set
\begin{displaymath}
\mu_{n}\coloneqq\big\langle(\ddc\varphi_{1,n}+\omega_{0})\wedge\ldots\wedge(\ddc\varphi_{d,n}+\omega_{0})\big\rangle\quad\text{and}\quad
\mu\coloneqq\big\langle(\ddc\varphi_{1}+\omega_{0})\wedge\ldots\wedge(\ddc\varphi_{d}+\omega_{0})\big\rangle.
\end{displaymath}
Since, by~\ref{item:11}, we can freely add a constant to the function $f$, we can assume that $f$ and the functions $f_{m}$ are all non-negative. 
Furthermore, we need the following data: We fix $\varepsilon>0$ and let $W\Subset V\Subset X\setminus\vert E\vert$ be relatively compact 
subsets such that $\mu(X\setminus W)<\varepsilon$. We also fix a continuous function $\chi$ on $X$ and a non-negative continuous function 
$\rho$ on $X$, which is identically $1$ on $W$ and $0$ on $X\setminus V$. Since the sequence of functions $(\varphi_{j,n})_{n\in\N}$ converges
uniformly to $\varphi_{j}$ on compact subsets contained in $U$, the functions $\varphi_{j,n}$ and $\varphi_{j}$ are uniformly bounded on $V$ 
for $j=1,\ldots,d$ and $n\in\N$. From~\cite[Theorem 4.26]{guedj17:_degen_monge}, we deduce that $\chi f_{m}\mu_{n}$ converges weakly to 
$\chi f\mu$ on $V$. The usual Bedford--Taylor theory now implies that also $\mu_{n}$ converges weakly to $\mu$ on $V$; of course, we could
also use the above statement~\ref{item:12} restricted to $V$ for that.
Hence, we find
\begin{displaymath}
\liminf_{n\to\infty}\mu_{n}(W)\ge\mu(W),
\end{displaymath}
and thus
\begin{displaymath}
\limsup_{n\to\infty}\mu_{n}(X\setminus W)\le\mu(X\setminus W)\le\varepsilon.
\end{displaymath}
Since the functions $\chi$, $\rho$, $f_{m}$, and $f$ are uniformly bounded on $X$, there is a constant $C>0$ such that the quantities
\begin{displaymath}
\limsup_{m,n\to\infty}\int_{X\setminus W}\rho\vert\chi f_{m}\vert\,\mu_{n},\quad\limsup_{m,n\to\infty}\int_{X\setminus W}\vert\chi f_{m}\vert\,
\mu_{n},\quad\int_{X\setminus W}\rho\vert\chi f\vert\,\mu,\quad\int_{X\setminus W}\vert\chi f\vert\,\mu 
\end{displaymath}
are all bounded by $\varepsilon C$. Since the $\chi f_{m}\mu_{n}$ converges weakly to $\chi f\mu$ on $V$ and $\rho $ is $0$ outside of $V$, 
we deduce
\begin{displaymath}
\lim_{m,n\to\infty}\int_{X}\rho\chi f_{m}\,\mu_{n}=\int_{X}\rho\chi f\,\mu. 
\end{displaymath}
Thus, we arrive at
\begin{align*}
&\limsup_{m,n\to\infty}\Bigg\vert\int_{X}\chi f_{m}\,\mu_{n}-\int_{X}\chi f\,\mu\Bigg\vert \\
&\qquad\le\limsup_{m,n\to\infty}\Bigg\vert\int_{X}\rho\chi f_{m}\,\mu_{n}-\int_{X}\rho\chi f\,\mu\Bigg\vert+\limsup_{m,n\to\infty}\Bigg\vert\int_
{X\setminus W}(1-\rho)\chi f_{m}\,\mu_{n}-\int_{X\setminus W}(1-\rho)\chi f\,\mu\Bigg\vert \\
&\qquad\le\limsup_{m,n\to\infty}\Bigg\vert\int_{X}\rho\chi f_{m}\,\mu_{n}-\int_{X}\rho\chi f\,\mu\Bigg\vert+4\cdot\varepsilon C=4\varepsilon C. 
\end{align*}
Letting $\varepsilon$ tend to $0$, we prove the claim. 
\end{proof}

We now link the intersection product of nef adelic arithmetic divisors to the non-pluripolar product of currents. For this, it is convenient to have 
the analogue of Lemma~\ref{lemm:9} in the current setting.

\begin{lemma}
\label{lemm:7}
Let $X$ be a complex projective manifold with boundary divisor $B$ and $U=X\setminus\vert B\vert$ such that $\overline{B}$ is a nef arithmetic 
divisor. Let $\overline{D}$ be a nef adelic arithmetic divisor on $U$. Then, there is a sequence $(\overline{D}_{n})_{n\in\N}=(D_{n},g_{D_{n}})_
{n\in\N}$ of nef model arithmetic divisors on $U$ converging monotonically decreasingly to $\overline{D}$ in the $\overline{B}$-adic topology. 
That is, for all $n\in\N$, we have $D_{n+1}\le D_{n}$ and $g_{D_{n+1}}\le g_{D_{n}}$.  
\end{lemma}
\begin{proof}
Repeat the proof of Lemma~\ref{lemm:9} by putting bars on all the divisors used there.
\end{proof}

Let $\overline D=(D,g_{D})$ be a nef adelic arithmetic divisor on $U$. Choose a sufficiently large ample reference divisor $E$ on $X$ such 
that $E\ge B$ and $E\ge D+2B$. Consider the line bundle $L=\caO(E)$ and a section $s$ with $\dv(s)=E$. The Green function $g_{D}$ 
defines a singular semipositive metric $\Vert\cdot\Vert_{g_{D}}$ on $L$ by the rule
\begin{displaymath}
-\log\Vert s\Vert_{g_{D}}^{2}=g_{D}.
\end{displaymath}
We consider the b-divisor $D(L,\Vert\cdot\Vert_{g_{D}},s)$ introduced in~\cite[\S~5.1]{botero21:_chern_weil_hilber_samuel}.

\begin{proposition}
\label{prop:11} 
With the above notations, the equality $D(L,\Vert\cdot\Vert_{g_{D}},s)=b(D)$ holds. 
\end{proposition}
\begin{proof}
By~\cite[Example~5.5]{botero21:_chern_weil_hilber_samuel}, the result is true if $(D,g_{D})\in\DivhR(U)^{\model}$. Moreover, the monotonicity 
of Lelong numbers implies that, if $(D_{1},g_{D_{1}})\le(D_{2},g_{D_{2}})$ and $E\ge D_{1},D_{2}$, then $D(L,\Vert\cdot\Vert_{g_{D_{1}}},s)\le 
D(L,\Vert\cdot\Vert_{g_{D_{2}}},s)$. Let $(D_{n},g_{D_{n}})_{n\in\N}$ be a Cauchy sequence of nef model arithmetic divisors on $U$ converging 
to $(D,g_{D})$. Let $0<\varepsilon<1$ and let  $n_{\varepsilon}\in\N$ be such that
\begin{displaymath}
-\varepsilon\overline{B}\le\overline D-\overline D_{n}\le\varepsilon\overline{B}
\end{displaymath}
for all $n\ge n_{\varepsilon}$. By the monotonicity discussed above, we have
\begin{displaymath}
D(L,\Vert\cdot\Vert_{g_{D}},s)\le D(L,\Vert\cdot\Vert_{g_{D_{n}}+\varepsilon g_{B}},s)\quad\text{and}\quad D(L,\Vert\cdot\Vert_{g_{D_{n}}},s)
\le D(L,\Vert\cdot\Vert_{g_{D}+\varepsilon g_{B}},s).
\end{displaymath}
In turn, by the case of model arithmetic divisors and the additivity of Lelong numbers, this implies that 
\begin{displaymath}
D(L,\Vert\cdot\Vert_{g_{D}},s)\le D_{n}+\varepsilon B\quad \text{and}\quad D_{n}\le D(L,\Vert\cdot\Vert_{g_{D}},s)+\varepsilon B.
\end{displaymath}
Hence the sequence $(D_{n})_{n\in\N}$ converges to $D(L,\Vert\cdot\Vert_{g_{D}},s)$, which proves the claimed equality.   
\end{proof}

\begin{corollary} 
\label{cor:7}
For $j=1,\ldots,d$, let $\overline{D}_{j}=(D_{j},g_{j})$ be nef adelic arithmetic divisors on $U$. Choose a sufficiently large arithmetic reference 
divisor $\overline{E}=(E,g_{E})$ on $X$ with $g_{E}$ a Green function of smooth type for $E$ such that $\overline{E}\ge\overline{B}$ and
$\overline{E}\ge\overline{D}_{j}+2\overline{B}$. Letting $\omega_{0}=\omega_{E}(g_{E})$, the functions $\varphi_{j}=g_{j}-g_{E}$ become 
$\omega_{0}$-plurisubharmonic on $X$ by Proposition~\ref{prop:7}. Then, the equality
\begin{displaymath}
D_{1}\cdot\ldots\cdot D_{d}=\int_{X}\big\langle(\ddc\varphi_{1}+\omega_{0})\wedge\ldots\wedge(\ddc\varphi_{d}+\omega_{0})\big\rangle  
\end{displaymath}
holds.
\end{corollary}
\begin{proof}
This is a consequence of Corollary~\ref{cor:6}, Proposition~\ref{prop:11}, and~\cite[Theorem~5.20]{botero21:_chern_weil_hilber_samuel}.
\end{proof}

\subsection{The global arithmetic case}
\label{sec:glob-arithm-case}

We recall the definition of an arithmetic variety. Our definition is a particular case of an arithmetic variety in the sense of~\cite{MR1087394}. 

\begin{definition} 
Let $K$ be a number field and $\caO_{K}$ its ring of integers. An \emph{arithmetic variety $\caX$ over $\Spec(\caO_{K})$} is a flat quasi-projective 
normal scheme over $\Spec(\caO_{K})$ such that its generic fiber $\caX_{K}=\caX\times_{\Spec(\caO_{K})}\Spec(K)$ is smooth. The arithmetic
variety $\caX$ is called \emph{projective}, if $\caX$ is projective over $\Spec(\caO_{K})$.
\end{definition}

If $\caX$ is an arithmetic variety over $\Spec(\caO_{K})$ and $\Sigma$ is the set of complex embeddings of $K$, we set
\begin{displaymath}
\caX_{\Sigma}\coloneqq\bigcup_{\sigma\in\Sigma}\caX_{K}\times_{\sigma}\Spec(\C),
\end{displaymath}
which gives rise to the complex manifold $X=\caX_{\Sigma}(\C)$. 

\begin{notation}
Given an arithmetic variety $\caX$ over $\Spec(\caO_{K})$ together with its complex manifold $X=\caX_{\Sigma}(\C)$ as above, we adopt in
the sequel the following notation: $\R$-Cartier divisors on $\caX$ will be denoted by calligraphic letters, the corresponding divisors induced by 
base change on $X$ will be denoted by the corresponding latin letter. More concretely, if $\caD\in\DivR(\caX)$ is an $\R$-Cartier divisor on 
$\caX$, it determines by base change a divisor $\caD_{\Sigma}$ on $\caX_{\Sigma}$, which gives rise to an $\R$-Cartier divisor on the complex
manifold $X$, which is consequently denoted by $D$.
\end{notation}

\begin{definition}
Let $\caX$ be an arithmetic variety over $\Spec(\caO_{K})$ and $X$ its associated complex manifold as above. An \emph{arithmetic divisor (with 
$\R$-coefficients) on $\caX$} is a pair $\overline{\caD}=(\caD,g_{D})$, where $\caD\in\DivR(\caX)$ and $g_{D}$ is a Green function on $X$ of 
smooth type for $D$. We let $\DivhR(\caX)$ denote the group of arithmetic divisors (with $\R$-coefficients) on $\caX$.
\end{definition}

We recall the definitions of nef and ample arithmetic divisors. 

\begin{definition}
\label{def:20}
Let $\caX$ be a projective  arithmetic variety over $\Spec(\caO_{K})$. An arithmetic divisor $\overline{\caD}=(\caD,g_{D})\in\DivhR(\caX)$ is called
\emph{nef}, if $\caD$ is relatively nef, the Green function $g_{D}$ for $D$ is of plurisubharmonic type, and the height $h_{\overline{\caD}}(x)\ge 0$ 
for every point $x\in\caX(\overline{K})$. It is called \emph{ample}, if $\caD$ is relatively ample, the Green function $g_{D}$ for $D$ is of plurisubharmonic 
type, and the height $h_{\overline{\caD}}(x)> 0$ for every point $x\in\caX(\overline{K})$.
\end{definition}

We now let $\caX$ be a projective arithmetic variety over $\Spec(\caO_{K})$ and let $X$ be its associated complex manifold as above. Let $\caB$ 
be an effective Cartier divisor on $\caX$ with support $\vert\caB\vert$ and let $\caU\coloneqq\caX\setminus\vert\caB\vert$, as well as $U\coloneqq 
X\setminus B$; we note that $\vert\caB\vert$ may contain whole fibers. Let $g_{B}$ be a Green function of smooth type for $B$ such that there is 
an $\eta>0$ with $g_{B}(x)>\eta$ for all $x\in U$. Such Green functions exist because $B$ being effective, any Green function of smooth type for 
$B$ is bounded below, and we can add to it a big enough constant to make it bigger than $\eta$. We call $\overline{\caB}=(\caB,g_{B})$ an \emph
{arithmetic boundary divisor for $\caX$}.

\smallskip
We now repeat the preceding subsections in the arithmetic setting. For this, we denote by $R(\caX,\caU)$ the category of all proper normal 
modifications of $\caX$ which are isomorphisms over $\caU$. 

\begin{definition} 
\label{def:18} 
The space of \emph{model arithmetic divisors on $\caU$ of smooth type}
is defined as the limit 
\begin{displaymath}
\DivhR(\caU)^{\model}\coloneqq\varinjlim_{\pi\in R(\caX,\caU)}\DivhR(\caX_{\pi}).
\end{displaymath}
Similarly, one can define \emph{model arithmetic divisors on $\caU$ of continuous type}, by asking that the Green functions in the definition of 
$\DivhR(\caX_{\pi})$ are only of continuous type. 
\end{definition}

\begin{definition}
\label{def:21}
A model arithmetic divisor on $\caU$, given by $\overline{\caD}=(\caD,g_{D})\in\DivhR(\caX_{\pi})$ for some $\pi\in R(\caX,\caU)$, is called a 
\emph{nef model arithmetic divisor on $\caU$}, if $\overline{\caD}$ is nef in the sense of Definition~\ref{def:20}. The set of these divisors defines 
a convex cone.
\end{definition}

On $\DivhR(\caU)^{\model}$ there is a $\overline{\caB}$-adic norm defined as in Subsection~\ref{sec:local-arithm-case}.

\begin{definition}
\label{def:12}
The space $\DivhR(\caU)^{\adel}$ of \emph{adelic arithmetic divisors on $\caU$} is the completion of $\DivhR(\caU)^{\model}$ with respect to the 
$\overline{\caB}$-adic topology induced by the $\overline{\caB}$-adic norm. Again, elements of $\DivhR(\caU)^{\adel}$ are represented by Cauchy 
sequences of model arithmetic divisors on $\caU$ of smooth type. As in Remark~\ref{rem:9}, a Cauchy sequence of model arithmetic divisors on 
$\caU$ of continuous type also defines an adelic arithmetic divisor on $\caU$.
\end{definition}

Clearly, if $(\caD_{n},g_{n})_{n\in\N}$ is a Cauchy sequence with respect to the $\overline{\caB}$-adic topology, then the sequence $(\caD_{n})_
{n\in\N}$ is a Cauchy sequence in the $\caB$-adic topology, hence defines an adelic divisor $\caD\in\DivR(\caU)^{\adel}$. Moreover, the sequence 
of functions $(g_{n})_{n\in\N}$ is Cauchy in the $g_{B}$-adic topology. This implies that the sequence of functions $(g_{n})_{n\in\N}$ converges
pointwise to a function $g$ on $U$. Moreover, this convergence is uniform on compact subsets of $U$. The global analogue of Proposition~\ref{prop:3} 
is given by the following statement.

\begin{corollary}
\label{cor:4}
There is a short exact sequence
\begin{displaymath}
0\longrightarrow C^{0}(U)_{0}\longrightarrow\DivhR(\caU)^{\adel}\longrightarrow\DivR(\caU)^{\adel}\longrightarrow 0,
\end{displaymath}
where $\DivR(\caU)^{\adel}$ are the geometric adelic divisors on $\caU$ as in Subsection~\ref{sec:geom-case-adel}. 
\end{corollary}

We next recall the notions of nef adelic arithmetic divisors and integrable adelic arithmetic divisors. These are adaptations of~\cite[\S 2.1.5]
{YuanZhang:adelic} and~\cite[Corollary~5.7]{zhang95:_posit} to divisors with real coefficients.

\begin{definition}
\label{def:13}
The cone $\DivhR(\caU)^{\nef}$ of \emph{nef adelic arithmetic divisors on $\caU$} is the closure of the convex cone of nef model arithmetic 
divisors on $\caU$ inside of $\DivhR(\caU)^{\adel}$, i.\,e., any nef adelic arithmetic divisor $\overline{\caD}$ on $\caU$ can be represented by a 
Cauchy sequence $(\overline{\caD}_{n})_{n\in\N}$ of nef model arithmetic divisors on $\caU$. The space of \emph{integrable adelic arithmetic 
divisors on $\caU$} is the vector space
\begin{displaymath}
\DivhR(\caU)^{\inte}\coloneqq\DivhR(\caU)^{\nef}-\DivhR(\caU)^{\nef}.
\end{displaymath}
\end{definition}

We make the following easy observation.

\begin{lemma}
Let $(\caD,g_{1}),(\caD,g_{2})\in\DivhR(\caU)^{\adel}$. Assume that $(\caD,g_{1})\in\DivhR(\caU)^{\nef}$ and that $(D,g_{2})\in\DivhR(U)^
{\nef}$. If $g_{2}\ge g_{1}$, then $(\caD,g_{2})\in\DivhR(\caU)^{\nef}$.  
\end{lemma}

The following results are proven in~\cite{YuanZhang:adelic}.

\begin{proposition}
\label{prop:4}
If $\pi\in R(\caX,\caU)$ and $\overline{\caB}'=(\caB',g_{B'})$ is an effective arithmetic divisor on $\caX_{\pi}$ with $\vert\caB'\vert=\caX_{\pi}
\setminus\caU$, then the $\overline{\caB}'$-adic norm is equivalent to the $\overline{\caB}$-adic norm. Therefore, the definitions of the spaces 
$\DivhR(\caU)^{\adel}$, $\DivhR(\caU)^{\nef}$, and $\DivhR(\caU)^{\inte}$ only depend on $\caU$ and not on a particular choice of arithmetic 
boundary divisor $\overline{\caB}$ for $\caX$.
\end{proposition}

\begin{theorem}
\label{thm:4} 
The intersection pairing of nef arithmetic divisors on $\caX$ extends by continuity to a unique symmetric multilinear pairing
\begin{displaymath}
\underbrace{\DivhR(\caU)^{\nef}\otimes\ldots\otimes\DivhR(\caU)^{\nef}}_{\text{$\dim(\caU)$-\emph{times}}}\longrightarrow\R.
\end{displaymath}
This pairing can be extended by linearity to a pairing 
\begin{displaymath}
\underbrace{\DivhR(\caU)^{\inte}\otimes\ldots\otimes\DivhR(\caU)^{\inte}}_{\text{$\dim(\caU)$-\emph{times}}}\longrightarrow\R.
\end{displaymath}
\end{theorem}

\begin{remark}
\label{rem:4}
As in the geometric case, Theorem~\ref{thm:4} means that, if $\overline{\caD}_{j}$ with $j=0,\ldots,d$ are nef adelic arithmetic divisors on 
$\caU$ represented by Cauchy sequences $(\overline{\caD}_{j,n})_{n\in\N}$ of nef model arithmetic divisors on $\caU$, then the limit 
\begin{equation}
\label{eq:15}
\overline{\caD}_{0}\cdot\ldots\cdot\overline{\caD}_{d}\coloneqq\lim_{(n_{0},\dots,n_{d})\to(\infty,\ldots,\infty)}\overline{\caD}_{0,n_{0}}\cdot\ldots
\cdot\overline{\caD}_{d,n_{d}}
\end{equation}
exists and is independent of the choice of approximating sequences.
\end{remark}

\begin{remark}
\label{rem:10}
If $\caV\subseteq\caU$ is a smaller open subset, then there is a map $\DivhR(\caU)\rightarrow\DivhR(\caV)$ and the intersection product is the 
same computed in both spaces. Therefore, for the purpose of computing arithmetic intersection products it is harmless to enlarge the boundary 
divisor, for example in order to ask it to be ample or nef. 
\end{remark}

The analogue of Lemma~\ref{lemm:9} in the current context is the following.

\begin{lemma}
\label{lemm:5} 
Let $\caX$ be a projective arithmetic variety over  $\Spec(\caO_{K})$ with boundary divisor $\caB$ and $\caU=\caX\setminus\caB$ such that
$\overline{\caB}$ is a nef arithmetic divisor. Let $\overline{\caD}$ be a nef adelic arithmetic divisor on $\caU$. Then, there is a sequence 
$(\overline{\caD}_{n})_{n\in\N}=(\caD_{n},g_{n})_{n\in\N}$ of nef model arithmetic divisors on $\caU$ converging monotonically decreasingly 
to $\overline{\caD}$ in the $\overline{\caB}$-adic topology. That is, for all $n\in\N$, we have $\caD_{n+1}\le\caD_{n}$ and $g_{n+1}\le g_{n}$.  
\end{lemma}
\begin{proof}
Copy the proof of Lemma~\ref{lemm:7}.
\end{proof}

The next result deals with integrable adelic arithmetic divisors on $\caU$ whose geometric part is trivial.

\begin{lemma}
\label{lemm:3}
Let $\caX$ be an arithmetic variety over  $\Spec(\caO_{K})$, $\overline{\caB}$ a nef arithmetic boundary divisor for $\caX$, and $(0,f)\in
\DivhR(\caU)^{\inte}$. Then, the following statements hold:
\begin{enumerate}
\item 
\label{item:17} 
We can choose sequences of nef model arithmetic divisors $(\caD_{n},g_{n})_{n\in\N}$ and $(\caD_{n},g'_{n})_{n\in\N}$ on $\caU$ with the same 
divisorial part, converging to $(\caD,g)$ and to $(\caD,g')$, respectively, and such that $f=g-g'$.
\item 
\label{item:18} 
If the function $f$ is globally bounded, i.\,e., $\vert f\vert\le C$, then we can choose the above sequences in such a way that $\vert g_{n}-g'_{n}\vert
\le C$ for all $n\in\N$.   
\end{enumerate}
\end{lemma}
\begin{proof}
To prove~\ref{item:17}, we note that since $(0,f)$ is an integrable adelic arithmetic divisor on $\caU$, there exist sequences of nef model arithmetic 
divisors $(\caD_{0,n},g_{0,n})_{n\in\N}$ and $(\caD'_{0,n},g'_{0,n})_{n\in\N}$ on $\caU$ converging to $(\caD,g)$ and $(\caD,g')$, respectively, such 
that $f=g-g'$. After adding to $(\caD_{0,n},g_{0,n})$ a small positive multiple of $\overline{\caB}$ (which continues to be a nef model arithmetic divisor 
on $\caU$ by the nefness of $\overline{\caB}$), we can further assume that $\caD_{0,n}\ge\caD'_{0,n}$. For $n\in\N$, we now set
\begin{displaymath}
\caD_{n}\coloneqq\caD_{0,n},\qquad g_{n}\coloneqq\max(g_{0,n},g'_{0,n}-n),\qquad g'_{n}\coloneqq\max(g'_{0,n},g_{0,n}-n). 
\end{displaymath}
Then, the functions $g_{n}$ and $g'_{n}$ are Green functions for $D_{n}$ of continuous and of plurisubharmonic type. We claim that the sequence 
$(\caD_{n},g_{n})_{n\in\N}$ converges in the $\overline{\caB}$-adic topology to $(\caD,g)$ and the sequence $(\caD_{n},g'_{n})_{n\in\N}$ converges 
in the $\overline{\caB}$-adic topology to $(\caD,g')$. To prove the claim we start by showing that for each $\varepsilon>0$, there exists an $n_{0}\in\N$ 
such that the inequalities
\begin{equation}
\label{claim1}
g-\varepsilon g_{B}\le g_{n}\le g+\varepsilon g_{B}
\end{equation}
hold for all $n\ge n_{0}$. By Corollary~\ref{cor:4}, the function $f$ is continuous on $U$, and we have that $\vert f\vert=o(g_{B})$ when approaching 
the boundary $B$. Furthermore, since $X$ is compact, there is a constant $C>0$ such that
\begin{displaymath}
\vert g-g'\vert=\vert f\vert\le C+\frac{\varepsilon}{2}g_{B}
\end{displaymath}
for given $\varepsilon>0$. Now, choose $n_{0}\ge C$ such that the inequalities
\begin{displaymath}
\vert g_{0,n}-g\vert\le\varepsilon g_{B}\qquad\text{and}\qquad\vert g'_{0,n}-g'\vert\le\frac{\varepsilon}{2}g_{B}
\end{displaymath}
hold for all $n\ge n_{0}$. Then, for $n\ge n_{0}$, we derive the bounds
\begin{displaymath}
g'_{0,n}-n\le g'-n+\frac{\varepsilon}{2}g_{B}\le g+C-n+\varepsilon g_{B}\le g+\varepsilon g_{B}.
\end{displaymath}
Since we have that $g-\varepsilon g_{B}\le g_{0,n}\le g+\varepsilon g_{B}$, we deduce for $n\ge n_{0}$ that
\begin{displaymath}
g-\varepsilon g_{B} \le g_{0,n}\le g_{n}=\max(g_{0,n},g'_{0,n}-n)\le g+\varepsilon g_{B}, 
\end{displaymath}
which proves the claimed inequality~\ref{claim1} and thus the convergence of the sequence $(g_{n})_{n\in\N}$ to $g$. Similarly, one shows that the 
sequence $(g'_{n})_{n\in\N}$ converges to $g'$. We note, if needed, using global regularization of plurisubharmonic functions that we can change the 
Cauchy sequences in question so that the Green functions are of smooth and of plurisubharmonic type.

To prove~\ref{item:18}, in case that $\vert f\vert\le C$, it is enough to define
\begin{displaymath}
g_{n}\coloneqq\max(g_{0,n},g'_{0,n}-C)\qquad\text{and}\qquad g'_{n}\coloneqq\max(g'_{0,n},g_{0,n}-C).
\end{displaymath}
The same argument as above now completes the proof in this case. 
\end{proof}

The last technical result we need to prove states that, given two adelic arithmetic divisors on $\caU$ with the same divisorial part, we can approximate 
the more singular one by Green functions that have the same singularity type as the less singular one.

\begin{proposition}
\label{prop:12}
Let $\caX$ be an arithmetic variety over  $\Spec(\caO_{K})$ and $\overline{\caB}$ a nef arithmetic boundary divisor for $\caX$. Let $(\caD,g)$ and 
$(\caD,g')$ be elements of $\DivhR(\caU)^{\adel}$ sharing the same divisorial part. Furthermore, assume that $g'\le g$ at every point of $U$. For $n
\in\N$, we set
\begin{displaymath}
\gamma_{n}=\max(g-n,g').
\end{displaymath}
Then, for each $n\in\N$, the arithmetic divisor $(\caD,\gamma_{n})$ belongs to $\DivhR(\caU)^{\adel}$ and the sequence $(\caD,\gamma_{n})_{n\in\N}$ 
converges to $(\caD,g')$ in the $\overline{\caB}$-adic topology. Moreover, if $(\caD,g)$ and $(\caD,g')$ belong to $\DivhR(\caU)^{\nef}$, the same is true 
for $(\caD,\gamma_{n})$. 
\end{proposition}
\begin{proof}
By Corollary~\ref{cor:4}, the difference $g-g'$ belongs to $C^{0}(U)_{0}$. Since $0\le\gamma_{n}-g'\le g-g'$, we deduce that $\gamma_{n}-g'$ also 
belongs to $C^{0}(U)_{0}$, and by the same corollary, the pair $(\caD,\gamma_{n})$ thus belongs to $\DivhR(\caU)^{\adel}$. 

We next prove that the sequence $(\caD,\gamma_{n})_{n\in\N}$ converges in the $\overline{\caB}$-adic topology to $(\caD,g')$. Since $g-g'$ belongs to 
$C^{0}(U)_{0}$, there is a continuous function $h$ that vanishes on $\vert B\vert$ and such that $g-g'=h\cdot g_{B}$ on $U$. For every $\varepsilon>0$, 
let
\begin{displaymath}
K_{\varepsilon}=\{x\in X\,\vert\,h(x)\ge\varepsilon\},
\end{displaymath}
which is a compact subset of $U$. Since $h\cdot g_{B}$ is continuous on $U$, there is an integer $n_{\varepsilon}\ge 1$ such that $h\cdot g_{B}\le 
n_{\varepsilon}$ in $K_{\varepsilon }$. For $n\ge n_{\varepsilon }$, we claim that
\begin{equation}
\label{eq:19}
0\le\gamma_{n}-g'\le\varepsilon g_{B}.
\end{equation}
Obviously, the left-hand inequality is clear. For the right-hand inequality, let $x\in U$. If $g(x)-n\le g'(x)$, then $\gamma_{n}(x)=g'(x)$, and the right-hand 
inequality follows from the positivity of $g_{B}$. If, on the other hand, $g(x)-n>g'(x)$, then $h(x)\cdot g_{B}(x)=g(x)-g'(x)>n\ge n_{\varepsilon}$. Therefore, 
$x\not\in K_{\varepsilon}$, and hence $h(x)<\varepsilon$. From this, we derive
\begin{displaymath}
\gamma_{n}(x)-g'(x)=g(x)-g'(x)-n\le g(x)-g'(x)=h(x)\cdot g_{B}(x)<\varepsilon g_{B}(x),
\end{displaymath}
which proves the right-hand inequality in~\eqref{eq:19}. The claimed convergence now follows directly from~\eqref{eq:19}.

Assume finally that $(\caD,g)$ and $(\caD,g')$ belong to $\DivR(\caU)^{\nef}$. By Lemma~\ref{lemm:3}, we can find sequences $(\caD_{m},g_{m})_{m
\in\N}$ and $(\caD_{m},g'_{m})_{m\in\N}$ of nef model arithmetic divisors on $\caU$ sharing the same divisorial part and converging to $(\caD,g)$ and 
$(\caD,g')$, respectively. Writing $\gamma_{n,m}=\max (g_{m}-n,g'_{m})$, we obtain a sequence $(\caD_{m},\gamma_{n,m})_{m\in\N}$ of nef continuous 
model arithmetic divisors on $\caU$ that converges to $(\caD,\gamma_{n})$ showing that $(\caD,\gamma_{n})\in\DivR(U)^{\nef}$.    
\end{proof}

\subsection{Adelic arithmetic line bundles}
\label{sec:adelic-line-bundles}

We start this subsection by recalling the notion of $\Q$-line bundles. Let $X$ be a scheme. Denote by $\bPic(X)$ the category of locally free sheaves 
of rank one and by $\Pic(X)$ the group of isomorphism classes of objects in $\bPic(X)$ with group law given by the tensor product. The category 
$\bPicQ(X)$ of $\Q$-line bundles on $X$ is the category whose objects are pairs $(n,L)$, where $n\in\N_{>0}$ and $L$ is a locally free sheaf of rank 
one on $X$, and whose morphisms are
\begin{displaymath}
\Hom_{\bPicQ(X)}\big((n,L),(n',L')\big)=\varinjlim_{a\in\N_{>0}}\Hom_{\bPic(X)}\big(L^{\otimes an'},L'^{\otimes an}\big);
\end{displaymath}
The group structure on $\bPicQ(X)$ is given by the tensor product

\begin{displaymath}
(n,L)\otimes(n',L')=(nn',L^{\otimes n'}\otimes~L'^{\otimes n}).
\end{displaymath}
The group of 
isomorphisms classes of $\bPicQ(X)$ is denoted by $\PicQ(X)$ and is nothing else than
\begin{displaymath}
\PicQ(X)=\Pic(X)\otimes_{\Z}\Q.
\end{displaymath}
The global $\Q$-sections of a $\Q$-line bundle are defined as
\begin{displaymath}
\Gamma\big(X,(n,L)\big)=\varinjlim_{a\in\N_{>0}}\Gamma\big(X,L^{\otimes a}\big),
\end{displaymath}
where the maps $\Gamma(X,L^{\otimes a})\rightarrow\Gamma(X,L^{\otimes ab})$ are given by $s\mapsto s^{\otimes b}$ for $b\in\N_{>0}$. A global 
$\Q$-section of $(n,L)$, represented by a global section $s$ of $L^{a}$, is formally given by $s^{\otimes 1/an}$. We note that, for $n,n',a,a'\in\N_{>0}$, 
the groups $\Gamma(X,(n,L^{\otimes a}))$ and $\Gamma(X,(n',L^{\otimes a'}))$ are canonically isomorphic and that we have the identification
\begin{displaymath}
\Gamma\big(X,(n,L)\big)=\Hom_{\bPicQ(X)}\big((1,\caO_X),(n,L)\big).
\end{displaymath}
The divisor of a global $\Q$-section is defined as
\begin{displaymath}
\dv(s^{\otimes 1/an})=\frac{1}{an}\dv(s)\in\DivQ(X).
\end{displaymath}

Let now $\caX$ be an arithmetic variety over $\Spec(\caO_{K})$. The category $\bPich(\caX)$ of hermitian line bundles on $\caX$ is the category 
whose objects are pairs $\overline{\caL}=(\caL,\Vert\cdot\Vert)$, where $\caL$ is a locally free sheaf of rank one on $\caX$ and $\Vert\cdot\Vert$ is 
a hermitian metric on the base change $\caL_{\Sigma}$ over $\caX_{\Sigma}$, and whose morphisms are 
\begin{displaymath}
\Hom_{\bPich(\caX)}(\overline{\caL},\overline{\caL}')=\big\{\varphi\in\Hom_{\bPic(\caX)}(\caL,\caL')\,\big\vert\,\Vert\varphi(v)\Vert\le\Vert v\Vert\big\}.
\end{displaymath}
For instance, $\overline{\caO}_{\caX}$ will denote the trivial line bundle $\caO_{\caX}$ with the metric $\Vert 1\Vert=1$. We note that $\Hom_{\bPich
(\caX)}(\overline{\caO},\overline{\caL})$ can be identified with the set of \emph{small sections}
\begin{displaymath}
\widehat{\Gamma}(\caX,\overline{\caL})=\big\{s\in\Gamma(\caX,\overline{\caL})\,\big\vert\,\Vert s\Vert\le 1\big\}.
\end{displaymath}
The group $\Pich(\caX)$ is the set of isomorphism classes of objects in $\bPich(\caX)$ with group law given by the tensor product. 

The category $\bPichQ(\caX)$ of hermitian $\Q$-line bundles on $\caX$ is the category whose objects are pairs $(n,\overline{\caL})$, where $n\in
\N_{>0}$ and $\overline{\caL}\in\bPich(\caX)$, and whose morphisms are 
\begin{displaymath}
\Hom_{\bPichQ(\caX)}\big((n,\overline{\caL}),(n',\overline{\caL}')\big)=\varinjlim_{a\in\N_{>0}}\Hom_{\bPich(\caX)}\big(\overline{\caL}^{\otimes an'},
\overline{\caL}'^{\otimes an}\big). 
\end{displaymath}
If $(n,\overline{\caL}_{0})\in\bPichQ(\caX)$ and $s_{0}$ is a rational section of $\caL_{0}^{\otimes a}$, a power of the underlying line bundle $\caL_
{0}$, then $s_{0}$ defines a section $s=s_{0}^{\otimes 1/an}$, whose associated arithmetic divisor is defined as
\begin{displaymath}
\divh(s)=\frac{1}{an}(\dv(s_{0}),-\log\Vert s_{0}\Vert^{2}).
\end{displaymath}

Following~\cite{YuanZhang:adelic}, we introduce the category of adelic arithmetic line bundles. To this end, we assume that the arithmetic variety 
$\caX$ is projective with arithmetic boundary divisor $\overline{\caB}$ and $\caU=\caX\setminus\vert\caB\vert$. Note that $\caU$ is a quasi-projective 
arithmetic variety.

Let $\caL\in\bPicQ(\caU)$. A sequence of triples $(\caX_{\pi_{n}},\overline{\caL}_{n},\ell_{n})_{n\in\N}$, where $\pi_{n}\in R(\caX,\caU)$, $\overline
{\caL}_{n}\in\bPichQ(\caX_{\pi_{n}})$, and $\ell_{n}\colon\caL\rightarrow\caL_{n}\vert_{\caU}$ is an isomorphism, is called a Cauchy sequence (in the 
$\overline{\caB}$-adic topology), if for any non-zero rational section $s$ of $\caL$, the sequence $(\divh(\ell_{n}(s)))_{n\in\N}$ is a Cauchy sequence 
with respect to the $\overline{\caB}$-adic norm of model arithmetic divisors.

The category of adelic arithmetic line bundles $\bPichQ(\caU)^{\adel}$
is the category of pairs
\begin{displaymath}
  (\caL,(\caX_{\pi_{n}},\overline{\caL}_{n},\ell_{n})_{n\in\N}),
\end{displaymath}
where $\caL\in\bPicQ(\caU)$ and $(\caX_{\pi_{n}},\overline{\caL}_{n},\ell_{n})_{n\in\N}$ is a Cauchy sequence as before. A 
morphism between $(\caL,(\caX_{\pi_{n}},\overline{\caL}_{n},\ell_{n})_{n\in\N})$ and $(\caL',(\caX'_{\pi_{n}},\overline{\caL}'_{n},\ell'_{n})_{n\in\N})$ 
is an isomorphism $\lambda\colon\caL\rightarrow\caL'$ over $\caU$ such that, if the write
\begin{displaymath}
(\caY_{\pi_{k}},\overline{\caM}_{k},m_{k})=
\begin{cases}
(\caX_{\pi_{n}},\overline{\caL}_{n},\ell_{n}),&\text{ if }k=2n, \\
(\caX'_{\pi_{n}},\overline {\caL}'_{n},\ell'_{n}\circ\lambda),&\text{ if }k=2n-1,
\end{cases}
\end{displaymath}
the sequence $(\caY_{\pi_{k}},\overline{\caM}_{k},m_{k})_{k\in\N}$ is Cauchy. As we have defined $\bPichQ(\caU)^{\adel}$, it is a groupoid, that is, 
all the morphisms are isomorphisms. The tensor product of hermitian line bundles induces a tensor product in $\bPichQ(\caU)^{\adel}$. The group 
of adelic arithmetic line bundles $\PichQ(\caU)^{\adel}$ is the set of isomorphism classes in $\bPichQ(\caU)^{\adel}$ with group law given by the 
tensor product.

\begin{remark}
\label{rem:15}
Let $\overline{\caL}=(\caL,(\caX_{\pi_{n}},\overline{\caL}_{n},\ell_{n})_{n\in\N})$ be an adelic arithmetic line bundle on $\caU$ and let $U=\caU(\C)$. 
All the line bundles $\caL_{n}$, when restricted to $U$, agree with $L=\caL\vert_{U}$. The metric of each $\overline{\caL}_{n}$ induces a metric 
$\Vert\cdot\Vert_{n}$ on $L$. These metrics converge (uniformly on compact subsets of $U$) to a metric $\Vert\cdot\Vert$ on $L$ over $U$. We 
call this metric \emph{the archimedean component} of $\overline{\caL}$.  
\end{remark}

\begin{remark}
\label{rem:13}
Let $\overline{\caL}=(\caL,(\caX_{\pi_{n}},\overline{\caL}_{n},\ell_{n})_{n\in\N})$ be an adelic arithmetic line bundle on $\caU$ and $s$ a non-zero 
rational section of $\caL$. Then, the sequence $(\divh(\ell_{n}(s)))_{n\in\N}$ determines an adelic arithmetic divisor in $\DivhR(\caU)^{\adel}$ that we 
will denote simply by $\divh(s)$. We will say that $\overline{\caL}=(\caL,(\caX_{\pi_{n}},\overline{\caL}_{n},\ell_{n})_{n\in\N})$ is nef if and only if $\divh
(s)$ is nef for any rational section $s$ of $\caL$.
\end{remark}

\subsection{Functorial properties of the adelic arithmetic intersection product}
\label{sec:funct-prop-adel}

We well only discuss the functoriality in the global arithmetic case. The first functorial property is the restriction to a smaller open subset.

Let $\caX$ be an arithmetic variety over $\Spec(\caO_{K})$. Let $\caU'\subseteq\caU$ be two open subsets. If $\pi\colon\caX_{1}\rightarrow\caX$ 
belongs to $R(\caX,\caU)$, it also belongs to $R(\caX,\caU')$. Therefore, there is a map
\begin{displaymath}
\DivhR(\caU)^{\model}\longrightarrow\DivhR(\caU')^{\model}.  
\end{displaymath}

Let $\overline{\caB}=(\caB,g_{B})$ and $\overline{\caB}'=(\caB',g_{B'})$ be boundary divisors with $\caU=\caX\setminus\vert\caB\vert$ and $\caU'=
\caX\setminus\vert\caB'\vert$. Then, there is a constant $c>0$ such that $\overline{\caB}\le c\overline{\caB}'$. Therefore, the $\overline{\caB}$-adic 
and $\overline{\caB}'$-adic norms satisfy
\begin{displaymath}
\Vert\cdot\Vert_{\overline{\caB}'}\le c\,\Vert\cdot\Vert_{\overline{\caB}}.
\end{displaymath}
This implies that there is a continuous restriction map
\begin{equation}
\label{eq:21}
\DivhR(\caU)^{\adel}\longrightarrow\DivhR(\caU')^{\adel}.
\end{equation}
This map sends $\DivhR(\caU)^{\nef}$ to $\DivhR(\caU')^{\nef}$.

More generally, let $\caX$ and $\caX'$ be arithmetic varieties over $\Spec(\caO_{K})$, and let $\overline{\caB}$ and $\overline{\caB}'$ be boundary 
divisors on $\caX$ and $\caX'$, respectively. Write $\caU=\caX\setminus\vert\caB\vert$ and $\caU'=\caX'\setminus\vert\caB'\vert$. For every dominant 
morphism $f\colon\caU'\rightarrow\caU$, there is a pull-back map
\begin{displaymath}
f_{\adel}^{\ast}\colon\DivhR(\caU)^{\adel}\longrightarrow\DivhR(\caU')^{\adel} 
\end{displaymath}
that sends $\DivhR(\caU)^{\nef}$ to $\DivhR(\caU')^{\nef}$. This pull-back map is constucted as follows. Given $\pi\in R(\caX,\caU)$,
there exist $\pi'\in R(\caX',\caU')$ and $f_{\pi}\colon\caX'_{\pi'}\rightarrow\caX_{\pi}$ such that the diagram
\begin{align*}
\xymatrix{\caU'\ar[d]\ar[r]^{f}&\caU\ar[d] \\
\caX'_{\pi'}\ar[r]^{f_{\pi }}&\caX_{\pi} 
}
\end{align*}
commutes. From this we obtain a pull-back map
\begin{displaymath}
f_{\mathrm{mod}}^{\ast}\colon\DivhR(\caU)^{\model}\longrightarrow\DivhR(\caU')^{\model}
\end{displaymath}
induced by mapping $(\caD,g_{D})\in\DivhR(\caX_{\pi})$ to $(f_{\pi}^{\ast}\caD,f_{\pi}^{\ast}g_{D})\in\DivhR(\caX'_{\pi'})$. After taking the respective 
$\overline{\caB}$-adic and $\overline{\caB}'$-adic completions, we obtain the pull-back maps
\begin{displaymath}
f_{\mathrm{adel}}^{\ast}\colon\DivhR(\caU)^{\adel}\longrightarrow\DivhR(\caU')^{\adel}\qquad\text{and}\qquad f_{\mathrm{adel}}^{\ast}\colon\DivhR
(\caU)^{\nef}\longrightarrow\DivhR(\caU')^{\nef}.
\end{displaymath}
For more details see~\cite[\S~2.5.5]{YuanZhang:adelic}.

\begin{remark}
\label{rem:44}
We have defined pull-back maps of adelic divisors only for dominant morphisms. By contrast, for adelic arithmetic line bundles, given an arbitary 
morphism $f\colon\caU'\rightarrow\caU$, there is a well defined pull-back functor $f^{\ast}\colon\bPichQ(\caU)^{\adel}\rightarrow\bPichQ(\caU')^
{\adel}$ that preserves nefness and tensor products.
\end{remark}

The intersection product and the pull-back are related by the following projection formula.

\begin{proposition}
\label{prop:16}
Let $f\colon\caU'\rightarrow\caU$ be a dominant morphism as before. Furthermore, assume that $f$ is projective and flat of relative dimension $e$, and
let $d$ be the relative dimension of $\caU$ over $\Spec(\caO_{K})$. Let $\overline{\caD}_{0},\ldots,\overline{\caD}_{d}$ be integrable adelic arithmetic 
divisors on $\caU$ and $\overline{\caE}_{1},\ldots,\overline{\caE}_{e}$ be integrable adelic arithmetic divisors on $\caU'$. 
Then, we have the formula
\begin{displaymath}
\overline{\caE}_{1}\cdot\ldots\cdot\overline{\caE}_{e}\cdot f_{\mathrm{adel}}^{\ast}\overline{\caD}_{0}\cdot\ldots\cdot f_{\mathrm{adel}}^{\ast}\overline
{\caD}_{d}=\deg(\caE_{1}\vert_{F},\ldots,\caE_{e}\vert_{F})\cdot\overline{\caD}_{0}\cdot\ldots\cdot\overline{\caD}_{d},
\end{displaymath}
where $F$ is a generic fiber of $f$ and $\deg(\caE_{1}\vert_{F},\ldots,\caE_{e}\vert_{F})$ is the degree of the generic fiber with respect to the underlying 
geometric divisors. 
\end{proposition}
\begin{proof}
See~\cite[Lemma~4.6.1]{YuanZhang:adelic}.
\end{proof}

\begin{proposition}
\label{prop:17}
Let $f\colon\caU'\rightarrow\caU$ be a finite flat dominant morphism of degree $n$ and let $d$ be the relative dimension of $\caU$ over $\Spec(\caO_
{K})$. Let $\overline{\caD}_{0},\ldots,\overline{\caD}_{d}$ be integrable adelic arithmetic divisors on $\caU$. Then, we have the formula
\begin{displaymath}
f_{\mathrm{adel}}^{\ast}\overline{\caD}_{0}\cdot\ldots\cdot f_{\mathrm{adel}}^{\ast}\overline{\caD}_{d}=n\,\overline{\caD}_{0}\cdot\ldots\cdot\overline
{\caD}_{d}.
\end{displaymath}
\end{proposition}
\begin{proof}
This result also is covered in~\cite[Lemma~4.6.1]{YuanZhang:adelic}.
\end{proof}

\section{Extension of Yuan--Zhang's arithmetic intersection product }
\label{sec:extension-yuan-zhang}

In the study of the arithmetic of pure Shimura varieties, log-log singular line bundles appear in a natural way. The paradigmatic example is the Hodge 
bundle on the moduli space of principally polarized abelian varieties with a level structure. In~\cite[\S~5.5]{YuanZhang:adelic}, it is proven that the 
Hodge bundle $\omega$ equipped with the Faltings metric is an example of an adelic arithmetic line bundle, which is denoted by $\overline{\omega}$.
However, we will prove in subsection~\ref{sec:non-integr-line} that the adelic arithmetic line bundle $\overline{\omega}$ is not integrable. Therefore, 
the arithmetic self-intersection product of $\overline{\omega}$ is not defined by the theory developed by X.~Yuan and S.\,W.~Zhang in~\cite
{YuanZhang:adelic}. On the other hand, we know from the results of~\cite{BurgosKramerKuehn:cacg} that this arithmetic self-intersection product 
exists. In this section, we will extend the theory of X.~Yuan and S.\,W.~Zhang in Theorem~\ref{thm:fund} to be able to define arithmetic self-intersection 
products for certain adelic arithmetic divisors, which are not necessarily integrable like $\overline {\omega}$. For this, we will make use of the theory 
of relative energy introduced in~\cite{darvas18:mnpp} in an essential way.

We emphasize that the corresponding constructions in the toric setting have been carried out in the thesis~\cite{Peralta} of Gari Peralta. In particular,
we note that in the toric setting the height with respect to a toric, continuous, and semipositive metric is given by an integral of a roof function over a 
polytope, and this formula has been extended to the case of toric, singular metrics, where the respective height is given as an integral of a roof function 
over a convex body. Moreover, the results of this section have been extended by Y.~Cai and W.~Gubler in~\cite{CaiGubler} to the setting of a proper
adelic base curve in the sense of H.~Chen and A.~Moriwaki allowing also at the non-archimedean places singular metrics.

\subsection{Adelic arithmetic intersections and mixed relative energy}
\label{fourth-section}

In this subsection, we relate the adelic arithmetic intersection product with the non-pluripolar product of currents and the mixed relative energy. 
This will enable us to define arithmetic self-intersection products for certain adelic arithmetic divisors, which are not necessarily integrable.

As in Subsection~\ref{sec:glob-arithm-case}, let $K$ be a number field with ring of integers $\caO_{K}$, let $\caX$ be a projective arithmetic 
variety of dimension $d$ over $\Spec(\caO_{K})$, and let $X$ be its associated complex manifold. Let $\overline{\caB}=(\caB,g_{B})$ be an 
arithmetic boundary divisor for $\caX$ and write as before $\caU=\caX\setminus\vert\caB\vert$, as well as $U=X\setminus B$. Furthermore,
we assume that $\overline{\caB}$ is nef to be able to use Lemma~\ref{lemm:3}.

\begin{lemma}
\label{lemm:6}
For $j=1,\ldots,d$, let $\overline{\caD}_{j}=(\caD_{j},g_{j})\in\DivhR(\caU)^{\nef}$ be nef adelic arithmetic divisors on $\caU$. Moreover, let 
$\overline{\caD}_{0}=(\caD_{0},g_{0}),\,\overline{\caD}'_{0}=(\caD_{0},g'_{0})\in\DivhR(\caU)^{\inte}$ be two integrable adelic arithmetic divisors 
on $\caU$ sharing the same divisorial part and such that $f=g_{0}-g'_{0}$ is globally bounded. Choose sufficiently large reference divisors 
$E_{j}$ on $X$ such that $E_{j}\ge B$ and $E_{j}\ge D_{j}+2B$. Then, the equality
\begin{displaymath}
\overline{\caD}_{0}\cdot\overline{\caD}_{1}\cdot\ldots\cdot\overline{\caD}_{d}-\overline{\caD}'_{0}\cdot\overline{\caD}_{1}\cdot\ldots\cdot\overline
{\caD}_{d}=\int_{X}f\,\big\langle\omega_{E_{1}}(g_{1})\wedge\ldots\wedge\omega_{E_{d}}(g_{d})\big\rangle 
\end{displaymath}
holds.
\end{lemma}
\begin{proof}
By multilinearity, the arithmetic intersection product $(0,f)\cdot\overline{\caD}_{1}\cdot\ldots\cdot\overline{\caD}_{d}$ does not depend on how 
we write $(0,f)$ as a difference of adelic arithmetic divisors on $\caU$ sharing the same divisorial part. Therefore, we can assume that $(\caD_
{0},g_{0})$ and $(\caD_{0},g'_{0})$ are nef adelic arithmetic divisors on $\caU$. Thus, there are Cauchy sequences $(\overline{\caD}_{0,n})_
{n\in\N}$, $(\overline{\caD}'_{0,n})_{n\in\N}$, and $(\overline{\caD}_{j,n})_{n\in \N}$ of nef model arithmetic divisors on $\caU$ converging in 
the $\overline{\caB}$-adic topology to $\overline{\caD}_{0}$, $\overline{\caD}'_{0}$, and $\overline{\caD}_{j}$ for $j=1,\ldots,d$, respectively.  
By Lemma~\ref{lemm:3}, we can assume that $\overline{\caD}_{0,n}=(\caD_{0,n},g_{0,n})$ and $\overline{\caD}'_{0,n}= (\caD_{0,n},g'_{0,n})$ 
share the same divisorial part and that the functions $g_{0,n}-g'_{0,n}$ are globally bounded by the same bound as the one for $f$. By assumption, 
we have $E_{j}\ge D_{j}+2B$, which implies that $E_{j}\ge D_{j,n}$ for all $n\in\N$ after possibly taking suitable subsequences. By Lemma
\ref{lem:greeneffreal}, the Green functions $g_{j,n}$ thus become Green functions for $E_{j}$ of plurisubharmonic type for all $j=1,\ldots,d$ 
and for all $n\in\N$. From this we obtain
\begin{align*}
\overline{\caD}_{0}\cdot\overline{\caD}_{1}\cdot\ldots\cdot\overline{\caD}_{d}-\overline{\caD}'_{0}\cdot\overline{\caD}_{1}\cdot\ldots\cdot\overline
{\caD}_{d}&=\lim_{n\to\infty}\big(\overline{\caD}_{0,n}\cdot\overline{\caD}_{1,n}\cdot\ldots\cdot\overline{\caD}_{d,n}-\overline{\caD}'_{0,n}\cdot
\overline{\caD}_{1,n}\cdot\ldots\cdot\overline{\caD}_{d,n}\big) \\[2mm]
&=\lim_{n\to\infty}\int_{X}(g_{0,n}-g'_{0,n})\,\omega_{D_{1,n}}(g_{1,n})\wedge\ldots\wedge\omega_{D_{d,n}}(g_{d,n}) \\[2mm]
&=\lim_{n\to\infty}\int_{X}(g_{0,n}-g'_{0,n})\,\big\langle\omega_{E_{1}}(g_{1,n})\wedge\ldots\wedge\omega_{E_{d}}(g_{d,n})\big\rangle,
\end{align*}
where, for the last equality, we have used Proposition~\ref{prop:5}. Proposition~\ref{prop:6}~\ref{item:16}  now shows that
\begin{displaymath}
\lim_{n\to\infty}\int_{X}(g_{0,n}-g'_{0,n})\,\big\langle\omega_{E_{1}}(g_{1,n})\wedge\ldots\wedge\omega_{E_{d}}(g_{d,n})\big\rangle=
\int_{X}(g_{0}-g'_{0})\,\big\langle\omega_{E_{1}}(g_{1})\wedge\ldots\wedge\omega_{E_{d}}(g_{d})\big\rangle,
\end{displaymath}
which proves the claim.
\end{proof}

Now we are able to relate the adelic arithmetic intersection product with the mixed relative energy.

\begin{theorem}
\label{thm:fund}
For $j=0,\ldots,d$, let $\overline{\caD}_{j}=(\caD_{j},g_{j})\in\DivhR(\caU)^{\nef}$ and $\overline{\caD}'_{j}=(\caD_{j},g'_{j})\in\DivhR(\caU)^
{\nef}$ be two nef adelic arithmetic divisors on $\caU$ sharing the same divisorial part. Choose sufficiently large arithmetic reference divisors 
$\overline{\caE}_{j}=(\caE_{j},g_{E_{j}})$ on $\caX$ such that $\overline{\caE}_{j}\ge\overline{\caB}$ and $\overline{\caE}_{j}\ge\overline
{\caD}_{j}+2\overline{\caB}$ with $g_{E_{j}}$ Green functions of smooth type for $E_{j}$, so that the functions $\varphi_{j}\coloneqq g_{j}-
g_{E_{j}}$ and $\varphi'_{j}\coloneqq g'_{j}-g_{E_{j}}$ become $\omega_{j}$-plurisubharmonic, where $\omega_{j}=\omega_{E_{j}}(g_{E_
{j}})$. Assume that the Green functions $g_{j}$'s are more singular than the $g'_{j}$'s, that is, $[\varphi_{j}]\le[\varphi'_{j}]$. As in Subsection
\ref{sec:relative-energy}, we write $\pmb{\varphi}=(\varphi_{0},\ldots,\varphi_{d})$, $\pmb{\varphi}'=(\varphi'_{0},\ldots,\varphi'_{d})$, and 
$\pmb{\omega}=(\omega_{0},\ldots,\omega_{d})$. Then, we have
\begin{enumerate}
\item 
\label{item:27} 
$\pmb{\varphi}\in\caE(X,\pmb{\omega},\pmb{\varphi}')$.
\item 
\label{item:30} 
$\pmb{\varphi}\in\caE^{1}(X,\pmb{\omega},\pmb{\varphi}')$ and the formula
\begin{displaymath}
\overline{\caD}_{0}\cdot\ldots\cdot\overline{\caD}_{d}=\overline{\caD}'_{0}\cdot\ldots\cdot\overline{\caD}'_{d}+I_{\pmb{\varphi}'}(\pmb
{\varphi})
\end{displaymath}
holds, where $I_{\pmb{\varphi}'}(\pmb{\varphi})$ is the mixed relative energy of $\pmb{\varphi}$ with respect to $\pmb{\varphi}'$ defined 
in Definition~\ref{def:mre}.
\end{enumerate}
\end{theorem}
\begin{proof}
To prove~\ref{item:27}, we just note that for $j=0,\ldots,d$, we have by Corollary~\ref{cor:7} 
\begin{displaymath}
\int_{X}\big\langle(\ddc\varphi_{j}+\omega_{j})^{\wedge d}\big\rangle=D_{j}^{d}=\int_{X}\big\langle(\ddc\varphi'_{j}+\omega_{j})^{\wedge 
d}\big\rangle.
\end{displaymath}
Together with the assumption $[\varphi_{j}]\le[\varphi'_{j}]$, this implies that $\varphi_{j}\in\caE(X,\omega _{j},\varphi'_{j})$.
  
To prove~\ref{item:30}, assume for the moment that $\vert g_{j}-g'_{j}\vert$ is bounded for $j=0,\ldots,d$. Since the nef adelic arithmetic 
divisors $\overline{\caD}_{j}$ and $\overline{\caD}'_{j}$ have the same divisorial part for $j=0,\ldots,d$, we compute, using Proposition
\ref{prop:5} and Lemma~\ref{lemm:6} that
\begin{align*}
&\overline{\caD}_{0}\cdot\ldots\cdot\overline{\caD}_{d}-\overline{\caD}'_{0}\cdot\ldots\cdot\overline{\caD}'_{d} \\
&\qquad=\sum_{k=0}^{d}\Big(\overline{\caD}_{0}\cdot\ldots\cdot\overline{\caD}_{k-1}\cdot\overline{\caD}_{k}\cdot\overline{\caD}'_{k+1}
\cdot\ldots\cdot\overline{\caD}'_{d}-\overline{\caD}_{0}\cdot\ldots\cdot\overline{\caD}_{k-1}\cdot\overline{\caD}'_{k}\cdot\overline{\caD}'_
{k+1}\cdot\ldots\cdot\overline{\caD}'_{d}\Big) \\
&\qquad=\sum_{k=0}^{d}\,\int_{X}(g_{k}-g'_{k})\big\langle\omega_{E_{0}}(g_{0})\wedge\ldots\wedge\omega_{E_{k-1}}(g_{k-1})\wedge
\omega_{E_{k+1}}(g'_{k+1})\wedge\ldots\wedge\omega_{E_{d}}(g'_{d})\big\rangle.
\end{align*}
Since we now have for $j=0,\ldots,d$ that
\begin{align*}
\omega_{E_{j}}(g_{j})&=\ddc g_{j}+\delta_{E_{j}}=\ddc(g_{j}-g_{E_{j}})+\delta_{E_{j}}+\ddc g_{E_{j}} \\[1mm]
&=\ddc\varphi_{j}+\delta_{E_{j}}+\omega_{j}-\delta_{E_{j}}=\ddc\varphi_{j}+\omega_{j},
\end{align*}
and similarly $\omega_{E_{j}}(g'_{j})=\ddc\varphi'_{j}+\omega_{j}$, we find, using $g_{k}-g'_{k}=\varphi_{k}-\varphi'_{k}$, that 
\begin{displaymath}
\overline{\caD}_{0}\cdot\ldots\cdot\overline{\caD}_{d}-\overline{\caD}'_{0}\cdot\ldots\cdot\overline{\caD}'_{d}=I_{\pmb{\varphi}'}(\pmb
{\varphi}),
\end{displaymath}
after recalling formula~\eqref{keyfor}.

Next, we consider the general case when $\vert g_{j}-g'_{j}\vert$ is not necessarily bounded. For each $n\in\N$, we write
\begin{displaymath}
\varphi_{j,n}=\max(\varphi'_{j}-n,\varphi_{j})\qquad\text{and}\qquad g_{j,n}=g_{E_{j}}+\varphi_{j,n}. 
\end{displaymath}
Since the functions $\varphi_{j}$ are more singular than the functions $\varphi'_{j}$, we deduce that $g_{j}-n\le g_{j,n}\le g'_{j}$. Thus, 
the functions $\vert g_{j,n}-g'_{j}\vert$ are bounded. Writing $\overline{\caD}_{j,n}=(\caD_{j},g_{j,n})$, we deduce from the preceding case 
that 
\begin{displaymath}
\overline{\caD}_{0,n}\cdot\ldots\cdot\overline{\caD}_{d,n}-\overline{\caD}'_{0}\cdot\ldots\cdot\overline{\caD}'_{d}=I_{\pmb{\varphi}'}(\pmb
{\varphi}_{n}),
\end{displaymath}
where $\pmb{\varphi}_{n}=(\varphi_{0,n},\dots,\varphi_{d,n})$. By the continuity of the arithmetic intersection pairing and Proposition~\ref
{prop:12}, we have that
\begin{displaymath}
\lim_{n\to\infty}\overline{\caD}_{0,n}\cdot\ldots\cdot\overline{\caD}_{d,n}=\overline{\caD}_{0}\cdot\ldots\cdot\overline{\caD}_{d}.
\end{displaymath}
Moreover, since the sequence $(\varphi_{j,n})_{n\in\N}$ converges monotonically to $\varphi_{j}$ for $j=0,\ldots,d$, we know from 
Proposition~\ref{prop:10} that 
\begin{displaymath}
\lim_{n\to\infty}I_{\pmb{\varphi}'}(\pmb{\varphi}_{n})=I_{\pmb{\varphi}'}(\pmb{\varphi}). 
\end{displaymath}
This completes the proof of the theorem.
\end{proof}

% \begin{definition} 
% \label{def:28}
% For $j=0,\ldots,d$, let $\overline{D}_{j}=(D_{j},g_{j})\in\DivhR(\caU)^{\nef}$ be nef adelic arithmetic divisors on $\caU$. Let $\overline{E}_
% {j}=(E_{j},g_{E_{j}})\ge\overline{D}_{j}+2\overline{B}$ be sufficiently large arithmetic reference divisors on $\caX$ with $g_{E_{j}}$ Green
% functions of smooth type for $E_{j}$, so that the functions $\varphi_{j}=g_j-g_{E_j}$ become $\omega_{j}$-plurisubharmonic, where
% $\omega_{j}=\omega_{E_{j}}(g_{E_{j}})$. As in section~\ref{sec:relative-energy}, we write $\pmb{\varphi}=(\varphi_{0},\ldots,\varphi_{d})$ 
% and similarly for $\pmb{\omega}$. Then, for $\pmb{\varphi}'=(\varphi'_{0},\ldots,\varphi'_{d})\in\caE(X,\pmb{\omega},\pmb{\varphi})$, we 
% define the arithmetic intersection product
% \begin{displaymath}
% (D_{0},\varphi'_{0}+g_{E_{0}})\cdot\ldots\cdot(D_{d},\varphi'_{d}+g_{E_{d}})\coloneqq\overline{D}_{0}\cdot\ldots\cdot\overline{D}_{d}+
% I_{\pmb{\varphi}}(\pmb{\varphi}').
% \end{displaymath}
% This intersection product can be $-\infty$; it is finite if and only if $\pmb{\varphi}'\in\caE^{1}(X,\pmb{\omega},\pmb{\varphi})$.
%\end{definition}

Before being able to define generalized arithmetic intersection products, we need the following definition.

\begin{definition}
Let $\overline{\caD}'=(\caD,g')\in\DivhR(\caU)^{\nef}$ be a nef adelic arithmetic divisor on $\caU$. A function $g$ is said to be a \emph
{Green function for $D$ of plurisubharmonic type and relative full mass with respect to $g'$}, if the following three conditions are satisfied:
\begin{enumerate}
\item
\label{item:32} 
The function $g-g'$ is bounded above, that is, $g$ is more singular than $g'$. 
\item 
\label{item:33} 
The function $g\vert_{U}$ is plurisubharmonic. 
\item 
\label{item:34} 
If $\overline{\caE}=(\caE,g_{E})$ is a model arithmetic divisor with $\caE\ge\caD$ on $\caX$ and $g_{E}$ a Green function of smooth type 
for $E$ with $\omega_{0}=\omega_{E}(g_{E})$, the function $\varphi'=g'-g_{E}$ is $\omega_{0}$-plurisubharmonic by the nefness of $\overline
{\caD}'$ and conditions~\ref{item:32} and~\ref{item:33} imply that the function $\varphi=g-g_{E}$ is also $\omega_{0}$-plurisubharmonic, 
and we further require the relative full mass condition 
\begin{displaymath}
\int_{X}\big\langle (\ddc\varphi+\omega_{0})^{\wedge d}\big\rangle=\int_{X}\big\langle(\ddc\varphi'+\omega_{0})^{\wedge d}\big\rangle;
\end{displaymath}
in other words, $\varphi\in\caE(X,\omega_{0},\varphi')$. 
\end{enumerate}
\end{definition}

Theorem~\ref{thm:fund} together with the above definition allows us to make the following definition for generalized arithmetic intersection 
numbers.

\begin{definition}
\label{def:29}
For $j=0,\ldots,d$, let $\overline{\caD}'_{j}=(\caD_{j},g'_{j})\in\DivhR(\caU)^{\nef}$ be nef adelic arithmetic divisors on $\caU$ and let 
$\overline{\caD}_{j}=(\caD_{j},g_{j})\in\DivhR(\caU)$ be adelic arithmetic divisors on $\caU$ so that $g_{j}$ are Green functions for 
$D_{j}$ of plurisubharmonic type and relative full mass with respect to $g'_{j}$. Choose sufficiently large arithmetic reference divisors 
$\overline{\caE}_{j}=(\caE_{j},g_{E_{j}})$ on $\caX$ such that $\overline{\caE}_{j}\ge\overline{\caB}$ and $\overline{\caE}_{j}\ge\overline
{\caD}_{j}+2\overline{\caB}$ with $g_{E_{j}}$ Green functions of smooth type for $E_{j}$, so that the functions $\varphi_{j}\coloneqq g_{j}-
g_{E_{j}}$ and $\varphi'_{j}\coloneqq g'_{j}-g_{E_{j}}$ become $\omega_{j}$-plurisubharmonic, where $\omega_{j}=\omega_{E_{j}}(g_{E_
{j}})$. Then, for $\pmb{\varphi}=(\varphi_{0},\ldots,\varphi_{d})\in\caE(X,\pmb{\omega},\pmb{\varphi}')$, the \emph{generalized arithmetic
intersection product $\overline{\caD}_{0}\cdot\ldots\cdot\overline{\caD}_{d}$} is defined as 
\begin{equation}
\label{eq:30}
\overline{\caD}_{0}\cdot\ldots\cdot\overline{\caD}_{d}\coloneqq\overline{\caD}'_{0}\cdot\ldots\cdot\overline{\caD}'_{d}+I_{\pmb{\varphi}'}
(\pmb{\varphi}).
\end{equation}
We will say that \emph{$(\overline{\caD}_{0},\ldots,\overline{\caD}_{d})$ has finite energy with respect to $(\overline{\caD}'_{0},\ldots,
\overline{\caD}'_{d})$}, if $I_{\pmb{\varphi}'}(\pmb{\varphi})>-\infty$. In this case, the arithmetic intersection product \eqref{eq:30} is a real 
number.

Similarly, if $\overline{\caD}'=(\caD,g')\in\DivhR(\caU)^{\nef}$ is a nef adelic arithmetic divisor on $\caU$ and $\overline{\caD}=(\caD,g)\in
\DivhR(\caU)$ is an adelic arithmetic divisor on $\caU$ so that $g$ is a Green function for $D$ of plurisubharmonic type and relative full 
mass with respect to $g'$, we say that \emph{$\overline{\caD}$ has finite energy with respect to $\overline{\caD}'$}, if $(\overline{\caD},
\ldots,\overline{\caD})$ has finite energy with respect to $(\overline{\caD}',\ldots,\overline{\caD}')$.
\end{definition}

\begin{lemma}
\label{lemm:11} 
The generalized arithmetic intersection product of Definition~\ref{def:29} does not depend on the choice of the arithmetic reference divisors 
$\overline{\caE}_{j}$ $(j=0,\ldots,d)$.
\end{lemma}
\begin{proof}
This follows easily from the definitions. 
\end{proof}

\subsection{Functoriality} 
\label{sixth-section}

In this subsection, we investigate the functorial behavior of the generalized arithmetic intersection product of Definition~\ref{def:29}. It turns
out that all the functorial properties of Subsection~\ref{sec:funct-prop-adel} carry over to the finite energy case. In particular, we have the 
following two propositions.

\begin{proposition}
\label{prop:18}
Let $f\colon\caU'\rightarrow\caU$ be a dominant morphism as in Subsection~\ref{sec:funct-prop-adel}. Furthermore, assume that $f$ is 
projective and flat of relative dimension $e$, and let $d$ be the relative dimension of $\caU$ over $\Spec(\caO_{K})$. Let $\overline{\caD}_
{0},\ldots,\overline{\caD}_{d}$ be adelic arithmetic divisors on $\caU$ and $\overline{\caE}_{1},\ldots,\overline{\caE}_{e}$ be adelic arithmetic 
divisors on $\caU'$ having finite energy with respect to some nef adelic arithmetic divisors. Then, we have the formula
\begin{displaymath}
\overline{\caE}_{1}\cdot\ldots\cdot\overline{\caE}_{e}\cdot f_{\mathrm{adel}}^{\ast}\overline{\caD}_{0}\cdot\ldots\cdot f_{\mathrm{adel}}^{\ast}
\overline{\caD}_{d}=\deg(\caE_{1}\vert_{F},\ldots,\caE_{e}\vert_{F})\cdot\overline{\caD}_{0}\cdot\ldots\cdot\overline{\caD}_{d},
\end{displaymath}
where $F$ is a generic fiber of $f$ and $\deg(\caE_{1}\vert_{F},\ldots,\caE_{e}\vert_{F})$ is the degree of the generic fiber with respect to the 
underlying geometric divisors. 
\end{proposition}

Similarly, we have

\begin{proposition}\label{prop:14}
Let $f\colon\caU'\rightarrow\caU$ be a finite flat dominant morphism of degree $n$ and let $d$ be the relative dimension of $\caU$ over 
$\Spec(\caO_{K})$. Let $\overline{\caD}_{0},\ldots,\overline{\caD}_{d}$ be adelic arithmetic divisors on $\caU$ having finite energy with 
respect to some nef adelic arithmetic divisors. Then, we have the formula
\begin{displaymath}
f_{\mathrm{adel}}^{\ast}\overline{\caD}_{0}\cdot\ldots\cdot f_{\mathrm{adel}}^{\ast}\overline{\caD}_{d}=n\,\overline{\caD}_{0}\cdot\ldots
\cdot\overline{\caD}_{d}.
\end{displaymath}
\end{proposition}

\section{The self-intersection of the line bundle of Siegel--Jacobi forms}
\label{seventh-section}

\subsection{The line bundle of Siegel--Jacobi forms}

The basic reference for this subsection is the book~\cite{FC} by G.~Faltings and C-L. Chai. We let $\sA_{g}$ denote the moduli stack 
of principally polarized abelian schemes of dimension $g$ over $\Spec(\Z)$ and we let $\pi\colon\sB_{g}\rightarrow\sA_{g}$ be the 
universal abelian scheme over $\Spec(\Z)$ with zero section $\varepsilon\colon\sA_{g}\rightarrow\sB_{g}$. Recall that $d'=g(g+1)/2$ 
is the relative dimension of $\sA_{g}$ over $\Spec(\Z)$ and that $d=d'+g$ is the relative dimension of $\sB_{g}$ over $\Spec(\Z)$. 

The \emph{Hodge bundle $\omega$ on $\sA_{g}$} is defined as the line bundle
\begin{displaymath}
\omega\coloneqq\varepsilon^{\ast}\det(\Omega^{1}_{\sB_{g}/\sA_{g}}).
\end{displaymath}
For $g>1$, the global sections of the line bundle $\omega^{\otimes k}$ are arithmetic Siegel modular forms of weight $k$, i.\,e., Siegel 
modular forms of weight $k$ with integral Fourier coefficients. By identifying $\sB_{g}$ with the dual abelian scheme $\sB_{g}^{\vee}$ 
via the principal polarization, the Poincar\'e bundle $\sP_{\sB_{g}}$ is the tautological invertible sheaf on the product $\sB_{g}\times_
{\sA_{g}}\sB_{g}$. By means of the diagonal morphism
\begin{displaymath}
\Delta\colon\sB_{g}\longrightarrow\sB_{g}\times_{\sA_{g}}\sB_{g},
\end{displaymath}
we obtain via pull-back the distinguished line bundle $\sL\coloneqq\Delta^{\ast}\sP_{\sB_{g}}$, which is $\pi$-ample and symmetric. The 
restriction of $\sL$ to any fiber of $\sB_{g}$ is the line bundle associated with twice the principal polarization of the fiber. The \emph{line 
bundle of Siegel--Jacobi forms of weight $k$ and index $m$ on $\sB_{g}$} is now defined as
\begin{displaymath}
\sJ_{k,m}\coloneqq\pi^{\ast}\omega^{\otimes k}\otimes\sL^{\otimes m}.
\end{displaymath}
For $g>1$, the global sections of the line bundle $\sJ_{k,m}$ are arithmetic Siegel--Jacobi forms of weight $k$ and index $m$, i.\,e., 
Siegel--Jacobi forms of weight $k$ and index $m$ with integral Fourier coefficients. The arithmetic theory of Siegel--Jacobi forms has 
been investigated in detail in~\cite{Kr}.

Next, we fix an integer $N\ge 3$ and let $\sA _{g,N}$ denote the fine moduli space of principally polarized abelian schemes of dimension 
$g$ with full level $N$-structure over $\Spec(\Z[1/N,\zeta_{N}])$, where $\zeta_{N}$ is a primitive $N$-th root of unity, and we let $\pi_{N}
\colon\sB_{g,N}\rightarrow\sA_{g,N}$ be the universal abelian scheme over $\Spec(\Z[1/N,\zeta_{N}])$ with zero section $\varepsilon_{N}
\colon\sA_{g,N}\rightarrow\sB_{g,N}$. We recall that $\sA_{g,N}$ and $\sB_{g,N}$ are smooth, quasi-projective schemes over $\Spec(\Z
[1/N,\zeta_{N}])$. We denote by $\rho'_{N}\colon\sA_{g,N}\rightarrow\sA_{g}$ and by $\rho_{N}\colon\sB_{g,N}\rightarrow\sB_{g}$ the 
respective morphisms of stacks forgetting the level $N$-structure.

As before, the \emph{Hodge bundle $\omega_{N}$ of level $N$ on $\sA_{g,N}$} is defined as the line bundle
\begin{displaymath}
\omega_{N}\coloneqq\varepsilon_{N}^{\ast}\det(\Omega^{1}_{\sB_{g,N}/\sA_{g,N}}).
\end{displaymath}
For $g>1$, the global sections of the line bundle $\omega_{N}^{\otimes k}$ are arithmetic Siegel modular forms of weight $k$ and level $N$, 
i.\,e., Siegel modular forms of weight $k$ and level $N$ with Fourier coefficients lying in $\Z[1/N,\zeta_{N}]$. Recalling from~\cite[Chapter~IV, 
Theorem~6.7]{FC} that $\sB_{g,N}$ is the pull-back of $\sB_{g}$ to $\sA_{g,N}$, we observe that $\omega_{N}\cong\rho_{N}^{\prime\ast}\,
\omega$. The pull-back $\sL_{N}\coloneqq\rho_{N}^{\ast}\sL$ becomes a line bundle on the scheme $\sB_{g,N}$, which is again $\pi_{N}
$-ample and symmetric. For the sequel, we note that the symmetric line bundle $\sL_{N}$ satisfies the theorem of the cube with respect to 
the morphism $[2]\colon\sB_{g,N}\rightarrow\sB_{g,N}$ given by fiberwise multiplication by $2$, i.\,e., there is an isomorphism
\begin{equation}
\label{eq:38}
[2]^{\ast}\sL_{N}\cong\sL_{N}^{\otimes 4}
\end{equation}
of line bundles over $\sB_{g,N}$. The \emph{line bundle of Siegel--Jacobi forms of weight $k$, index $m$, and level $N$ on $\sB_{g,N}$} is 
now defined as
\begin{displaymath}
\sJ_{k,m,N}\coloneqq\pi_{N}^{\ast}\omega_{N}^{\otimes k}\otimes\sL_{N}^{\otimes m}.
\end{displaymath}
For $g>1$, the global sections of the line bundle $\sJ_{k,m,N}$ are arithmetic Siegel--Jacobi forms of weight $k$, index $m$, and level $N$, 
i.\,e., Siegel--Jacobi forms of weight $k$, index $m$, and level $N$ with Fourier coefficients lying in $\Z[1/N,\zeta_{N}]$. 

\subsection{Siegel--Jacobi forms over the complex numbers}

In this subsection, we recall the classical definition of Siegel--Jacobi forms over the complex numbers. We start by recalling the construction 
of $\sA_{g,N}(\C)$ and $\sB_{g,N}(\C)$ as quotient spaces. Let $g\ge 1$ be an integer. The symplectic group $\Sp(2g,\R)$ is the group of 
real $(2g\times 2g)$-matrices of the form
\begin{equation}
\label{eq:34}
\begin{pmatrix}A&B\\C&D\end{pmatrix}
\end{equation}
such that
\begin{displaymath}
A^{t}C=C^{t}A,\quad D^{t}B=B^{t}D,\quad A^{t}D-C^{t}B=\id_{g},
\end{displaymath}
where $\id_{g}$ is the $(g\times g)$-identity matrix. There is an inclusion 
\begin{displaymath}
\Sp(2g,\R)\hookrightarrow\Sp(2g+2,\R),
\end{displaymath}
sending a matrix of the form~\eqref{eq:34} to the matrix
\begin{displaymath}
\begin{pmatrix}A&0&B&0\\0&1&0&0\\C&0&D&0\\0&0&0&1\end{pmatrix}.
\end{displaymath}
For any commutative ring $R$, let $H_{R}^{(g,1)}$ be the Heisenberg group
\begin{displaymath}
H_{R}^{(g,1)}=\big\{[(\lambda,\mu),x]\,\big\vert\,\lambda,\mu\in R^{(1,g)},\,x\in R\big\},
\end{displaymath}
where $R^{(1,g)}$ denotes the set of row vectors of size $g$ with coefficients in $R$, with composition law given by
\begin{displaymath}
[(\lambda,\mu),x]\circ[(\lambda',\mu'),x']=[(\lambda+\lambda',\mu+\mu'),x+x'+\lambda{\mu'}^{t}-\mu{\lambda'}^{t}]. 
\end{displaymath}
The real Heisenberg group $H_{\R}^{(g,1)}$ can be realized as the subgroup of $\Sp(2g+2,\R)$ consisting of matrices of the form
\begin{displaymath}
\begin{pmatrix}\id_{g}&0&0&\mu^{t}\\\lambda &1&\mu&x\\0&0&\id_{g}&-\lambda^{t}\\0&0&0&1\end{pmatrix}.
\end{displaymath}
The full Jacobi group $G^{(g,1)}_{\R}=\Sp(2g,\R)\ltimes H^{(g,1)}_{\R}$ is the subgroup of $\Sp(2g+2,\R)$ generated by $\Sp(2g,\R)$ and 
$H^{(g,1)}_{\R}$.

Let
\begin{displaymath}
\H_{g}=\{Z=X+iY\,\vert\,X,Y\in\mathrm{Mat}_{g\times g}(\R),\,Z^{t}=Z,\,Y>0\} 
\end{displaymath}
be the Siegel upper half-space. The group $\Sp(2g,\R)$ acts transitively on $\H_{g}$, where for $M=\big(\begin{smallmatrix}A&B\\C&D
\end{smallmatrix}\big)\in\Sp(2g,\R)$ and $Z \in\H_{g}$ this action is given by 
\begin{displaymath}
Z\longmapsto M\langle Z\rangle=(AZ+B)(CZ+D)^{-1}.
\end{displaymath}
On the other hand, the group $G^{(g,1)}_{\R}$ acts transitively on $\H_{g}\times\C^{(1,g)}$, where for $\big(M,[(\lambda,\mu),x]\big)\in G^
{(g,1)}_{\R}$ and $(Z,W)\in\H_{g}\times\C^{(1,g)}$ this action is given by
\begin{displaymath}
(Z,W)\longmapsto\big(M\langle Z\rangle,(W+\lambda Z+\mu)(CZ+D)^{-1}\big).
\end{displaymath}
For a subgroup $\Gamma\subseteq\Sp(2g,\Z)$ of finite index, we write $\widetilde{\Gamma}=\Gamma\ltimes H^{(g,1)}_{\Z}\subseteq G^
{(g,1)}_{\Z}$, and define the quotient spaces $\sA_{g,\Gamma}\coloneqq\Gamma\backslash\H_{g}$ and $\sB_{g,\Gamma}\coloneqq
\widetilde{\Gamma}\backslash\H_{g}\times \C^{(1,g)}$. In particular, these definitions apply to the principal congruence subgroup of level 
$N$ of the symplectic group defined by
\begin{displaymath}
\Gamma(N)=\ker\big(\Sp(2g,\Z)\longrightarrow\Sp(2g,\Z/N\Z)\big),
\end{displaymath}
for which we then find that
\begin{displaymath}
\sA_{g,\Gamma(N)}=\sA_{g,N}(\C)\quad\text{and}\quad\sB_{g,\Gamma(N)}=\sB_{g,N}(\C).
\end{displaymath}

We next recall the classical definition of Siegel--Jacobi forms by introducing a suitable automorphy factor for the group $G^{(g,1)}_{\R}$
(see, e.\,g., \cite{Ziegler-J}). For $k,m\in\Z$ and $M=\big(\begin{smallmatrix}A&B\\C&D\end{smallmatrix}\big)\in\Sp(2g,\R)$, $\zeta=[(\lambda,
\mu),x]\in H_{\R}^{(g,1)}$, we define for a complex-valued function $f\colon\H_{g}\times\C^{(1,g)}\rightarrow\C$ the slash-operators
\begin{displaymath}
\big(f\vert_{k,m}M\big)(Z,W)\coloneqq\det(CZ+D)^{-k}e^{-2\pi imW(CZ+D)^{-1}CW^{t}}f\big(M\langle Z\rangle,W(CZ+D)^{-1}\big)
\end{displaymath}
and 
\begin{displaymath}
\big(f\vert_{k,m}\zeta\big)(Z,W)\coloneqq e^{2\pi im(\lambda Z\lambda^{t}+2\lambda W^{t}+(x+\mu\lambda^{t}))}f(Z,W+\lambda Z+\mu). 
\end{displaymath}
A quadratic matrix $T$ is called \emph{half-integral}, if $2T$ has integral entries, while the diagonal entries of $T$ are even. Note that if 
$T$ is symmetric, then $T$ is half-integral if and only if the associated quadratic form is integral.

\begin{definition} 
\label{def:SJ_forms}
A meromorphic function $f\colon\H_{g}\times\C^{(1,g)}\rightarrow\C$ is called a \emph{meromorphic Siegel--Jacobi form of weight $k$ 
and index $m$ for the subgroup $\Gamma\subseteq\Sp(2g,\Z)$ of finite index}, if the following conditions are satisfied:
\begin{enumerate}
\item 
\label{item:35} 
$f\vert_{k,m}M=f$ for all $M\in\Gamma$.
\item 
\label{item:36} 
$f\vert_{k,m}\zeta=f$ for all $\zeta\in H_{\Z}^{(g,1)}$.
\end{enumerate}
A meromorphic Siegel--Jacobi form $f$ of weight $k$ and index $m$ for $\Gamma$ is called a \emph{Siegel--Jacobi form of weight $k$ 
and index $m$ for $\Gamma$}, if $f$ is holomorphic and
\begin{enumerate}[resume]  
\item 
\label{item:37} 
for each $M\in\Sp(2g,\Z)$, the function $f\vert_{k,m}M$ has a Fourier expansion of the form
\begin{displaymath}
\big(f\vert_{k,m}M\big)(Z,W)=\sum_{\substack{T=T^{t}\ge 0\\T\text{ half-integral}}}\sum_{R\in\Z^{(g,1)}}c(T,R)e^{2\pi i/n_{M}\tr(TZ)}e^{2\pi 
iWR}
\end{displaymath}
for some suitable positive integer $n_{M}$ depending only on $\Gamma$ and such that $c(T,R)\neq 0$ implies 
\begin{displaymath}
\begin{pmatrix}T/n_{M}&R/2\\R^{t}/2&m\end{pmatrix}\ge 0.
\end{displaymath}
\end{enumerate}
The $\C$-vector space of Siegel--Jacobi forms of weight $k$ and index $m$ for $\Gamma$ is denoted by $J_{k,m}(\Gamma)$.

A Siegel--Jacobi form $f\in J_{k,m}(\Gamma)$ is said to be a \emph{Siegel--Jacobi cusp form of weight $k$ and index~$m$ for $\Gamma$}, 
if $c(T,R)\neq 0$ implies
\begin{displaymath}
\begin{pmatrix}T/n_{M}&R/2\\R^{t}/2&m\end{pmatrix}>0.
\end{displaymath}
\end{definition}

We note that when $g>1$, if $f$ is holomorphic, condition~\ref{item:37} is a consequence of conditions~\ref{item:35} and~\ref{item:36} 
due to the Koecher principle (see~\cite[Lemma~1.6]{Ziegler-J}).

\begin{remark} 
Fix $N\ge 3$. The meromorphic Siegel--Jacobi forms of weight $k$ and index $m$ for $\Gamma (N)$ are the meromorphic sections of the 
line bundle $\sJ_{k,m,N}\otimes_{\Z[1/N,\zeta_{N}]}\C$ over $\sB_{g,N}(\C)$. For $g>1$, the Siegel--Jacobi forms of weight $k$ and index 
$m$ for $\Gamma(N)$ correspond to the global sections of the line bundle $\sJ_{k,m,N}\otimes_{\Z[1/N,\zeta_{N}]}\C$ over $\sB_{g,N}(\C)$, 
i.\,e., we have
\begin{displaymath}
J_{k,m}(\Gamma(N))=\Gamma(\sB_{g,N}(\C),\sJ_{k,m,N}\otimes_{\Z[1/N,\zeta_{N}]}\C).
\end{displaymath}
For $g=1$, Jacobi forms of weight $k$ and index $m$ for $\Gamma(N)$ correspond to the global sections of a certain extension of $\sJ_
{k,m,N}\otimes_{\Z[1/N,\zeta_{N}]}\C$ to a suitable compactification of $\sB_{g,N}(\C)$.
\end{remark}

\subsection{Toroidal compactifications}

We give a brief account of the theory of toroidal compactifications of the smooth, quasi-projective schemes $\sA_{g,N}$ and $\sB_{g,N}$
over $\Spec(\Z[1/N,\zeta_{N}])$ for $N\ge 3$. For more details we refer again to the book~\cite{FC} by G.~Faltings and C.-L.~Chai. 

Let $C_{g}$ be the open cone of real symmetric positive definite $(g\times g)$-matrices and $\overline{C}_{g}$ the cone of real symmetric 
positive semidefinite $(g\times g)$-matrices with rational kernel. Let $\widetilde{C}_{g}\subseteq\overline{C}_{g}\times\R^{(1,g)}$ be the 
cone defined by
\begin{displaymath}
\widetilde{C}_{g}=\big\{(Y,\beta)\in\overline{C}_{g}\times\R^{(1,g)}\,\big\vert\,\exists\,\alpha\in\R^{(1,g)}\colon\,\beta=\alpha Y
\big\}.
\end{displaymath}
We will denote by $\overline{C}_{g,\Z}$ the subset of half-integral matrices of $\overline{C}_{g}$. Furthermore, we set $\widetilde{C}_{g,
\Z}=\widetilde{C}_{g}\cap(\overline{C}_{g,\Z}\times\Z^{(1,g)})$.
 
Let $P_{g}\subset\Sp(2g,\Z)$ denote the parabolic subgroup consisting of the symplectic matrices $\big(\begin{smallmatrix}A&B\\0&D
\end{smallmatrix}\big)$. Then, there is a group homomorphism
\begin{displaymath}
P_{g}\longrightarrow\GL(g,\Z),
\end{displaymath}
given by the assignment $\big(\begin{smallmatrix}A&B\\0&D\end{smallmatrix}\big)\mapsto A$. We denote the image of the intersection 
$\Gamma(N)\cap P_{g}$ by this homomorphism in $\GL(g,\Z)$ by $\overline{\Gamma}(N)$. Similarly, we denote the image of the intersection 
$\widetilde{\Gamma}(N)\cap P_{g+1}$ in $\GL(g+1,\Z)$ by $\overline{\widetilde{\Gamma}}(N)$. Then, the group $\overline{\widetilde
{\Gamma}}(N)$ is given as the semi-direct product $\overline{\Gamma}(N)\ltimes\Z^{(1,g)}$ and is therefore contained in the subgroup 
of matrices of the form $\big(\begin{smallmatrix}A&0\\\lambda&1\end{smallmatrix}\big)$, where $A\in\GL(g,\Z)$ and $\lambda\in\Z^{(1,g)}$; 
in the sequel, we denote the elements of $\overline{\widetilde{\Gamma}}(N)$ by $(A,\lambda)$. The group $\overline{\Gamma}(N)$ acts on 
$\overline{C}_{g}$ by the rule
\begin{displaymath}
A\cdot Y=AY A^{t},
\end{displaymath}
and the group $\overline{\widetilde{\Gamma}}(N)$ acts on $\widetilde{C}_{g}$ by the rule
\begin{equation}
\label{eq:8}
(A,\lambda)\cdot(Y,\beta)=(AY A^{t},(\beta+\lambda Y)A^{t}). 
\end{equation}

\begin{definition}
\label{def:24}
An \emph{admissible cone decomposition of $\overline{C}_{g}$} is a set of cones $\Sigma$ in $\overline{C}_{g}$ such that the following
conditions are satisfied:
\begin{enumerate}
\item 
Each $\sigma\in\Sigma$ is generated by a finite set of elements of $\overline{C}_{g,\Z}$ and contains no lines. In other words, it is a 
rational polyhedral strictly convex cone.
\item 
If $\sigma$ belongs to $\Sigma$, each face of $\sigma$ belongs to $\Sigma$.
\item 
If $\sigma$ and $\sigma'$ belong to $\Sigma$, their intersection is a common face.
\item 
The union of all the cones of $\Sigma$ is $\overline{C}_{g}$.
\item 
The group $\overline{\Gamma}(N)$ leaves $\Sigma$ invariant with finitely many orbits.
\end{enumerate}
\end{definition}

\begin{definition}
\label{def:25}
Let $\Sigma$ be an admissible cone decomposition of $\overline{C}_{g}$. An \emph{admissible cone decomposition of $\widetilde{C}_
{g}$ over $\Sigma$} is a set of cones $\Pi$ in $\widetilde{C}_{g}$ such that the following conditions are satisfied:
\begin{enumerate}
\item 
Each $\tau\in\Pi $ is generated by a finite set of elements of $\widetilde{C}_{g,\Z}$ and contains no lines.
\item 
If $\tau$ belongs to $\Pi$, each face of $\tau$ belongs to $\Pi$.
\item 
If $\tau$ and $\tau'$ belong to $\Pi$, their intersection is a common face.
\item 
The union of all the cones of $\Pi$ is $\widetilde{C}_{g}$.
\item 
The group $\overline{\widetilde{\Gamma}}(N)$ leaves $\Pi$ invariant with finitely many orbits.
\item 
For each $\tau\in\Pi$, the projection of $\tau$ to $\overline{C}_{g}$ is contained in a cone $\sigma\in\Sigma$.
\end{enumerate}
\end{definition}

We say that $\Sigma$ and $\Pi$ are \emph{smooth}, if every cone of $\Sigma$ and every cone of $\Pi$ is generated by part of a $\Z
$-basis of the abelian groups generated by $\overline{C}_{g,\Z}$ and $\widetilde{C}_{g,\Z}$, respectively. We say that $\Pi$ is \emph
{equidimensional (over $\Sigma$)}, if for every cone $\tau$, the projection of $\tau$ to $\overline{C}_{g}$ is a cone $\sigma\in\Sigma$.

\begin{definition}
\label{def:26}
Let $\Sigma$ be an admissible cone decomposition of $\overline{C}_{g}$. An \emph{admissible divisorial function on $\Sigma$} is a 
continuous and $\overline{\Gamma}(N)$-invariant function $\phi\colon\overline{C}_{g}\rightarrow\R$ satisfying the following properties:
\begin{enumerate}
\item 
It is conical, in the sense that $\phi(tY)=t\phi(Y)$ for all $Y\in\overline{C}_{g}$ and $t\in\R_{\ge 0}$.
\item 
It is linear on each cone $\sigma\in\Sigma$.
\item
It takes integral values on $\overline{C}_{g,\Z}$.
\end{enumerate}
An admissible divisorial function $\phi$ on $\Sigma$ is called \emph{strictly anti-effective}, if, in addition, it satisfies
\begin{enumerate}[resume]
\item 
$\phi(Y)>0$ for all $Y\in\overline{C}_{g}\setminus\{0\}$.
\end{enumerate}
A strictly anti-effective admissible divisorial function $\phi$ on $\Sigma$ is called an \emph{admissible polarization function on $\Sigma$},
if the following two additional properties are satisfied:
\begin{enumerate}[resume]
\item 
The function $\phi$ is concave.
\item
The function $\phi$ is strictly concave on $\Sigma$, in the sense that, if $\sigma'$ is a cone in $\overline{C}_{g}$ such that the restriction 
of $\phi$ is linear on $\sigma'$, then $\sigma'$ is contained in a cone $\sigma\in\Sigma$; in other words, the maximal cones of $\Sigma$ 
are the maximal cones of linearity of $\phi$.  
\end{enumerate}
\end{definition}

\begin{remark}
For a given admissible cone decomposition $\Sigma$ of $\overline{C}_{g}$ it may happen that there are no admissible polarization functions 
on $\Sigma$. An admissible cone decomposition $\Sigma$ of $\overline{C}_{g}$ that admits an admissible polarization function is called 
\emph{projective}. As explained in~\cite[Chapter~V, \S~5]{FC}, every admissible cone decomposition $\Sigma$ of $\overline{C}_{g}$ admits 
a smooth projective refinement. 
\end{remark}

\begin{definition}
\label{def:27}
Let $\Sigma$ be an admissible cone decomposition of $\overline{C}_{g}$ and $\Pi$ an admissible cone decomposition of $\widetilde{C}_
{g}$ over $\Sigma$. An \emph{admissible divisorial function on $\Pi $} is a continuous and $\overline{\widetilde{\Gamma}}(N)$-invariant 
function $\phi\colon\widetilde{C}_{g}\rightarrow\R$ satisfying the following properties:
\begin{enumerate}
\item  
It is conical, in the sense that $\phi(t\widetilde{Y})=t\phi(\widetilde{Y})$ for all $\widetilde{Y}=(Y,\beta)\in\widetilde{C}_{g}$ and $t\in\R_{\ge 
0}$.
\item
It is linear on each cone $\tau\in\Pi$.
\item
It takes rational values on $\widetilde{C}_{g,\Z}$ with bounded denominators.
\item 
\label{item:31}
For $\lambda\in\Z^{(1,g)}$ and $\widetilde{Y}=(Y,\beta)\in\widetilde{C}_{g}$, the condition
\begin{displaymath}
\phi(\widetilde{Y})-\phi(\lambda\cdot\widetilde{Y})=\lambda Y\lambda^{t}+2\beta\lambda^{t}
\end{displaymath}
holds, where we recall from~\eqref{eq:8} that $\lambda\cdot\widetilde{Y}\coloneqq(\id_{g},\lambda)\cdot(Y,\beta)=(Y,\beta+\lambda Y)$.
\end{enumerate}
An admissible divisorial function $\phi$ on $\Pi$ is called an \emph{admissible polarization function on $\Pi$}, if the following two additional 
properties are satisfied:
\begin{enumerate}[resume]
\item 
The function $\phi $ is concave.
\item 
The function $\phi$ is strictly concave over each cone $\sigma\in\Sigma$, that is, for each maximal cone $\tau\in\Pi$ over $\sigma\in\Sigma$, 
there is a linear function $\varphi_{\tau}$ such that $\varphi_{\tau}(\widetilde{Y})=\phi(\widetilde{Y})$ for each $\widetilde{Y}=(Y,\beta)\in\tau$, 
but $\varphi_{\tau}(\widetilde{Y})>\phi(\widetilde{Y})$ for each $\widetilde{Y}=(Y,\beta)\notin\tau$, but with $Y\in\sigma$.
\end{enumerate}
\end{definition}

\begin{remark}
For a given admissible cone decomposition $\Pi$ of $\widetilde{C}_{g}$, even smooth, it may be possible that there are no admissible 
polarization functions on $\Pi$.  Nevertheless, there is always a refinement $\Pi'$ of $\Pi $ such that there exists an admissible polarization 
function on $\Pi'$. As in the case of the admissible cone decompositions $\Sigma$ of $\overline{C}_{g}$, an admissible cone decomposition
$\Pi$ of $\widetilde{C}_{g}$ that admits an admissible polarization function is called \emph{projective}.
\end{remark}

The theory of toroidal embeddings allows us to compactify the smooth, quasi-projective schemes $\sA_{g,N}$ and $\sB_{g,N}$ over $\Spec
(\Z[1/N,\zeta_{N}])$ for $N\ge 3$. The precise statements are as follows. 

\begin{theorem}
\label{thm:6}
Let $\Sigma$ be a projective admissible cone decomposition of $\overline{C}_{g}$ and $\Pi$ a projective admissible cone decomposition 
of $\widetilde{C}_{g}$ over $\Sigma$.
\begin{enumerate}
\item
 There is a projective scheme
$\overline{\sA}_{g,N,\Sigma}$ over $\Spec(\Z[1/N,\zeta_{N}])$ determined by the
cone decomposition $\Sigma$, that 
contains $\sA_{g,N}$ as an open dense subscheme.
\item 
There is a projective scheme $\overline{\sB}_{g,N,\Pi}$ over
$\Spec(\Z[1/N,\zeta_{N}])$ determined by the cone decomposition $\Pi$, that contains  
$\sB_{g,N}$ as an open dense subscheme and a projective morphism
\begin{displaymath}
\pi_{\Sigma,\Pi}\colon\overline{\sB}_{g,N,\Pi}\longrightarrow\overline{\sA}_{g,N,\Sigma}
\end{displaymath}
that extends the canonical projection $\pi_{N}\colon\sB_{g,N}\rightarrow\sA_{g,N}$.
\item 
If $\Sigma$ or $\Pi$ are smooth, then the corresponding schemes $\overline {\sA}_{g,N,\Sigma}$ or $\overline{\sB}_{g,N,\Pi}$ are smooth 
over $\Spec(\Z[1/N,\zeta_{N}])$, respectively. If $\Pi$ is equidimensional over $\Sigma$, then $\pi_{\Sigma,\Pi }$ is equidimensional. 
\item 
The projective scheme $\overline{\sA}_{g,N,\Sigma}$ admits a stratification by locally closed subschemes indexed by the $\overline{\Gamma}
(N)$-orbits of $\Sigma$, i.\,e., we have
\begin{displaymath}
\overline{\sA}_{g,N,\Sigma}=\bigcup_{\overline{\sigma}\in\Sigma/\overline{\Gamma}(N)}\overline{\sA}_{g,N,\overline{\sigma}}\,.
\end{displaymath}
The correspondence between cones $\overline{\sigma}\in\Sigma/\overline{\Gamma}(N)$ and strata $\overline{\sA}_{g,N,\overline{\sigma}}$
reverses dimensions; a stratum $\overline{\sA}_{g,N,\overline{\sigma}}$ lies in the closure of another stratum $\overline{\sA}_{g,N,\overline
{\sigma}'}$ if and only if there are representatives $\sigma$ and $\sigma'$ of $\overline{\sigma}$ and $\overline{\sigma}'$, respectively, such
that $\sigma'$ is a face of $\sigma$.
\item 
There is an analogous stratification of $\overline{\sB}_{g,N,\Pi}$ indexed by the $\overline{\widetilde\Gamma }(N)$-orbits of $\Pi$.
\end{enumerate}
\end{theorem}

\begin{remark}
\label{rem:loccor}
As in Theorem~\ref{thm:6}, we assume that $\Sigma$ and $\Pi$ are smooth projective admissible cone decompositions of $\overline{C}_{g}$
and $\widetilde{C}_{g}$, respectively, so that we have a projective morphism of smooth projective schemes $\pi_{\Sigma,\Pi}\colon\overline{\sB}_
{g,N,\Pi}\rightarrow\overline{\sA}_{g,N,\Sigma}$ over $\Spec(\Z[1/N,\zeta_{N}])$. A $\overline{\widetilde\Gamma}(N)$-orbit $\overline{\tau}$ now
determines a stratum $\overline{\sB}_{g,N,\overline{\tau}}$ of $\overline{\sB}_{g,N,\Pi}$. In the sequel, we will need local coordinates around a 
point $p_{\overline{\tau}}\in\overline{\sB}_{g,N,\overline{\tau}}(\C)$, which we recall from~\cite[Chapter~V, p.~141]{FC}; we note that these local
coordinates will depend on the choice of a representative $\tau$ of the orbit under consideration. Assume that $\dim(\tau)=n$ and recall that 
$\dim(\sB_{g,N}(\C))=d$. The chosen cone $\tau$ is then generated by the set $\{(Y_{1},\beta_{1}),\ldots,(Y_{n},\beta_{n})\}$, where $Y_{j}$ 
are half-integral symmetric positive semidefinite matrices and $\beta_{j}$ are integral row vectors of the form $\beta_{j}=\alpha_{j}Y_{j}$, which 
is part of an integral basis 
\begin{displaymath}
\{(Y_{1},\beta_{1}),\ldots,(Y_{n},\beta_{n}),(Y_{n+1},\beta_{n+1}),\ldots,(Y_{d},\beta_{d})\}
\end{displaymath}
of $\widetilde{C}_{g,\Z}$. Then, there exists a pair $(Y_{0},\beta_{0})$ with $Y_{0}$ a real symmetric positive definite $(g\times g)$-matrix and 
a row vector $\beta_{0}\in\R^{(1,g)}$ satisfying $\beta_{0}=\alpha_{0}Y_{0}$, and there are real numbers $0<r_{1},\ldots,r_{d}<1/e$ such that 
the set of pairs
\begin{displaymath}
V'\coloneqq\Bigg\{(Z,W)=i(Y_{0},\beta_{0})+\sum_{j=1}^{d}z_{j}(Y_{j},\beta_{j})\Bigg\vert\,z_{j}\in\C:\,
\begin{alignedat}{2}
2\pi\im(z_{j})&>-\log(r_{j}),\ &j&=1,\ldots,n \\
\vert z_{j}\vert&<r_{j},&j&=n+1,\ldots,d
\end{alignedat}
\Bigg\}
\end{displaymath}
constitutes an open subset of $\H_{g}\times\C^{(1,g)}$. We note that the set of pairs $\{(Y_{1},\beta_{1}),\ldots,(Y_{d},\beta_{d})\}$ can also be 
chosen to be an integral basis of any maximal cone $\tau'$ that has $\tau$ as a face. By choosing $V'$ small enough, the uniformization map 
$\H_{g}\times\C^{(1,g)}\rightarrow\sB_{g,N}(\C)$ maps the open subset $V'\subseteq\H_{g}\times\C^{(1,g)}$ to the coordinate neighbourhood 
$V\subseteq\overline{\sB}_{g,N,\Pi}(\C)$ centered at $p_{\overline{\tau}}$ by means of the assignment
\begin{equation}
\label{eq:31}
(z_{1},\ldots,z_{d})\mapsto(q_{1},\ldots,q_{d}),\text{ where }q_{j}=\Bigg\lbrace
\begin{alignedat}{2}
&e^{2\pi iz_{j}},\ &j&=1,\ldots,n, \\
&z_{j},&j&=n+1,\ldots,d,
\end{alignedat}
\end{equation}
so that $V'$ maps surjectively onto $V\cap\sB_{g,N}(\C)$. By construction, the coordinate neighbourhood $V$ is of the form $\Delta_{r_{1}}
\times\ldots\times\Delta _{r_{d}}$, where $\Delta_{r_{j}} \subseteq\C$ is the open disk of radius $r_{j}$ ($j=1,\ldots,d$) centered at the origin.
\end{remark}

\begin{remark}
\label{rem:14}
We end this subsection by mentioning that the projectivity assumption made in Theorem~\ref{thm:6} is used to show that formal versions 
of the moduli schemes under consideration are algebraizable.
\end{remark}

\subsection{The adelic arithmetic line bundle of Siegel--Jacobi forms}
\label{adJkmN}

The aim of this subsection is to upgrade the line bundle of Siegel--Jacobi forms $\sJ_{k,m,N}$ of weight $k$, index $m$, and level $N\ge 
3$ in a canonical way to an adelic arithmetic line bundle on $\sB_{g,N}$. However, as we show in Section~\ref{sec:non-integr-line}, this 
adelic arithmetic line bundle is not integrable, in general. Nevertheless, we will show in the subsequent subsections that this adelic arithmetic
line bundle has finite energy, so it will have a well-defined arithmetic self-intersection product that we will compute. 

We start by showing that the Hodge bundle $\omega_{N}$ of level $N$
can be upgraded to an adelic arithmetic line bundle on the fine 
moduli space $\sA_{g,N}$. Our presentation will follow the argument given in~\cite[\S~5.5]{YuanZhang:adelic}. We start from the moduli 
stack $\sA_{g}$ over $\Spec(\Z)$ and its Hodge bundle $\omega$. In analogy to Theorem~\ref{thm:6}, it is shown in~\cite[Chapter~IV, 
Theorem~5.7]{FC} that after choosing a smooth admissible cone
decomposition $\Sigma$ of $\overline{C}_{g}$, one obtains a smooth 
proper algebraic stack $\overline{\sA}_{g,\Sigma}$ over $\Spec(\Z)$, which contains $\sA_{g}$ as an open dense algebraic substack. 
The universal abelian scheme $\pi\colon\sB_{g}\rightarrow\sA_{g}$ can now be extended to a semi-abelian scheme $\overline{\pi}\colon
\sG_{g}\rightarrow\overline{\sA}_{g,\Sigma}$ with zero section $\overline{\varepsilon}\colon\overline{\sA}_{g,\Sigma}\rightarrow\sG_{g}$. 
The Hodge bundle $\omega$ on $\sA_{g}$ extends to a line bundle on $\overline{\sA}_{g,\Sigma}$, again denoted by $\omega$, defined 
by
\begin{displaymath}
\omega\coloneqq\overline{\varepsilon}^{\ast}\det(\Omega^{1}_{\sG_{g}/\overline{\sA}_{g,\Sigma}}).
\end{displaymath}
By~\cite[Chapter~V, Theorem~2.3]{FC}, after contracting the toroidal compactification $\overline{\sA}_{g,\Sigma}$ through the linear systems 
associated to powers of $\omega$, one obtains the minimal compactification $\sA_{g}^{\ast}$, which is a normal proper scheme of finite type 
over $\Spec(\Z)$ and provides a compactification of the coarse moduli space $\sA_{g}'$ of principally polarized abelian schemes of dimension 
$g$ over $\Spec(\Z)$; in particular, we have a natural morphism $\rho'\colon\sA_{g}\rightarrow\sA_{g}'$. We note that the construction of $\sA_
{g}^{\ast}$ turns out to be independent of the particular choice of toroidal compactification $\overline{\sA}_{g,\Sigma}$. By the very construction 
the line bundle $\omega$ on $\overline{A}_{g,\Sigma}$ descends as a $\Q$-line bundle on $\sA_{g}^{\ast}$ and, after restriction, gives rise to 
a $\Q$-line bundle on $\sA_{g}'$, which we denote again by $\omega$.

We will now equip the $\Q$-line bundle $\omega_{\otimes\Z}\C$ on $\sA_{g}'(\C)$ with a hermitian metric. For this, we first note that $\sA_
{g}'(\C)$ can be identified with the quotient space $\Sp(2g,\Z)\backslash\H_{g}$, so that the $\Sp(2g,\Z)$-orbit of a point $Z\in\H_{g}$ corresponds
to the isomorphism class of the abelian variety $A_{Z}=\C^{(1,g)}/(\Z^{(1,g)}Z\oplus\Z^{(1,g)})$. Given a holomorphic section $s$ of $\omega_
{\otimes\Z}\C$ over $\sA_{g}'(\C)$, the Faltings metric on $\omega_{\otimes\Z}\C$ 
is now defined by the formula
\begin{equation}
\label{Falnorm}
\Vert s(Z)\Vert_{\mathrm{Fal}}^{2}\coloneqq\frac{i^{g^{2}}}{2^{g}}\int_{A_{Z}}s(Z)\wedge \overline{s(Z)}.
\end{equation}

\begin{lemma}
\label{lem:omega'}
With the above notations, the hermitian line bundle $\overline{\omega}\coloneqq(\omega,\Vert\cdot\Vert_{\mathrm{Fal}})$ defines an adelic 
arithmetic line bundle on $\sA_{g}'$, i.\,e., $\overline{\omega}\in\bPichQ(\sA_{g}')^{\adel}$.
\end{lemma}
\begin{proof}
Let $\caB$ denote the boundary divisor of the projective arithmetic variety $\sA_{g}^{\ast}$ over $\Spec(\Z)$. Then, we obtain $\sA_{g}'=\sA_
{g}^{\ast}\setminus\vert\caB\vert$. As in Subsection~\ref{sec:glob-arithm-case}, we upgrade $\caB$ to an arithmetic boundary divisor $\overline
{\caB}=(\caB,g_{B})$ by means of a suitable Green function $g_{B}$ of continuous type for $B$. Denoting by $\sA_{g}^{\ast\ast}$ the blow-up of 
$\sA_{g}^{\ast}$ along $\caB$, the Faltings metric will have logarithmic singularities along the boundary $\sA^{\ast\ast}_{g}\setminus\sA'_{g}$.  

In order to complete the proof of the lemma, we additionally equip the $\Q$-line bundle $\omega$ on $\sA_{g}^{\ast}$ with a smooth hermitian
metric $\Vert\cdot\Vert'$ to obtain the hermitian line bundle $\overline{\omega}'$. We now choose a rational $\Q$-section $s_{\omega}$ of
$\omega$ to obtain the arithmetic divisor $\divh_{\overline{\omega}'}(s_{\omega})$, which we can view as an adelic arithmetic divisor on
$\sA_{g}'$. Consider next the function $f=-\log(\Vert s_{\omega}\Vert_{\mathrm{Fal}}^{2}\,/\,\Vert s_{\omega}\Vert'^{\,2})$ on $\sA_{g}'(\C)$. 
Since the hermitian metric $\Vert\cdot\Vert'$ extends smoothly to $\sA_{g}^{\ast}(\C)$, but the Faltings metric $\Vert\cdot\Vert_{\mathrm{Fal}}$ 
develops logarithmic singularities along the boundary $B$, we find that $f=o(g_{B})$. In this way, we find that
\begin{displaymath}
\divh_{\overline{\omega}}(s_{\omega})=\divh_{\overline{\omega}'}(s_{\omega})+(0,f).
\end{displaymath}
By Corollary~\ref{cor:4}, we obtain that $\divh_{\overline{\omega}}(s_{\omega})\in\DivhR(\sA_{g}')^{\adel}$, which implies that $\overline
{\omega}\in\bPichQ(\sA_{g}')^{\adel}$.
\end{proof}

We are now able to upgrade the Hodge bundle $\omega_{N}$ of level $N$ to an adelic arithmetic line bundle on $\sA_{g,N}$. For this, we fix 
a projective toroidal compactification $\overline{\sA}_{g,N,\Sigma}$ of $\sA_{g,N}$ over $\Spec(\Z[1/N,\zeta_{N}])$ as in Theorem~\ref{thm:6}. 
Of course, the scheme $\overline{\sA}_{g,N,\Sigma}$ is not projective over $\Spec(\Z)$. However, since $\overline{\sA}_{g,N,\Sigma}$ is 
projective over $\Spec(\Z[1/N])$, we can find a projective scheme $\caX'$ over $\Spec(\Z)$ that contains $\overline{\sA}_{g,N,\Sigma}$ as 
an open dense subscheme. By construction, we have that $\caX'\setminus\sA_{g,N}$ is the support of an effective divisor, namely the union 
of the fibers over the primes dividing $N$ together with the union of the closure of the components of the boundary divisor of the toroidal 
compactification under consideration. In this way, we can view $\sA_{g,N}$ as on open dense subset of the projective arithmetic variety 
$\caX'$ over $\Spec(\Z)$. 

\begin{lemma}
\label{lem:omegaN}
With the above notations, the hermitian line bundle $\overline{\omega}_{N}\coloneqq(\omega_{N},\Vert\cdot\Vert_{\mathrm{Fal}})$ defines an 
adelic arithmetic line bundle on $\sA_{g,N}$, i.\,e., $\overline{\omega}_{N}\in\bPichQ(\sA_{g,N})^{\adel}$.
\end{lemma}
\begin{proof}
By considering the pull-back of the line bundle $\omega$ on $\sA_{g}'$ by means of the composition of the two morphisms
\begin{equation}
\label{eq:2mor}
\sA_{g,N}\overset{\rho_{N}'}\longrightarrow\sA_{g}\overset{\rho'}\longrightarrow\sA_{g}',
\end{equation}
we obtain the Hodge bundle $\omega_{N}$ of level $N$ on $\sA_{g,N}$. Now, Remark~\ref{rem:44} shows that the adelic arithmetic line bundle 
$\overline{\omega}=(\omega,\Vert\cdot\Vert_{\mathrm{Fal}})$ on $\sA_{g}'$ pulls back to the adelic arithmetic line bundle $\overline{\omega}_
{N}=(\omega_{N},\Vert\cdot\Vert_{\mathrm{Fal}})$ on $\sA_{g,N}$, as claimed.
\end{proof}

Next, we show that the line bundle $\sL_{N}$ can be upgraded to an adelic arithmetic line bundle on the universal abelian scheme $\sB_{g,N}$. 
We introduce an analogous set-up as above. We fix a projective toroidal compactification $\overline{\sB}_{g,N,\Pi}$ of $\sB_{g,N}$ over $\Spec
(\Z[1/N,\zeta_{N}])$ as in Theorem~\ref{thm:6}. Of course, the scheme $\overline{\sB}_{g,N,\Pi}$ is not projective over $\Spec(\Z)$. However, 
since $\overline{\sB}_{g,N,\Pi}$ is projective over $\Spec(\Z[1/N])$, we can find a projective scheme $\caX$ over $\Spec(\Z)$ that contains 
$\overline{\sB}_{g,N,\Pi}$ as an open dense subscheme. By construction, we have that $\caX\setminus\sB_{g,N}$ is the support of an effective 
divisor, namely the union of the fibers over the primes dividing $N$ together with the union of the closure of the components of the boundary 
divisor of the toroidal compactification under consideration. In this way, we can view $\sB_{g,N}$ as on open dense subset of the projective 
arithmetic variety $\caX$ over $\Spec(\Z)$. 

\begin{lemma}
\label{lem:ellN}
With the above notations, the symmetric line bundle $\sL_{N}$ has a unique extension to an adelic arithmetic line bundle $\overline{\sL}_{N}
\coloneqq(\sL_{N},\Vert\cdot\Vert_{\mathrm{cub}})$ on $\sB_{g,N}$, i.\,e., $\overline{\sL}_{N}\in\bPichQ(\sB_{g,N})^{\adel}$, making the isomorphism
\eqref{eq:38} provided by the theorem of the cube into an isometry. Moreover, the adelic arithmetic line bundle $\overline{\sL}_{N}$ is integrable.
\end{lemma}
\begin{proof}
The claim follows immediately as an application of~\cite[Theorem~6.1.3]{YuanZhang:adelic} to the symmetric line bundle $\sL_{N}$ on the
universal abelian scheme $\pi_{N}\colon\sB_{g,N}\rightarrow\sA_{g,N}$.
\end{proof}

Lemmas~\ref{lem:omegaN} and~\ref{lem:ellN} now lead to the main result of this subsection.

\begin{proposition}
\label{prop:jacadel}
With the above notations, the line bundle $\sJ_{k,m,N}$ of Siegel--Jacobi forms of weight $k$, index $m$, and level $N$ has an extension to 
an adelic arithmetic line bundle $\overline{\sJ}_{k,m,N}$ on $\sB_{g,N}$, i.\,e., $\overline{\sJ}_{k,m,N}\in\bPichQ(\sB_{g,N})^{\adel}$.
\end{proposition}
\begin{proof}
As we have $\sJ_{k,m,N}=\pi_{N}^{\ast}\omega_{N}^{\otimes k}\otimes\sL_{N}^{\otimes m}$, the proof follows immediately by pulling back the
$k$-th power of the adelic arithmetic line bundle $\overline{\omega}_{N}^{\otimes k}=(\omega_{N}^{\otimes k},\Vert\cdot\Vert_{\mathrm{Fal}}^{k})$
of Lemma~\ref{lem:omegaN} from $\sA_{g,N}$ via $\pi_{N}$ to $\sB_{g,N}$ and then tensor this pull-back with the $m$-th power of the adelic 
arithmetic line bundle $\overline{\sL}_{N}^{\otimes m}=(\sL_{N}^{\otimes m},\Vert\cdot\Vert_{\mathrm{cub}}^{m})$ of Lemma~\ref{lem:ellN} to 
obtain $\overline{\sJ}_{k,m,N}=(\sJ_{k,m,N},\Vert\cdot\Vert_{\mathrm{Fal}}^{k}\cdot\Vert\cdot\Vert_{\mathrm{cub}}^{m})\in\bPichQ(\sB_{g,N})^
{\adel}$.
\end{proof}

\begin{remark}
\label{rem:natinv}
Proposition~\ref{prop:jacadel} shows that the complex line  bundle of Siegel--Jacobi forms $\sJ_{k,m,N}\otimes_{\Z[1/n,\zeta_{N}]}\C$ on $\sB_
{g,N}(\C)$ is equipped with the hermitian metric $\Vert\cdot\Vert_{\mathrm{Fal}}^{k}\cdot\Vert\cdot\Vert_{\mathrm{cub}}^{m}$. On the other, it is
well-known that it can also be equipped with the natural invariant metric $\Vert\cdot\Vert_{\mathrm{inv}}$, which is defined for a meromorphic
section $f$ by
\begin{displaymath}
\Vert f(Z,W)\Vert_{\mathrm{inv}}^{2}\coloneqq\vert f(Z,W)\vert^{2}\det(Y)^{k}e^{-4\pi m\beta Y^{-1}\beta^{t}},
\end{displaymath}
where $Z\in\H_{g}$, $Y=\im(Z)$ and $W\in\C^{(1,g)}$, $\beta=\im(W)$. We claim that
\begin{displaymath}
\Vert\cdot\Vert_{\mathrm{inv}}=\Vert\cdot\Vert_{\mathrm{Fal}}^{k}\cdot\Vert\cdot\Vert_{\mathrm{cub}}^{m}.
\end{displaymath}
In order to prove this claim, it suffices to consider the two cases $k=1,m=0$ and $k=0,m=1$. In the first case, the meromorphic section $f$
becomes a meromorphic modular form of weight $1$ for $\Gamma(N)$, which gives rise to the translation-invariant differential form $f(Z)\,\dd 
w_{1}\wedge\ldots\wedge\dd w_{g}$ on $\C^{(1,g)}$. Therefore, formula~\eqref{Falnorm} defining the Faltings metric shows
\begin{align*}
\Vert f(Z)\Vert_{\mathrm{Fal}}^{2}&=\frac{i^{g^{2}}}{2^{g}}\int_{A_{Z}}f(Z)\,\dd w_{1}\wedge\ldots\wedge\dd w_{g}\wedge\overline{f(Z)}\,\dd
\overline{w}_{1}\wedge\ldots\wedge\dd\overline{w}_{g} \\
&=\bigg(\frac{i}{2}\bigg)^{g}\int_{A_{Z}}\vert f(Z)\vert^{2}\,\dd w_{1}\wedge\dd\overline{w}_{1}\wedge\ldots\wedge\dd w_{g}\wedge\dd\overline
{w}_{g} \\[2mm]
&=\vert f(Z)\vert^{2}\,\det(Y)=\Vert f(Z)\Vert_{\mathrm{inv}}^{2}.
\end{align*}
In the second case, the meromorphic section $f$ becomes a meromorphic Siegel--Jacobi form of weight $0$ and index $1$ for $\Gamma(N)$, 
i.\,e., a meromorphic section of the line bundle $\sL_{N}$. Recalling that the isomorphism~\eqref{eq:38} becomes an isometry when $\sL_{N}$ 
is equipped with the hermitian metric $\Vert\cdot\Vert_{\mathrm{cub}}$, we now verify that this is also the case when $\sL_{N}$ is equipped with
the natural invariant metric $\Vert\cdot\Vert_{\mathrm{inv}}$. By the unicity of the hermitian metric $\Vert\cdot\Vert_{\mathrm{cub}}$, the second
claimed equality then also follows.

The section $[2]^{\ast}f$ corresponds to the function $f_{2}(Z,W)\coloneqq f(Z,2W)$, which is easily checked to be a meromorphic Siegel--Jacobi 
form of weight $0$ and index $4$ for $\Gamma(N)$, and thus corresponds to a meromorphic section of the line bundle $\sL_{N}^{\otimes 4}$. 
Therefore, the isomorphism~\eqref{eq:38} sends $[2]^{\ast}f$ to the section of $\sL_{N}^{\otimes 4}$ corresponding to the function $f_{2}(Z,W)$. 
Since $f(Z,W)^{4}$ also corresponds to a section of $\sL_{N}^{\otimes 4}$, there is a meromorphic function $h(Z,W)$ such that $f_{2}(Z,W)=h(Z,
W)\,f(Z,W)^{4}$. By construction, the function $h$ is $\widetilde{\Gamma}(N)$-invariant and thus descends to a meromorphic function on $\sB_
{g,N}(\C)$, again denoted by $h$. The isomorphism~\eqref{eq:38} now becomes an isometry if and only if we have
\begin{displaymath}
\Vert f_{2}(Z,W))\Vert_{\mathrm{inv}}=\vert h(Z,W)\vert\,\Vert f(Z,W)\Vert_{\mathrm{inv}}^{4}.
\end{displaymath}
Indeed, by the definition of the natural invariant metric, we compute
\begin{align*}
\Vert f_{2}(Z,W))\Vert_{\mathrm{inv}}&=\vert f(Z,2W)\vert e^{-2\pi(2\beta)Y^{-1}(2\beta)^{t}}=\vert h(Z,W)\vert\,\vert f(Z,W)^{4}\vert e^{-2\pi 4\beta 
Y^{-1}\beta^{t}} \\
&=\vert h(Z,W)\vert\big(\vert f(Z,W)\vert e^{-2\pi\beta Y^{-1}\beta^{t}}\big)^{4}=\vert h(Z,W)\vert\,\Vert (f(Z,W)\Vert_{\mathrm{inv}}^{4},
\end{align*}
which completes the proof of the second claimed equality.
\end{remark}

\begin{remark}
We note that the proof of Lemma~\ref{lem:omega'} (and thus also of Lemma~\ref{lem:omegaN}) and the proof of Lemma~\ref{lem:ellN} are of 
very different nature. The adelic arithmetic line bundle $\overline{\omega}=(\omega,\Vert\cdot\Vert_{\mathrm{Fal}})$ on $\sA'_{g}$ is constructed 
by means of the $\Q$-line bundle $\omega$ defined on the projective model $\sA_{g}^{\ast}$ with the logarithmically singular metric $\Vert\cdot
\Vert_{\mathrm{Fal}}$. However, as we will show in the next subsection, the adelic arithmetic line bundle $\overline{\omega}$ is not integrable
in the sense of~\cite{YuanZhang:adelic}, in general. Therefore, we cannot apply the techniques of that paper to define an arithmetic self-intersection
product for $\overline{\omega}$, although this can be done using the techniques of~\cite{BurgosKramerKuehn:accavb}. 

On the other hand, with regard to the adelic arithmetic line bundle $\overline{\sL}_{N}$, even the underlying geometric adelic line bundle is not a 
model line bundle, but it is a true adelic line bundle. Thus, the techniques of~\cite{BurgosKramerKuehn:accavb} cannot be applied in this case.  
Nevertheless, it is an integrable adelic arithmetic line bundle on $\sB_{g,N}$, therefore its arithmetic self-intersection product can be defined using 
the techniques developed in~\cite{YuanZhang:adelic}.

Thus, our objective in the remaining subsections is to use our extension of the formalism of~\cite{YuanZhang:adelic} developed in the previous 
sections to define the arithmetic self-intersection product for the adelic arithmetic line bundle $\overline{\sJ}_{k,m,N}$ of Siegel--Jacobi forms on 
$\sB_{g,N}$ and compute it.
\end{remark}

\subsection{Non-integrability of the Hodge line bundle}
\label{sec:non-integr-line}
In this subsection, we show that the adelic arithmetic line bundle $\overline{\omega}$ on $\sA'_{g}$ is not integrable, in general. To fix ideas and 
to make everything as explicit as possible, we consider the case $g=1$, and we will show that the adelic arithmetic $\Q$-line bundle $\overline
{\omega}$ on $\sA'_{1}$ is not integrable. To this end, we fix the affine coordinate $q$ on $\sA_{1}^{\ast}\cong\P^{1}$ that is given by $q=e^{2\pi 
iz}$ in a neighborhood of the cusp $\sA_{1}^{\ast}\setminus\sA'_{1}$, where $z$ is the standard coordinate on $\sA'_{1}$; in this coordinate, the 
cusp $\infty$ is described by $q=0$. Moreover, $q$ is also the affine coordinate of an integral model $\P^{1}_{\Z}$, and we know that $\omega^
{\otimes 12}\cong\caO(1)$ is a line bundle. Let now $s$ be a section of $\omega^{\otimes 12}$ over $\sA_{1,\Q}^{\ast}$ such that $\dv(s)=[\infty]$ 
on $\P^{1}_{\Q}$. By Remark~\ref{rem:natinv}, we have on $\P^{1}(\C)$ that 
\begin{displaymath}
-\frac{1}{12}\log\Vert s\Vert^{2}=-\log\Big(-\frac{1}{2\pi}\log\vert q\vert\Big)+\varphi,
\end{displaymath}
where $\varphi$ is bounded in a neighborhood of $q=0$. Put $g=(-1/12)\log\Vert s\Vert^{2}$. Then, $g$ is the Green function for the adelic 
arithmetic divisor $(1/12)\divh(s)$ corresponding to the adelic arithmetic $\Q$-line bundle $\overline{\omega}$. The function $g$ converges to 
$-\infty$, as $q$ tends to zero. Heuristically, the adelic arithmetic line bundle $\overline{\omega}$ cannot be integrable because the singularity
of the function $g$ at $q=0$ tells us that the height of the cusp $\infty$ ought to be $-\infty$, but this cannot be achieved by a difference of two
nef adelic arithmetic line bundles since the height of a point with respect to a nef adelic arithmetic line bundle is always finite and positive.  

To prove that the adelic arithmetic line bundle $\overline{\omega}$ is not integrable, we will use freely the theory of Berkovich spaces and, in 
particular, that theory applied to the projective line $\P^{1}$. More details about the relationship between adelic arithmetic divisors and Berkovich
spaces can be found in Section~3 of~\cite{YuanZhang:adelic}. For an extensive study of potential theory on the Berkovich projective line the
reader is referred to~\cite{bakerrumely}.

Consequently, we next discuss adelic arithmetic divisors on the projective line $\P^{1}$. For any divisor $D$ of $\P^{1}_{\Q}$, let $\overline{D}$ 
be the divisor of $\P^{1}_{\Z}$ obtained by Zariski closure. For every finite place $v$, the model $\overline{D}$ of $D$ defines a non-archimedean 
Green function $g_{D,v}^{\can}$ on the Berkovich projective line $\P^{1,\an}_{\C_{v}}$ that is given as follows. Let $\zeta_{v}$ be the Gauss point 
of $\P^{1,\an}_{\C_{v}}$ determined by the previous choice of coordinates. Then, the function $g_{D,v}^{\can}$ is characterized by being the unique 
Green function for $D$ of continuous type, which is harmonic on $\P^{1,\an}_{\C_{v}}\setminus\vert D\vert\cup\{\zeta_{v}\}$ and which vanishes
at the point $\zeta_{v}$.

To give a more concrete description, recall from~\cite{bakerrumely} that $\P^{1,\an}_{\C_{v}}$ has the structure of an $\R$-tree, where we can 
take the Gauss point $\zeta_{v}$ as the root and the leaves are the points of type I (that can be identified with $\P^{1}(\C_{v})$) and of type IV. 
For each point $x\in\P^{1,\an}_{\C_{v}}$, a tangent vector at the point $x$ is defined to be a connected component of $\P^{1,\an}_{\C_{v}}\setminus 
\{x\}$. Points of types I and IV have only one tangent vector (as they are terminal leaves), points of type III have two tangent vectors (as they are
non-branching inner points), and points of type II have more than two tangent vectors as they are branching points. If we write
\begin{displaymath}
D=\sum_{p\in\P^{1}(\overline{\Q})}a_{p}[p],
\end{displaymath}
then $g_{D,v}^{\can}$ is the unique piecewise linear function on $\P^{1,\an}_{\C_{v}}$ whose value at $\zeta_{v}$ is zero and whose slope at 
a tangent vector $t_{x}\in T_{x}$ for $x\in\P^{1,\an}_{\C_{v}}$ is equal to 
\begin{align*}
\begin{cases}
\phantom{-}\sum_{p\in t_{x}}a_{p},&\text{ if }\zeta_{v}\not\in t_{x}, \\
-\sum_{p\not\in t_{x}}a_{p}=-\deg(D)+\sum_{p\in t_{x}}a_{p},&\text{ if }\zeta_{v}\in t_{x}.
\end{cases}
\end{align*}

Let $\overline{\caB}=(\caB,g_{B})$ be a boundary divisor on $\P^{1}_{\Z}$. Then, $B=\caB_{\Q}$ is an effective divisor on $\P^{1}_{\Q}$. Let
$D$ be an $\R$-divisor on $\P^{1}_{\Q}$ and let $\overline{\caD}=(\caD,g_{D})$ be an adelic arithmetic divisor on $\P^{1}_{\Z}$ with $D=
\caD_{\Q}$. For every finite place $v$ of $\Q$, let $g_{D,v}$ be the non-archimedean Green function for $D$ on the Berkovich projective
line $\P^{1,\an}_{\C_{v}}$ corresponding to the adelic arithmetic divisor $\overline{\caD}$; we write $g_{D,\infty}=g_{D}$ for the (archimedean)
Green function for $D$. Let $(\overline{\caD}_{n})_{n\in\N}$ be a Cauchy sequence of model arithmetic divisors $(\caD_{n},g_{D_{n},\infty})$
converging in the $\overline{\caB}$-adic topology to $\overline{\caD}$. After restricting to a suitable subsequence, we can assume that the 
inequalities
\begin{equation}
\label{eq:6}
-\overline{\caB}\le\overline{\caD}_{n}-\overline{\caD}_{0}\le\overline{\caB}
\end{equation}
hold for all $n\in\N$. Since $\P^{1}_{\Q}$ is $1$-dimensional, all regular birational models of it are isomorphic to itself. Thus, each divisor
$D_{n}=\caD_{n,\Q}$ is a divisor on $\P^{1}_{\Q}$, and the sequence $(D_{n})_{n\in\N}$ converges to $D$ in the $B$-adic topology.

Finally, let  $g'_{D,\infty}$ be a Green function of continuous type for $D$ on $\P^{1}(\C)$. Then, choose for each $n\in\N$, a Green function 
$g'_{D_{n},\infty}$ of continuous type for $D_{n}$ on $\P^{1}(\C)$ so that the sequence $(\caD_{n},g'_{D_{n},\infty})_{n\in\N}$ converges in 
the $\overline{\caB}$-adic topology to $(\overline{\caD},g'_{D,\infty})$.

\begin{proposition}
\label{prop:15} 
With the above notations, assume that the divisor $D$ is (geometrically) nef, that each divisor $\caD_{n}$ is relatively nef, and that $g_{D_
{n},\infty}$ is a Green function of continuous and subharmonic type for $D_{n}$. Then, the following statements hold:
\begin{enumerate}
\item
\label{item:15}
The function $g_{D,\infty}-g'_{D,\infty}$ is bounded from above.
\item
\label{item:38}
For every finite place $v$ of $\Q$, the function $g_{D,v}-g^{\can}_{D,v}$ is bounded from above. In fact, the sequence $(g_{D_{n},v}-g^{\can}_
{D_{n},v})_{n\in\N}$  is uniformly bounded from above.
\item
\label{item:39} 
There exists a finite set $S$ of finite places of $\Q$ such that for each finite place $v\not\in S$, the inequality $g_{D,v}-g^{\can}_{D,v}\le 0$ 
holds. In fact, for each $n\in\N$, the inequality $g_{D_{n},v}-g^{\can}_{D_{n},v}\le 0$ holds. 
\end{enumerate}
\end{proposition}
\begin{proof}
\ref{item:15} The sequence $(g_{D_{n},\infty}-g'_{D_{n},\infty})_{n\in\N}$ is a sequence of continuous functions that converge uniformly to
$g_{D,\infty}-g'_{D,\infty}$ on compacta of $\P^{1}(\C)\setminus\vert B\vert$. Hence, the function $g_{D,\infty}-g'_{D,\infty}$ is continuous on 
$\P^{1}(\C)\setminus\vert B\vert$. Thus, we are left to prove that the function $g_{D,\infty}-g'_{D,\infty}$ is also bounded from above in a 
neighborhood of any point of $\vert B\vert$. So, let $x\in\vert B\vert$, let $a$ be the order of $D$ at $x$, and let $f$ be a local equation for 
the point $x$. Since $g'_{D,\infty}$ is a Green function of continuous type for $D$, the function $g'_{D,\infty}+a\log\vert f\vert$ is continuous 
in a neighborhood of $x$. Therefore, we need to prove that $g_{D,\infty}+a\log\vert f\vert$ is bounded from above in a neighborhood of $x$. 
Let $a_{n}$ denote the order of $D_{n}$ at $x$. The functions $g_{D_{n},\infty}+a_{n}\log\vert f\vert$ are continuous and subharmonic in a 
neighborhood of $x$. Let $\Delta $ be a small disk centered in $x$. Then, the functions $g_{D_{n},\infty}+a_{n}\log\vert f\vert$ converge 
uniformly on $\partial\Delta$ to the function $g_{D,\infty}+a\log\vert f\vert$. Therefore, the sequence $g_{D_{n},\infty}+a_{n}\log\vert f\vert$ 
is uniformly bounded from above on $\partial\Delta$. The maximum principle for subharmonic functions now implies that the sequence $g_
{D_{n},\infty}+a_{n}\log\vert f\vert$ has also to be uniformly bounded from above in $\Delta$, which proves that $g_{D,\infty}+a\log\vert f\vert$ 
is bounded from above in $\Delta$. This completes the proof of the first part of the proposition.

\ref{item:38} This part is proven in an way that is analogue to the preceeding part using the non-archimedean maximum principle (see~\cite
[Proposition 8.14]{bakerrumely}).

\ref{item:39} Let $\caX$ be a model over $\Spec(\Z)$ of $\P^{1}_{\Q}$, on which $\mathcal{D}_{0}$ is defined. Then, there is an integer $N
>0$ such that the isomorphism $\caX_{\Q}\cong\P^{1}_{\Q}$ extends to an isomorphism $\caX_{\Spec(\Z[1/N])}\cong\P^{1}_{\Spec(\Z[1/N])}$ 
over $\Spec(\Z[1/N])$. Let $S_{1}$ be the set of finite places of $\Q$ dividing $N$, let $S_{2}$ be the set of finite places of $\Q$ over which 
$\caD_{0}$ has a vertical component, and let $S_{3}$ be the set of finite places of $\Q$ over which $\caB$ has vertical components; put $S=
S_{1}\cup S_{2}\cup S_{3}$. Let $v\not\in S$ be a finite place of $\Q$ and let $F$ be the fiber of $\caX$ over $v$, which is an irreducible 
Cartier divisor since $v\not\in S_{1}$. Since $v\not\in S_{2}$, the order of $\caD_{0}$ on $F$ is zero. Since $v\not\in S_{3}$, the order of  
$\caB$ on $F$ is also zero. Then, condition~\eqref{eq:6} implies that the order of $\caD_{n}$ on $F$ has to be zero for all $n\in\N$, from
which we derive that $g_{D_{n},v}(\zeta_{v})=0$, because the Gauss point $\zeta_{v}$ of $\P^{1,\an}_{\C_{\nu}}$ is the divisorial point 
corresponding to the divisor $F$. Hence, we obtain $g_{D_{n},v}(\zeta _{v})=g^{\can}_{D_{n},v}(\zeta_{v})$. Since both functions are 
non-archimedean Green functions for the divisor $D_{n}$, they have the same slope at the unique tangent vector of any point in $\vert 
D_{n}\vert$. Since $g^{\can}_{D_{n},v}(\zeta_{v})$ is harmonic on $\P^{1,\an}_{\C_{v}}\setminus\vert D\vert\cup\{\zeta_{v}\}$, while $g_
{D_{n},v}(\zeta_{v})$ is subharmonic in the same dense open subset, we claim that
\begin{equation}
\label{eq:20}
g_{D_{n},v}\le g^{\can}_{D_{n},v}.
\end{equation}
Indeed, let $f=g_{D_{n},v}-g^{\can}_{D_{n},v}$. It is subharmonic on $\P^{1,\an}_{\C_{v}}\setminus\vert D\vert\cup\{\zeta_{v}\}$. Since
$g_{D_{n},v}$ and $g^{\can}_{D_{n},v}$ have the same slope at tangent vectors of points of $\vert D\vert$, the function $f$ extends to 
a continuous function on the whole of $\P^{1,\an}_{\C_{v}}$ that is constant on a neighborhood of  each point of $\vert D\vert$. Therefore,
$f$ is subharmonic on $\P^{1,\an}_{\C_{v}}\setminus\{\zeta_{v}\}$. By the maximum principle of subharmonic functions, the maximum of 
$f$ is attained at the boundary of  $\P^{1,\an}_{\C_{v}}\setminus\{\zeta_{v}\}$. That is, at the point $\zeta_{v}$. Hence, we find $f(x)\le 
f(\zeta_{v})=0$, and thereby obtain the claimed inequality~\eqref{eq:20}. Since the sequence $(g_{D_{n},v})_{n\in\N}$ converges to 
$g_{D,v}$ and the sequence $(g^{\can}_{D_{n},v})_{n\in\N}$ converges to $g^{\can}_{D,v}$, we deduce $g_{D,v}\le g^{\can}_{D,v}$, 
as to be proven. This completes the proof of the proposition.
\end{proof}
We also need the following lemma. 
\begin{lemma}
\label{lemm:15}
Let $\Delta\subset\C$ be a closed disk centered at the origin and let $(f_{n})_{n\in\N}$ be a sequence of continuous, subharmonic functions 
on $\Delta$ that converge uniformly on compacta of $\Delta\setminus\{0\}$ to a continuous function $f$. Assume that there is an increasing 
function $\varphi\colon\R_{>0}\rightarrow\R$ with $\lim_{t\to 0}\varphi(t)=-\infty$ such that $f(z)\le\varphi(\vert z\vert)$ for all $z\in\Delta
\setminus\{0\}$. Then, we have
\begin{displaymath}
\liminf_{n\to\infty}f_{n}(0)=-\infty.
\end{displaymath}
\end{lemma}
\begin{proof}  
Assume that $\liminf_{n\to\infty}f_{n}(0)>-\infty$. Then, there is a constant $K$ such that $f_{n}(0)\ge K$ for all $n\in\N$. On the other hand,
there is an $\varepsilon>0$ such that $f(z)\le\varphi(\varepsilon)<K$ for all $z$ with $\vert z\vert=\varepsilon$. Let $S_{\varepsilon }$ be 
the circle of radius $\varepsilon$. Since $f_{n}$ converges uniformly to $f$ on $S_{\varepsilon}$, there exists $n_{0}\in\N$ such that $f_
{n_{0}}(z)<K$ for all $z\in S_{\varepsilon}$. Since $f_{n_{0}}$ is subharmonic, we conclude by the maximum principle that $f_{n_{0}}(0)<K$,
thereby obtaining a contradiction. Hence, we find that $\liminf_{n\to\infty}f_{n}(0)=-\infty$.
\end{proof}

\begin{theorem}
\label{thm:12} 
The adelic arithmetic line bundle $\overline{\omega}$ is not integrable on any open subset $\caU$ of $\mathscr{A}'_{1}$.
\end{theorem}
\begin{proof}
Assume that $\overline{\omega}$ is integrable. Then, we have that
\begin{displaymath}
\frac{1}{12}\divh(s)=\overline{\caD}-\overline{\caE},
\end{displaymath}
where $\overline{\caD}=(\caD,g_{D,\infty})$ and $\overline{\caE}=(\caE,g_{E,\infty})$ are nef adelic arithmetic divisors on $\caU$ with 
underlying geometric divisors $D$ and $E$, respectively. After adding a divisor of the form $\divh(f)$ to both $\overline{\caD}$ and 
$\overline{\caE}$, we can assume that the cusp $\infty$ (given by $q=0$) is not in the support of $D$ nor $E$. By Proposition~\ref
{prop:15}~\ref{item:15}, the Green function $g_{E,\infty}$ is bounded from above in a neighborhood of the cusp $\infty$. By the nature 
of the singularities of the metric of $\overline{\omega}$, we deduce from this that there is a constant $K$ such that, in a neighborhood 
of the cusp $\infty$, the estimate 
\begin{equation}
\label{eq:11}
g_{D,\infty}\le-\log\Big(-\frac{1}{2\pi}\log\vert q\vert\Big)+K
\end{equation}
holds. Let $(\overline{\caD}_{n})_{n\in\N}$ be a Cauchy sequence of nef model arithmetic divisors converging to $\overline{\caD}$. Since 
the cusp $\infty$ is contained in the boundary divisor, it may happen that the cusp $\infty$ is in the support of some of the divisors $D_{n}$. 
After enlarging the boundary divisor $\overline{\caB}$, if needed, and adding a rational divisor, we can also assume that the cusp $\infty$
is not in the support of any of the $D_{n}$'s. We next apply the construction preceding Proposition~\ref{prop:15} to $\overline{\caD}$ and 
the sequence $(\overline{\caD}_{n})_{n\in\N}$; furthermore, we put 
\begin{displaymath}
\overline{\caD}'=(\overline{\caD},g'_{D,\infty})\qquad\text{and}\qquad\overline{\caD}'_{n}=(\overline{\caD}_{n},g'_{D_{n},\infty}).
\end{displaymath}
In this way, we obtain (recall that the cusp $\infty$ is given by $q=0$)
\begin{align*}
0&\le\height_{\overline{\caD}_{n}}(0) \\
&=\height_{\overline{\caD}'_{n}}(0)+g_{D_{n},\infty}(0)-g'_{D_{n},\infty}(0)+\sum_{v\in S}\big(g_{D_{n},v}(0)-g^{\can}_{D_{n},v}(0)\big)+
\sum_{v\not\in S}\big(g_{D_{n},v}(0)-g^{\can}_{D_{n},v}(0)\big),
\end{align*}
where $S$ is the finite set of finite places of $\Q$ introduced in Proposition~\ref{prop:15}. However, we will show below that the preceding 
inequality is violated since the lower limit of the right-hand side as $n$ tends to infinity is $-\infty$. Indeed, first, the sequence $(\height_
{\overline{\caD}'_{n}}(0))_{n\in\N}$ is bounded since it converges to $\height_{\overline{\caD}'}(0)$. Second, the sequence $(g'_{D_{n},
\infty}(0))_{n\in\N}$ is also bounded since it converges to $g'_{D,\infty}(0)$. Third, the sum $\sum_{v\in S}(g_{D_{n},v}(0)-g^{\can}_{D_{n},
v}(0))$ is bounded from above since it is a finite sum of terms uniformly bounded from above by Proposition~\ref{prop:15}~\ref{item:38}. 
Fourth, the sum $\sum_{v\not\in S}(g_{D_{n},v}(0)-g^{\can}_{D_{n},v}(0))$ is bounded from above since all the summands are negative by 
Proposition~\ref{prop:15}~\ref{item:39}. Note that, for fixed $n\in\N$, the number of non-zero terms in the latter sum is finite. Finally, by
equation~\eqref{eq:11}, the functions $g_{D,\infty}$ and $g_{D_{n},\infty}$ satisfy the hypothesis of Lemma~\ref{lemm:15}. Therefore, 
we get $\liminf_{n\to\infty}g_{D_{n},\infty}(0)=-\infty$ giving the desired contradiction and thereby proving the result.
\end{proof}

% In view of the Proposition~\ref{prop:15} and the proof of Theorem~\ref{thm:12} one may ask the following question.
%
% \begin{question}
% Let  $\caU$ be a quasi-projective arithmetic variety (of arbitrary dimension) over $\Spec(\Z)$ and $\overline{\caD}_{1}$ and $\overline
% {\caD}_{2}$ be two integrable adelic arithmetic divisors on $\caU$ that have the same underlying adelic geometric divisor $D$. For any 
% place $v$ of $\Q$, let $g_{1,v}$ and $g_{2,v}$ be the associated Green functions. Is it true that the difference $g_{1,v}-g_{2,v}$ is always 
% bounded?
% \end{question}

\subsection{Preparatory lemmas}
\label{sec:norm-like-functions}

For the proof of the finiteness of the arithmetic self-intersection
product for the adelic arithmetic line bundle $\overline{\sJ}_{k,m,N}$
on $\sB_{g, 
N}$, we need to study and estimate the quantity
\begin{displaymath}
\beta Y^{-1}\beta^{t}=\im(W)\im(Z)^{-1}\im(W)^{t}
\end{displaymath}
and certain derivatives thereof emerging from the definition of the natural invariant metric introduced in Remark~\ref{rem:natinv}, as $(Z,W)$ ranges 
through the open subset $V'\subseteq\H_{g}\times\C^{(1,g)}$ introduced in Remark~\ref{rem:loccor}. Recalling from Remark~\ref{rem:loccor} that
\begin{displaymath}
(Z,W)=i(Y_{0},\beta_{0})+\sum_{j=1}^{d}z_{j}(Y_{j},\beta_{j})
\end{displaymath}
for suitable $z_{j}=x_{j}+iy_{j}\in\C$ ($j=1,\ldots,d$), we find that
\begin{displaymath}
Y=Y(y_{1},\ldots,y_{d})=Y_{0}+\sum_{j=1}^{d}y_{j}Y_{j}\qquad\text{and}\qquad\beta=\beta(y_{1},\ldots,y_{d})=\beta_{0}+\sum_{j=1}^{d}y_{j}\beta_
{j}.
\end{displaymath}
Thus, we will have to investigate the function $\varphi\colon\R^{d}_{\ge 0}\rightarrow\R$ defined by
\begin{displaymath}
\varphi(y_{1},\ldots,y_{d})=\beta Y^{-1}\beta^{t}
\end{displaymath}
with $Y=Y(y_{1},\ldots,y_{d})$ and $\beta=\beta(y_{1},\ldots,y_{d})$ as above.

The first task is to compute the Hessian matrix of the function $\varphi$. To this end, we introduce for $j=0,\ldots,d$ the quantities $W_{j}=\beta_
{j}-\beta Y^{-1}Y_{j}$; they satisfy the linear relation
\begin{equation}
\label{eq:32}
W_{0}+\sum_{j=1}^{d}y_{j}W_{j}=0,
\end{equation}
which follows by a direct computation.

\begin{lemma}
\label{lemm:12}
With the above notations, the equality
\begin{displaymath}
\frac{\partial^{2}\varphi}{\partial y_{i}\partial y_{j}}=2W_{i}Y^{-1}W_{j}^{t} 
\end{displaymath}
holds for $i,j=1,\ldots,d$. Therefore, the Hessian matrix of $\varphi$ is a Gram matrix and, in particular, the rank of the Hessian of $\varphi$ is at 
most $g$.
\end{lemma}
\begin{proof}
This is a simple computation. Indeed, for the first partial derivative of $\varphi$, we compute
\begin{displaymath}
\frac{\partial\varphi}{\partial y_{i}}=2\beta_{i}Y^{-1}\beta^{t}-\beta Y^{-1}Y_{i}Y^{-1}\beta^{t}.
\end{displaymath}
Therefore, the second partial derivative of $\varphi$ becomes
\begin{align*}
\frac{\partial^{2}\varphi}{\partial y_{i}\partial y_{j}}&=2\beta_{i}Y^{-1}\beta_{j}^{t}-2\beta_{i}Y^{-1}Y_{j}Y^{-1}\beta^{t}-2\beta Y^{-1}Y_{i}Y^{-1}\beta_
{j}^{t}+2\beta Y^{-1}Y_{i}Y^{-1}Y_{j}Y^{-1}\beta^{t} \\
&=2(\beta_{i}-\beta Y^{-1}Y_{i})Y^{-1}(\beta_{j}-\beta Y^{-1}Y_{j})^{t} \\[2mm]
&=2W_{i}Y^{-1}W_{j}^{t}.
\end{align*}
This completes the proof of the lemma.
\end{proof}

In order to ease the bound of the entries of the Hessian matrix above, we introduce the following condition.

\begin{definition}
\label{def:flco}
We say that \emph{the matrices $Y_{0},\ldots,Y_{d}$ satisfy the flag condition}, if $\ker(Y_{j})\subseteq\ker(Y_{j+k})$ for $j=0,\ldots,d-1$ and $k=0,
\ldots,d-j$.
\end{definition}

From now on, we assume that the matrices $Y_{0},\ldots,Y_{d}$ satisfy the flag condition. We write $\rk_{j}=\rk(Y_{j})$ and choose a basis such 
that each matrix $Y_{j}$ can be written in block form as
\begin{displaymath}
Y_{j}=\begin{pmatrix}Y'_{j}&0\\0&0\end{pmatrix},
\end{displaymath}
where $Y'_{j}$ is a real symmetric positive definite $(\rk_{j}\times\rk_{j})$-matrix for $j=0,\ldots,d$. Note that, by hypothesis, $Y'_{0}=Y_{0}$ and
thus $\rk_{0}=g$. The following result is proven in~\cite[Lemma~3.6]{bghdj_sing}.

\begin{lemma}
\label{lemm:13} 
With the above notations, there is a constant $C>0$ such that the upper bound
\begin{displaymath}
(Y^{-1})_{k,\ell}\le\frac{C}{\displaystyle 1+\sum_{j\colon\rk_{j} \ge\min(k,\ell)}y_{j}}
\end{displaymath}
holds for $k,\ell=1,\ldots,g$.
\end{lemma}

Note that the number $1$ in the denominator above comes from the fact that $Y_{0}$ has maximal rank being positive definite and that $Y_{0}$
occurs with the coefficient $1$ in the definition of $Y=Y(y_{1},\ldots,y_{d})$. Next, we will bound the entries of $W_{j}$.

\begin{lemma}
\label{lemm:14} 
With the above notations, we have the following two results for the row vectors $W_{j}$ for $j=0,\ldots,d$: If $g\ge\ell>\rk_{j}$, the relation
\begin{displaymath}
(W_j)_{\ell}=0
\end{displaymath}
holds. Moreover, if $0\le\ell\le\rk_{j}$, there is a constant $C>0$ such that the upper bound
\begin{displaymath}
(W_{j})_{\ell}\le C
\end{displaymath}
holds. 
\end{lemma}
\begin{proof}
The first statement follows immediately from the choice of basis that puts each $Y_{j}$ in the desired block form. In order to prove the second claim, 
observing that $W_{j}=\beta_{j}-\beta Y^{-1}Y_{j}$, it is enough to show that the entries of $\beta Y^{-1}$ are uniformly bounded. In this regard, by 
the definition of $\beta=\beta(y_{1},\ldots,y_{d})$, there is a constant $C_{1}>0$ such that the $k$-th entry of $\beta$ can be bounded as
\begin{displaymath}
(\beta)_{k}\le C_{1}\bigg(1+\sum_{j\colon\rk_{j}\ge k}y_{j}\bigg).
\end{displaymath}
By means of Lemma~\ref{lemm:13} (with the constant $C$ renamed as $C_{2}$), we then derive the bound
\begin{displaymath}
(\beta Y^{-1})_{\ell}=\sum_{k=1}^{g}(\beta)_{k}(Y^{-1})_{k,\ell}\le\sum_{k=1}^{g}C_{1}C_{2}\frac{1+\displaystyle\sum_{j\colon\rk_{j}\ge k}y_{j}}{1+
\displaystyle\sum_{j\colon\rk_{j}\ge\min(k,\ell)}y_{j}}\le g\,C_{1}C_{2}.
\end{displaymath}
This completes the proof of the second claim.
\end{proof}

As a consequence of the above two lemmas we deduce the following bound for the elements of the Hessian matrix.

\begin{lemma}
\label{lemm:17}
With the above notations, there are constants $C,C'>0$ such that the upper bounds
\begin{displaymath}
W_{i}Y^{-1}W_{j}^{t}\le\frac{C}{1+\displaystyle\sum_{k\colon\rk_{k}\ge\min(\rk_{i},\rk_{j})}y_{k}}\le\frac{C'}{1+y_{i}^{1/2}y_{j}^{1/2}}
\end{displaymath}
hold.
\end{lemma}
\begin{proof}
As we have
\begin{displaymath}
W_{i}Y^{-1}W_{j}^{t}=\sum_{k,\ell=1}^{g}(W_{i})_{k}(Y^{-1})_{k,\ell}(W_{j})_{\ell},
\end{displaymath}
the first claimed inequality is an immediate consequence of Lemma~\ref{lemm:13} and Lemma~\ref{lemm:14}. The second claimed inequality can 
then be easily derived.
\end{proof}

For the sequel, we will need finer estimates. For this, we let $I\subseteq\{1,\ldots,d\}$ denote a subset and we write $W_{I}=\sum_{j\in I}y_{j}W_
{j}$. To uniformize our notation, we set $y_{0}=1$ and, given $I\subseteq\{1,\ldots,d\}$, we put $I_{0}^{\complement}=\{0,\ldots,d\}\setminus I$. 
For later purposes, we put $I^{\complement}=\{1,\ldots,d\}\setminus I$ and note that $I_{0}^{\complement}=I^{\complement}\cup\{0\}$.

\begin{lemma}
\label{lemm:18}
With the above notations, there is a constant $C>0$ such that for all subsets $I\subseteq\{1,\ldots,d\}$ the  upper bound
\begin{displaymath}
W_{I}Y^{-1}W_{I}^{t}\le C\frac{\sum_{i\in I}y_{i}\sum_{j\in I_{0}^{\complement}}y_{j}}{\sum_{k=0}^{d}y_{k}}
\end{displaymath}
holds. 
\end{lemma}
\begin{proof}
By equation~\eqref{eq:32}, we have that $W_{I}=-W_{I_{0}^{\complement}}$, from which we deduce
\begin{displaymath}
W_{I}Y^{-1}W_{I}^{t}=\big\vert W_{I}Y^{-1}W_{I_{0}^{\complement}}^{t}\big\vert.
\end{displaymath}
By Lemma~\ref{lemm:17}, there is a constant $C_{1}>0$ such that 
\begin{align}
\notag
W_{I}Y^{-1}W_{I}^{t}&\le C_{1}\sum_{i\in I}\sum_{j\in I_{0}^{\complement}}\frac{y_{i}y_{j}}{1+\displaystyle\sum_{k\colon\rk_{k}\ge\min(\rk_{i},\rk_
{j})}y_{k}} \\
\label{eq:47}
&\le C_{1}\sum_{\ell=1}^{g}\frac{\displaystyle\sum_{\substack{i\in I\\ \rk_{i}=\ell}}y_{i}\sum_{\substack{j\in I_{0}^{\complement}\\ \rk_{j}\ge\ell}}y_
{j}+\sum_{\substack{i\in I\\ \rk_{i}>\ell}}y_{i}\sum_{\substack{j\in I_{0}^{\complement}\\ \rk_{j}=\ell}}y_{j}}{1+\displaystyle\sum_{k\colon\rk_{k}\ge
\ell}y_{k}}.
\end{align}
Before continuing, we observe that given $a,b,c,d\ge 0$ with $a+b>0$, then the inequality
\begin{displaymath}
\frac{ab}{a+b}\le\frac{(a+c)(b+d)}{a+b+c+d}
\end{displaymath}
holds, since we easily verify that
\begin{displaymath}
(a+b)(a+c)(b+d)-ab(a+b+c+d)=a^{2}d+b^{2}c+acd+bcd\ge 0.
\end{displaymath}
From this observation we can add the missing terms in the denominators of the quantity~\eqref{eq:47}, as long as we add the same term to each 
of the factors of the corresponding numerator. Since we can further add positive terms to the numerators, we deduce
\begin{displaymath}
W_{I}Y^{-1}W_{I}^{t}\le C_{1}\sum_{\ell=1}^{g}\frac{2\sum_{i\in I}y_{i}\sum_{j\in I_{0}^{\complement}}y_{j}}{\sum_{k=0}^{d}y_{k}},
\end{displaymath}
which implies the result. 
\end{proof}

The key estimate of this subsection is the next result.

\begin{proposition}
\label{prop:9}
Let $J,K\subseteq\{1,\ldots,d\}$ of cardinality $r$, $I=J\cap K$, and fix $\varepsilon>0$. Then, with the above notations, there is a constant $C>0$ 
such that for $(y_{1},\ldots,y_{d})\in\R_{\ge\varepsilon }^{d}$, the upper bound
\begin{displaymath}
\Bigg\vert\det\Bigg(\frac{\partial^{2}\varphi}{\partial y_{j}\partial y_{k}}\Bigg)_{\substack{j\in J\\k\in K}}\Bigg\vert\le C\,\frac{\sum_{i\in I_{0}^{\complement}}
y_{i}}{\sum_{i=0}^{d}y_{i}}\frac{1}{\prod_{j\in J}y_{j}^{1/2}\prod_{k\in K}y_{k}^{1/2}}
\end{displaymath}
holds.
\end{proposition}
\begin{proof}
By Lemma~\ref{lemm:12}, we have
\begin{equation}
\label{eq:49}
\det\Bigg(\frac{\partial^{2}\varphi}{\partial y_{j}\partial y_{k}}\Bigg)_{\substack{j\in J\\k\in K}}=2^{r}\det\big(W_{j}Y^{-1}W_{k}^{t}\big)_{\substack{j\in 
J\\k\in K}}.
\end{equation}
We now fix $i_{0}\in I$ and note that the determinant of the Gram matrix on the right-hand side of~\eqref{eq:49} remains unchanged, if we replace 
both occurrences of $W_{i_{0}}$ by $(1/y_{i_{0}})W_{I}$. Recalling that the determinant of a real symmetric positive definite matrix can be bounded 
by the product of its diagonal elements, we are led by means of Lemma~\ref{lemm:17} and Lemma~\ref{lemm:18} to the upper bound
\begin{align*}
\det\Bigg(\frac{\partial^{2}\varphi}{\partial y_{j}\partial y_{k}}\Bigg)_{\substack{j\in J\\k\in K}}&\le\frac{C_{i_{0}}}{y_{i_{0}}^{2}}\frac{\sum_{i\in I}y_{i}
\sum_{j\in I_{0}^{\complement}}y_{j}}{\sum_{k=0}^{d}y_{k}}\frac{1}{\prod_{\substack{j\in J\\j\neq i_{0}}}y_{j}^{1/2}\prod_{\substack{k\in K\\k\neq i_
{0}}}y_{k}^{1/2}} \\
&=\frac{C_{i_{0}}}{y_{i_{0}}}\frac{\sum_{i\in I}y_{i}\sum_{j\in I_{0}^{\complement}}y_{j}}{\sum_{k=0}^{d}y_{k}}\frac{1}{\prod_{j\in J}y_{j}^{1/2}\prod_
{k\in K}y_{k}^{1/2}},
\end{align*}
where $C_{i_{0}}>0$ is a constant; observe that in the application of Lemma~\ref{lemm:17} we are allowed to remove the summand $1$ in the
denominator because we have $(y_{1},\ldots,y_{d})\in\R_{\ge\varepsilon }^{d}$. Since the inequality just obtained is true for each $i_{0}\in I$, we
find a constant $C>0$ such that
\begin{displaymath}
\sum_{i_{0}\in I}y_{i_{0}}\det\Bigg(\frac{\partial^{2}\varphi}{\partial y_{j}\partial y_{k}}\Bigg)_{\substack{j\in J\\k\in K}}\le C\frac{\sum_{i\in I}y_{i}\sum_
{j\in I_{0}^{\complement}}y_{j}}{\sum_{k=0}^{d}y_{k}}\frac{1}{\prod_{j\in J}y_{j}^{1/2}\prod_{k\in K}y_{k}^{1/2}},
\end{displaymath}
from which the claim follows immediately.
\end{proof}

\begin{lemma}
\label{lemm:19} 
Let $\varepsilon >0$ be a real number and $m,n>0$ integers. Then, there is a constant $C>0$ such that for all $s_{1},\ldots,s_{m},t_{1},\ldots,t_{n}
\ge\varepsilon $, the bound
\begin{displaymath}
\sum_{j=1}^{m}s_{j}\Bigg(\prod_{k=1}^{n}t_{k}\Bigg)^{1/n}\le C\,\sum_{k=1}^{n}t_{k}\prod_{j=1}^{m}s_{j}
\end{displaymath}
holds.
\end{lemma}
\begin{proof}
Without loss of generality, we can assume that $s_{1}\ge s_{j}$ for $j=1,\ldots,m$. Since $s_{j}\ge\varepsilon$ for $j=1,\ldots,m$, we obtain
\begin{displaymath}
\sum_{j=1}^{m}s_{j}\le  ms_{1}\le ms_{1}\frac{s_{2}}{\varepsilon}\cdot\ldots\cdot\frac{s_{m}}{\varepsilon}=\frac{m}{\varepsilon^{m-1}}\prod_{j=1}^{m}
s_{j}.
\end{displaymath}
On the other hand, using the inequality between the geometric and the arithmetic mean, leads to 
\begin{displaymath}
\Bigg(\prod_{k=1}^{n}t_{k}\Bigg)^{1/n}\le\frac{1}{n}\sum_{k=1}^{n}t_{k}.
\end{displaymath}
Putting the above two inequalities together yields
\begin{displaymath}
\sum_{j=1}^{m}s_{j}\Bigg(\prod_{k=1}^{n}t_{k}\Bigg)^{1/n}\le\frac{m}{\varepsilon^{m-1}}\sum_{k=1}^{n}t_{k}\prod_{j=1}^{m}s_{j},
\end{displaymath}
which proves the claim.
\end{proof}

\subsection{The arithmetic self-intersection of the adelic arithmetic line bundle \texorpdfstring{$\overline{\sJ}_{k,m,N}$}{JkmN}}
\label{sec:arithm-self-inters}

We return to the setting introduced in Subsection~\ref{sec:glob-arithm-case}. For this, we realize $\sB_{g,N}$, as in Subsection~\ref{adJkmN} 
preceding Lemma~\ref{lem:ellN}, as an open dense subset $\caU$ of a projective arithmetic variety $\caX$ of relative dimension $d=g(g+1)/2+g$ 
over $\Spec(\Z)$. The difference $\caX\setminus\sB_{g,N}$ is then the support of an effective divisor $\caB$, namely the union of the fibers over 
the primes dividing $N$ together with the union of the closure of the components of the boundary divisor of the toroidal compactification $\overline
{\sB}_{g,N,\Pi}$ under consideration. As usual, we write $X=\caX(\C)$, i.\,e., $X=\overline{\sB}_{g,N,\Pi}(\C)$, for the associated complex manifold, 
$U=\sB_{g,N}(\C)$, and $B$ for the divisor induced by $\caB$ on $X$. We upgrade $\caB$ to an arithmetic boundary divisor $\overline{\caB}=
(\caB,g_{B})$ by means of a suitable Green function $g_{B}$ of continuous type for $B$.

As discussed in Subsection~\ref{adJkmN}, the adelic arithmetic line bundle $\overline{\sL}_{N}$ is integrable on $\sB_{g,N}$, while Theorem~\ref
{thm:12} in Subsection~\ref{sec:non-integr-line} shows that the adelic arithmetic line bundle $\overline{\omega}_{N}$ is not integrable on $\sA_{g,
N}$, in general. However, we can change the metric $\Vert\cdot\Vert_{\mathrm{Fal}}$ on $\overline{\omega}_{N}$ to obtain an integrable adelic 
arithmetic line bundle $\overline{\omega}_{N}'$ on $\sA_{g,N}$ as follows. By the definition of $\sA_{g}^{\ast}$, we can choose $r\in\N$ large 
enough so that $\omega^{\otimes r}$ becomes a very ample line bundle on $\sA_{g}^{\ast}$ giving rise to the closed immersion $\iota\colon\sA_
{g}^{\ast}\hookrightarrow\P^{n}_{\Z}$ for suitable $n\in\N$. Choosing next a smooth projective admissible cone decomposition $\Sigma$ of 
$\overline{C}_{g}$ and extending the two morphisms $\rho_{N}'$ and $\rho'$ in diagram~\eqref{eq:2mor} to the respective compactifications, 
we obtain the diagram
\begin{displaymath}
\overline{\sA}_{g,N,\Sigma}\overset{\rho_{N,\Sigma}'}\longrightarrow\overline{\sA}_{g,\Sigma}\overset{\rho_{\Sigma}'}\longrightarrow\sA_{g}^{\ast}.
\end{displaymath}
By construction, we find that
\begin{displaymath}
\omega_{N}=(\rho_{\Sigma}'\circ\rho_{N,\Sigma}')^{\ast}\omega=(\rho_{\Sigma}'\circ\rho_{N,\Sigma}')^{\ast}\iota^{\ast}\,\caO_{\P^{n}_{\Z}}(1)^
{\otimes 1/r}=(\iota\circ\rho_{\Sigma}'\circ\rho_{N,\Sigma}')^{\ast}\,\caO_{\P^{n}_{\Z}}(1)^{\otimes 1/r}.
\end{displaymath}
By now equipping $\caO_{\P^{n}_{\Z}}(1)$ with the Fubini--Study metric $\Vert\cdot\Vert_{\mathrm{FS}}$, we obtain after tensoring $1/r$-times 
and pull-back by $\iota\circ\rho_{\Sigma}'\circ\rho_{N,\Sigma}'$ the desired integrable adelic arithmetic line bundle $\overline{\omega}_{N}'$ on 
$\sA_{g,N}$, since the adelic arithmetic line bundle $\overline{\caO}_{\P^{n}_{\Z}}(1)=(\caO_{\P^{n}_{\Z}}(1),\Vert\cdot\Vert_{\mathrm{FS}})$ is 
integrable on $\P^{n}_{\Z}$. In this way, next to the adelic arithmetic line bundle $\overline{\sJ}_{k,m,N}$, we obtain the integrable adelic arithmetic 
line bundle 
\begin{displaymath}
\overline{\sJ}_{k,m,N}'\coloneqq\pi_{N}^{\ast}\overline{\omega}_{N}'^{\otimes k}\otimes\overline{\sL}_{N}^{\otimes m}
\end{displaymath}
on $\sB_{g,N}$.

In the next step, we shift from the language of adelic arithmetic line bundles to adelic arithmetic divisors on $\sB_{g,N}$. For this, we choose a 
rational $\Q$-section $s_{\omega}$ of $\omega_{N}$ and a rational $\Q$-section $s_{\sL}$ of $\sL_{N}$, and define the adelic arithmetic divisors
\begin{align*}
\overline{\caD}_{k,m,N}&=(\caD_{k,m,N},g_{k,m,N})\coloneqq\divh_{\overline{\sJ}_{k,m,N}}(s_{\omega}^{\otimes k}\otimes s_{\sL}^{\otimes m}), \\
\overline{\caD}_{k,m,N}'&=(\caD_{k,m,N},g_{k,m,N}')\coloneqq\divh_{\overline{\sJ}_{k,m,N}'}(s_{\omega}^{\otimes k}\otimes s_{\sL}^{\otimes m}).
\end{align*}
By construction, the adelic arithmetic divisors $\overline{\caD}_{k,m,N}'$ are integrable on $\sB_{g,N}$, so their arithmetic self-intersection product 
can be defined using the techniques developed in~\cite{YuanZhang:adelic}. However, we are interested in the arithmetic self-intersection product 
of the adelic arithmetic divisors $\overline{\caD}_{k,m,N}$, which cannot be defined by these techniques, although they have the same divisorial 
part as the adelic arithmetic divisors $\overline{\caD}_{k,m,N}'$. In order to solve this problem, we will use the generalized arithmetic self-intersection 
product of the adelic arithmetic divisors $\overline{\caD}_{k,m,N}$ involving finite energies as developed in Subsection~\ref{fourth-section}. For 
this, we need some more notations. First, we find by the linearity that
\begin{displaymath}
\overline{\caD}_{k,m,N}=k\overline{\caD}_{1,0,N}+m\overline{\caD}_{0,1,N},\quad\overline{\caD}_{k,m,N}'=k\overline{\caD}_{1,0,N}'+m\overline
{\caD}_{0,1,N}
\end{displaymath}
and
\begin{displaymath}
\caD_{k,m,N}=k\caD_{1,0,N}+m\caD_{0,1,N},\quad g_{k,m,N}=kg_{1,0,N}+mg_{0,1,N},\quad g_{k,m,N}'=kg_{1,0,N}'+mg_{0,1,N}.
\end{displaymath}
Next, we choose a sufficiently large arithmetic reference divisor $\overline{\caE}=(\caE,g_{E})$ on $\caX$ with $g_{E}$ a Green function of 
smooth type for $E$ such that
\begin{equation}
\label{eq:refdiv}
\overline{\caE}\ge\overline{\caB},\qquad
\overline{\caE}\ge\overline{\caD}_{1,0,N}'+2\overline{\caB},\qquad\overline{\caE}\ge\overline{\caD}_{0,1,N}+2\overline{\caB}.
\end{equation}
On the associated complex manifold $X=\caX(\C)$, we then define the functions
\begin{displaymath}
\varphi_{1,0,N}=g_{1,0,N}-g_{E}\qquad\text{and}\qquad\varphi_{1,0,N}'=g_{1,0,N}'-g_{E},
\end{displaymath}
and let $\omega_{0}=\omega_{E}(g_{E})$. By construction, the function $\varphi_{1,0,N}'$ is $\omega_{0}$-plurisubharmonic.

We claim that $\varphi_{1,0,N}\in\caE(X,\omega_{0},\varphi_{1,0,N}')$: To prove the claim, we start by noting that $g_{1,0,N}=o(g_{B})$ by
arguing as in Lemma~\ref{lem:omega'} and $g_{E}=O(g_{B})$ since $E\ge D_{1,0,N}+2B$, from which we derive
\begin{displaymath}
\varphi_{1,0,N}=g_{1,0,N}-g_{E}<\infty.
\end{displaymath}
Next, observing that the rational section $s_{\omega}$ of $\omega_{N}$ induces a meromorphic section $s=s_{\omega,\C}$ of $\omega_{N}
\otimes_{\Z[1/N,\zeta_{N}]}\C$ over $\sA_{g,N}(\C)$, we recall that
\begin{displaymath}
\log\Vert s\Vert_{\mathrm{inv}}^{2}=\log\vert s\vert^{2}+\log\det(Y),
\end{displaymath}
from which we compute
\begin{align*}
\ddc\varphi_{1,0,N}+\omega_{0}&=\ddc(g_{1,0,N}-g_{E})+\omega_{0}=\ddc g_{1,0,N}-\ddc g_{E}+\omega_{0} \\
&=-\ddc\log\Vert s\Vert_{\mathrm{inv}}^{2}+\delta_{E}=-\ddc\log\vert s\vert^{2}+\delta_{D_{1,0,N}}-\ddc\log\det(Y)+\delta_{C},
\end{align*}
where $C=E-D_{1,0,N}\ge 0$. By the meromorphicity of the section $s$, we therefore arrive at the inequality
\begin{displaymath}
\ddc\varphi_{0,1,N}+\omega_{0}\ge-\ddc\log\det(Y)
\end{displaymath}
$\sA_{g,N}(\C)$. Now, we recall from~\cite[pp.~71--73]{BV} that the function $-\log\det(Y)$ is convex on $\H_{g}$, whence
\begin{displaymath}
\ddc\varphi_{0,1,N}+\omega_{0}\ge 0,
\end{displaymath}
from which we conclude that $\varphi_{0,1,N}$ is $\omega_{0}$-plurisubharmonic on $X$. Finally, since the natural invariant metric $\Vert
\cdot\Vert_{\mathrm{inv}}$ is a good metric in the sense of Mumford~\cite{mumford:hp}, we find that  $\varphi_{1,0,N}$ has relative full 
mass with respect to $\varphi_{1,0,N}'$. Therefore, we end up with $\varphi_{1,0,N}\in\caE(X,\omega_{0},\varphi_{1,0,N}')$.

More generally, multiplying~\eqref{eq:refdiv} by $(k+m)$, we obtain
\begin{displaymath}
(k+m)\overline{\caE}\ge k\overline{\caD}_{1,0,N}'+2k\overline{\caB}+m\overline{\caD}_{0,1,N}+2m\overline{\caB}=\overline{\caD}_{k,m,N}'+
2(k+m)\overline{\caB}.
\end{displaymath}
By then defining the functions
\begin{displaymath}
\varphi_{k,m,N}=g_{k,m,N}-(k+m)g_{E}\qquad\text{and}\qquad\varphi_{k,m,N}'=g_{k,m,N}'-(k+m)g_{E},
\end{displaymath}
one shows in the same way as above that $\varphi_{k,m,N}'$ is $(k+m)\omega_{0}$-plurisubharmonic and that $\varphi_{k,m,N}\in\caE
(X,(k+m)\omega_{0},\varphi_{k,m,N}')$. 

\medskip
\noindent
We are now able to state the first main theorem of this section.

\begin{theorem}
\label{thm:9}
For non-negative integers $k,m$, the adelic arithmetic divisor $\overline{\caD}_{k,m,N}$ on $\sB_{g,N}$ has finite energy with respect to 
$\overline{\caD}_{k,m,N}'$. Therefore, the generalized arithmetic intersection product $\overline{\caD}_{k,m,N}^{d+1}$ is a well-defined real 
number.
\end{theorem}
\begin{proof}
In order to prove Theorem~\ref{thm:9}, it is enough to show that (using the preceding notation)
\begin{displaymath}
\varphi_{k,m,N}\in\caE^{1}(X,(k+m)\omega_{0},\varphi_{k,m,N}'). 
\end{displaymath}
Recalling that $X=\overline{\sB}_{g,N,\Pi}(\C)$ and $U=\sB_{g,N}(\C)$, and taking into account Lemma~\ref{lemm:10}, we are left to show that
\begin{displaymath}
\int_{X}(\varphi_{k,m,N}-\varphi_{k,m,N}')\big\langle(\ddc\varphi_{k,m,N}+(k+m)\omega_{0})^{\wedge d}\big\rangle>-\infty.
\end{displaymath}
By the multilinearity of the non-pluripolar product recalled in~Theorem~\ref{thm:2}~\ref{item:4}), it is enough to show that for any $0\le r\le d$, we 
have
\begin{displaymath}
\int_{X}(\varphi_{1,0,N}-\varphi_{1,0,N}')\big\langle(\ddc\varphi_{1,0,N}+\omega_{0})^{\wedge d-r}\wedge(\ddc\varphi_{0,1,N}+\omega_{0})^{\wedge 
r}\big\rangle>-\infty.
\end{displaymath}
Since the restrictions of the currents $\ddc\varphi_{1,0,N}+\omega_{0}$ and $\ddc\varphi_{0,1,N}+\omega_{0}$ to $U$ are smooth, we have
\begin{align*}
&\int_{X}(\varphi_{1,0,N}-\varphi_{1,0,N}')\big\langle(\ddc\varphi_{1,0,N}+\omega_{0})^{\wedge d-r}\wedge(\ddc\varphi_{0,1,N}+\omega_{0})^{\wedge 
r}\big\rangle \\
&\qquad=\int_{U}(\varphi_{1,0,N}-\varphi_{1,0,N}')(\ddc\varphi_{1,0,N}+\omega_{0})^{\wedge d-r}\wedge(\ddc\varphi_{0,1,N}+\omega_{0})^{\wedge r}.
\end{align*}
We will now prove the claimed boundedness result locally on the coordinate neighbourhoods $V\subseteq X$ introduced in Remark~\ref{rem:loccor}.
Recall that in order to define such a coordinate neighbourhood $V$ around a point $p_{\overline{\tau}}$ of a stratum $\overline{\sB}_{g,N,\overline
{\tau}}(\C)$ of $\overline{\sB}_{g,N,\Pi}(\C)$, we started from an integral basis
\begin{equation}
\label{eq:39}
\{(Y_{1},\beta_{1}),\ldots,(Y_{d},\beta_{d})\}
\end{equation}
of a maximal cone $\tau'\in\Pi$ that has $\tau$ as a face, together with a pair $(Y_{0},\beta_{0})$ with $Y_{0}$ a real symmetric positive definite 
$(g\times g)$-matrix and a row vector $\beta_{0}\in\R^{(1,g)}$. After possibly replacing the smooth projective admissible cone decomposition $\Pi$
of $\widetilde{C}_{g}$ by its barycentric subdivision, we may assume by~\cite[Lemma~3.5]{bghdj_sing} without loss of generality that the matrices 
$Y_{0},\ldots,Y_{d}$ satisfy the flag condition of Definition~\ref{def:flco} for each maximal cone $\tau'\in\Pi$. The coordinate neighbourhood $V$
with coordinates $(q_{1},\ldots,q_{d})$ arount the point $p_{\overline{\tau}}$ is then characterized as follows:
\begin{enumerate}
\item 
There is a subset $L\subseteq\{1,\ldots,d\}$ such that $V\setminus U$ is described by the equation $\prod_{j\in L}q_{j}=0$.
\item 
There are real numbers $0<r_{1},\ldots,r_{d}<1/e$ such that the set of pairs
\begin{displaymath}
V'\coloneqq\Bigg\{(Z,W)=i(Y_{0},\beta_{0})+\sum_{j=1}^{d}z_{j}(Y_{j},\beta_{j})\Bigg\vert\,z_{j}\in\C:\,
\begin{alignedat}{2}
2\pi\im(z_{j})&>-\log(r_{j}),\ &j&\in L \\
\vert z_{j}\vert&<r_{j},&j&\notin L
\end{alignedat}
\Bigg\}
\end{displaymath}
constitutes an open subset of $\H_{g}\times\C^{(1,g)}$. 
\item 
By choosing $V'$ small enough, the uniformization map $\H_{g}\times\C^{(1,g)}\rightarrow\sB_{g,N}(\C)$ maps the open subset $V'\subseteq\H_
{g}\times\C^{(1,g)}$ to the coordinate neighbourhood $V\subseteq X$ centered at $p_{\overline{\tau}}$ by means of the assignment
\begin{displaymath}
(z_{1},\ldots,z_{d})\mapsto(q_{1},\ldots,q_{d}),\text{ where }q_{j}=\Bigg\lbrace
\begin{alignedat}{2}
&e^{2\pi iz_{j}},\ &j&\in L, \\
&z_{j},&j&\notin L,
\end{alignedat}
\end{displaymath}
so that $V'$ maps surjectively onto $V\cap U$.
\end{enumerate}
We note that we need to keep the order prescribed by the flag condition, so we cannot assume that the normal crossing divisor $V\setminus U$ 
is given by the vanishing of the first coordinates.

Our task now is to prove that the integral
\begin{equation}
\label{eq:41}
\int_{V\cap U}(\varphi_{1,0,N}-\varphi_{1,0,N}')(\ddc\varphi_{1,0,N}+\omega_{0})^{\wedge d-r}\wedge(\ddc\varphi_{0,1,N}+\omega_{0})^{\wedge r}
\end{equation}
is finite. In fact, we will show that the integrand of~\eqref{eq:41} is absolutely integrable by bounding its absolute value by an integrable function. 
We need some more notation. We write
\begin{displaymath}
\eta_{1}=(\varphi_{1,0,N}-\varphi_{1,0,N}')(\ddc\varphi_{1,0,N}+\omega_{0})^{\wedge d-r}\qquad\text{and}\qquad\eta_{2}=(\ddc\varphi_{0,1,N}+
\omega_{0})^{\wedge r};
\end{displaymath}
then the integrand of~\eqref{eq:41} becomes $\eta=\eta_{1}\wedge\eta_{2}$. We further write
\begin{displaymath}
\eta_{2}=\sum_{J,K}\eta_{2,J,K}\qquad\text{with}\qquad\eta_{2,J,K}=f_{2,J,K}\,\dd q_{J}\wedge\dd\overline{q}_{K},
\end{displaymath}
where the sum runs over all subsets $J,K\subseteq\{1,\ldots,d\}$ of cardinality $r$ so that if $J=\{j_{1},\ldots,j_{r}\}$, then $\dd q_{J}=\dd q_{j_{1}}
\wedge\ldots\wedge\dd q_{j_{r}}$, and similarly for $\dd\overline{q}_{K}$. This gives
\begin{displaymath}
\eta=\sum_{J,K}\eta_{1,J^{\complement},K^{\complement}}\wedge\eta_{2,J,K},
\end{displaymath}
where $J^{\complement}=\{1,\ldots,d\}\setminus J$ and $K^{\complement}=\{1,\ldots,d\}\setminus K$; we also write
\begin{displaymath}
\eta_{1,J^{\complement},K^{\complement}}=f_{1,J^{\complement},K^{\complement}}\,\dd q_{J^{\complement}}\wedge\dd\overline{q}_{K^
{\complement}}. 
\end{displaymath}
Since the natural invariant metric $\Vert\cdot\Vert_{\mathrm{inv}}$ on the Hodge bundle $\overline{\omega}_{N}$ is a good metric in the sense 
of Mumford~\cite{mumford:hp} (see also~\cite{BurgosKramerKuehn:accavb}), we know that the differential forms $\eta_{1,J^{\complement},K^
{\complement}}$ have log-log growth when approaching the boundary $V\setminus U$. This means that there is a constant $C>0$ and an integer 
$M\ge 0$ such that the upper bound
\begin{equation}
\label{eq:50}
\vert f_{1,J^{\complement},K^{\complement}}\vert\le C\,\frac{\big(\sum_{i=1}^{d}\log(\log\vert q_{i}\vert^{2})\big)^{M}}{\prod_{\substack{j\in J^
{\complement}\\k\in K^{\complement}}}\big\vert q_{j}\log\vert q_{j}\vert\big\vert\,\big\vert q_{k}\log\vert q_{k}\vert\big\vert}
\end{equation}
holds on $V\cap U$.

Next, we need to find a bound for the functions $f_{2,J,K}$ on $V\cap U$. For this, we recall that the rational section $s_{\sL}$ of $\sL_{N}$
induces a meromorphic section $s=s_{\sL,\C}$ of $\sL_{N}\otimes_{\Z[1/N,\zeta_{N}]}\C$ over $U$. As $(Z,W)$ ranges through the open subset 
$V'\subseteq\H_{g}\times\C^{(1,g)}$ with coordinates $(z_{1},\ldots,z_{d})$ as described above, we compute
\begin{displaymath}
\log\Vert s\Vert_{\mathrm{inv}}^{2}=\log\vert s\vert^{2}-4\pi\im(W)\im(Z)^{-1}\im(W)^{t}=\log\vert s\vert^{2}-4\pi\beta Y^{-1}\beta^{t},
\end{displaymath}
where
\begin{displaymath}
Y=Y(y_{1},\ldots,y_{d})=Y_{0}+\sum_{j=1}^{d}y_{j}Y_{j}\qquad\text{and}\qquad\beta=\beta(y_{1},\ldots,y_{d})=\beta_{0}+\sum_{j=1}^{d}y_{j}
\beta_{j},
\end{displaymath}
recalling that $z_{j}=x_{j}+iy_{j}$ ($j=1,\ldots,d$). From this, we deduce the following equalities on the open subset $V'$ (which maps surjectively 
onto $V\cap U$)
\begin{align*}
\ddc\varphi_{0,1,N}+\omega_{0}&=\ddc(g_{0,1,N}-g_{E})+\omega_{0}=\ddc g_{0,1,N}-\ddc g_{E}+\omega_{0} \\
&=-\ddc\log\Vert s\Vert_{\mathrm{inv}}^{2}+\delta_{E}=-\ddc\log\vert s\vert^{2}+\delta_{D_{0,1,N}}+4\pi\ddc\beta Y^{-1}\beta^{t}+\delta_{C},
\end{align*}
where $C=E-D_{0,1,N}\ge 0$ such that $\vert C_{\C}\vert\subseteq V\setminus U$. By the meromorphicity of the section $s$, we therefore have
the equality 
\begin{equation}
\label{eq:44}
\ddc\varphi_{0,1,N}+\omega_{0}=4\pi\ddc\beta Y^{-1}\beta^{t}
\end{equation}
on the open subset $V'$.

Now, we replace the summand $Y_{0}$ of $Y$ by $Y_{0}-\sum_{j\notin L}r'_{j}Y_{j}$ with $r'_{j}$ being strictly smaller than $r_{j}$, which is still 
a real symmetric positive definite matrix by the definition of the open subset $V'$. Then, we replace $z_{j}$ by $z_{j}+ir'_{j}$ for $j\notin L$. After 
possibly shrinking $V'$ once again, we can thus assume that $y_{j}\ge\varepsilon>0$ for all $j\in\{1,\ldots,d\}$ and that the uniformization map is 
given by $q_{j}=z_{j}-ir'_{j}$ for all $j\notin L$. By recalling the function $\varphi\colon\R^{d}_{\ge\varepsilon}\rightarrow\R$ introduced in Subsection
\ref{sec:norm-like-functions}, we deduce from equation~\eqref{eq:44} that
\begin{equation}
\label{eq:45}
f_{2,J,K}=(2i)^{r}\,\det\Bigg(\frac{\partial^{2}\varphi}{\partial y_{j}\partial y_{k}}\Bigg)_{\substack{j\in J\\k\in K}}\,\,\prod_{\substack{j\in J\\k\in K}}\frac
{\partial y_{j}}{\partial q_{j}}\frac{\partial y_{k}}{\partial\overline{q}_{k}}.
\end{equation}
Using Proposition~\ref{prop:9}, we find that there are constants $C,C'>0$ such that the bound
\begin{align}
\notag
\vert f_{2,J,K}\vert&\le C\,\frac{\sum_{i\in I_{0}^{\complement}}y_{i}}{\sum_{i=0}^{d}y_{i}}\frac{1}{\prod_{j\in J}y_{j}^{1/2}\prod_{k\in K}y_{k}^{1/2}}\,
\Bigg\vert\prod_{\substack{j\in J\\k\in K}}\frac{\partial y_{j}}{\partial q_{j}}\frac{\partial y_{k}}{\partial\overline{q}_{k}}\Bigg\vert \\
 \label{eq:46} 
&\le C'\,\frac{\sum_{i\in I_{0}^{\complement}}\big\vert\log\vert q_{i}\vert\big\vert}{\sum_{i=0}^{d}\big\vert\log\vert q_{i}\vert\big\vert}\frac{1}{\prod_{j\in J}
\big\vert\log\vert q_{j}\vert\big\vert^{1/2}\prod_{k\in K}\big\vert\log\vert q_{k}\vert\big\vert^{1/2}}\prod_{\substack{j\in J\\k\in K}}\frac{1}{\vert q_{j}q_{k}
\vert},
\end{align}
holds on $V\cap U$; recall that $I=J\cap K$ and $I_{0}^{\complement}=\{0,\ldots,d\}\setminus I$, and that we use the convention $y_{0}=\log\vert q_
{0}\vert=1$.

The idea to prove the boundedness of the integral~\eqref{eq:41} is now simple. We observe that an integral of the form
\begin{displaymath}
\int_{\vert q\vert\le 1}\frac{\dd q\wedge\dd\overline{q}}{\vert q\vert^{2}\big\vert\log\vert q\vert\big\vert^{2a}}
\end{displaymath}
is convergent as long as $a>1/2$. This shows that the respective integrals over the components of $\eta_{2}$ would be divergent because for each
term of the form $\dd q_{i}/q_{i}$, we only have a term of the form $\log\vert q_{i}\vert$ of degree $-1/2$, while the respective integrals over the
components of $\eta_{1}$ have an excess of convergence because for each term of the form $\dd q_{i}/q_{i}$, we have a term of the form $\log
\vert q_{i}\vert$ of degree $-1$, which is more than what is needed. Luckily, the factor $\sum_{i\in I_{0}^{\complement}}\big\vert\log\vert q_{i}\vert
\big\vert/\sum_{i=0}^{d}\big\vert\log\vert q_{i}\vert\big\vert$ in~\eqref{eq:46} will allow us to transfer the excess of convergence from $\eta_{1}$ to 
$\eta_{2}$. 

Before doing so, we observe that the integrand of~\eqref{eq:41} vanishes unless $r=g$. Indeed, the differential form $\eta_{1}$ is the pull-back of a 
differential form on the base $\sA_{g,N}(\C)$ that has complex dimension $d'=d-g$. Therefore, $\eta_{1}$ vanishes unless $d-r\le d'=d-g$. In other 
words, $\eta_{1}$ vanishes unless $r\ge g$. On the other hand, equation~\eqref{eq:45} shows that $\eta_{2}$ vanishes unless $r\le g$, since the 
rank of the Hessian matrix of $\varphi$ is at most $g$ by Lemma~\ref{lemm:12}. We thus conclude that the integrand of the integral~\eqref{eq:41} 
vanishes unless $r=g$.    

Finally, we have all the ingredients in place to prove Theorem~\ref{thm:9}. Let $J,K\subseteq\{1,\ldots,d\}$ be subsets of cardinality $g$. As before, 
we write $I=J\cap K$, and we further set $I'=J^{\complement}\cap K^{\complement}$ and $I''=\{1,\ldots,d\}\setminus(I\cup I')$. Combining the upper 
bounds~\eqref{eq:50} and~\eqref{eq:46} on $V\cap U$, we obtain
\begin{displaymath}
\vert f_{1,J^{\complement},K^{\complement}}f_{2,J,K}\vert\le C\frac{\sum_{i\in I_{0}^{\complement}}\big\vert\log\vert q_{i}\vert\big\vert}{\sum_{i=0}^{d}
\big\vert\log\vert q_{i}\vert\big\vert}\frac{1}{\prod_{j\in I}\big\vert\log\vert q_{j}\vert\big\vert\prod_{k\in I''}\big\vert\log\vert q_{k}\vert\big\vert^{3/2}\prod_
{\ell\in I'}\big\vert\log\vert q_{\ell}\vert\big\vert^{2}}\prod_{i=1}^{d}\frac{1}{\vert q_{i}\vert^{2}},
\end{displaymath}
where $C>0$ is again a constant. Here we see that we have a potential problem of convergence with respect to the integration along the variables 
$q_{j}$ for $j\in I$, which can be resolved using Lemma~\ref{lemm:19}. Namely, using this lemma we deduce that there is a further constant $C'>0$ 
such that
\begin{align*}
\frac{\sum_{i\in I_{0}^{\complement}}\big\vert\log\vert q_{i}\vert\big\vert}{\sum_{i=0}^{d}\big\vert\log\vert q_{i}\vert\big\vert}\Bigg(\prod_{j\in I}\big\vert
\log\vert q_{j}\vert\big\vert\Bigg)^{1/n}\le C'\,\frac{\sum_{j\in I}\big\vert\log\vert q_{j}\vert\big\vert}{\sum_{i=0}^{d}\big\vert\log\vert q_{i}\vert\big\vert}
\prod_{i\in I_{0}^{\complement}}\big\vert\log\vert q_{i}\vert\big\vert\le \\
C'\prod_{i\in I_{0}^{\complement}}\big\vert\log\vert q_{i}\vert\big\vert=C'\prod_{i\in I^{\complement}}\big\vert\log\vert q_{i}\vert\big\vert=C'\prod_{k\in 
I''}\big\vert\log\vert q_{k}\vert\big\vert\prod_{\ell\in I'}\big\vert\log\vert q_{\ell}\vert\big\vert,
\end{align*}
where $n$ denotes the cardinality of $I$. Rearranging terms and taking the fourth root immediately leads to the inequality
\begin{displaymath}
\Bigg(\frac{\sum_{i\in I_{0}^{\complement}}\big\vert\log\vert q_{i}\vert\big\vert}{\sum_{i=0}^{d}\big\vert\log\vert q_{i}\vert\big\vert}\Bigg)^{1/4}\frac{1}
{\prod_{k\in I''}\big\vert\log\vert q_{k}\vert\big\vert^{1/4}\prod_{\ell\in I'}\big\vert\log\vert q_{\ell}\vert\big\vert^{1/4}}\le C'\frac{1}{\prod_{j\in I}\big\vert
\log\vert q_{j}\vert\big\vert^{1/4n}}.
\end{displaymath}
With this inequality at hand and observing that $\sum_{i\in I_{0}^{\complement}}\big\vert\log\vert q_{i}\vert\big\vert\big/\sum_{i=0}^{d}\big\vert\log
\vert q_{i}\vert\big\vert\le 1$, we arrive with a further constant $C''>0$ at the bound
\begin{displaymath}
\vert f_{1,J^{\complement},K^{\complement}}f_{2,J,K}\vert\le C''\frac{1}{\prod_{j\in I}\big\vert\log\vert q_{j}\vert\big\vert^{1+1/4n} \prod_{k\in I''}\big
\vert\log\vert q_{k}\vert\big\vert^{1+1/4}\prod_{\ell\in I'}\big\vert\log\vert q_{\ell}\vert\big\vert^{1+3/4}}\prod_{i=1}^{d}\frac{1}{\vert q_{i}\vert^{2}},
\end{displaymath}
which is integrable on $V\cap U$ and thus proves Theorem~\ref{thm:9}.
\end{proof}

Once we know that the arithmetic self-intersection product of the adelic arithmetic line bundle $\overline{\sJ}_{k,m,N}$ of Siegel--Jacobi forms of weight 
$k$, index $m$, and level $N$ is well-defined, it can easily be computed using the functorial properties discussed in Subsection~\ref{sixth-section}. As 
in the beginning of Subsection~\ref{sec:arithm-self-inters}, let $s_{\omega}$ be a rational $\Q$-section of $\omega_{N}$ and $s_{\sL}$ a rational $\Q
$-section of $\sL_{N}$, which allow us to define the adelic arithmetic divisor
\begin{displaymath}
\overline{\caD}_{k,m,N}=\divh_{\overline{\sJ}_{k,m,N}}(s_{\omega}^{\otimes k}\otimes s_{\sL}^{\otimes m})
\end{displaymath}
on $\sB_{g,N}$ corresponding to the adelic arithmetic line bundle $\overline{\sJ}_{k,m,N}$. In the same way, we define the adelic arithmetic divisor
\begin{displaymath}
\overline{\caE}_{k,N}=\divh_{\overline{\omega}_{N}^{\otimes k}}(s_{\omega}^{\otimes k})
\end{displaymath}
on $\sA_{g,N}$ corresponding to the adelic arithmetic line bundle $\overline{\omega}_{N}^{\otimes k}$ of Siegel modular forms of weight $k$ and level 
$N$.

\begin{theorem}
\label{thm:10}
With the above notations, we have
\begin{displaymath}
\overline{\caD}_{k,m,N}^{d+1}=\frac{(d+1)!}{(d'+1)!}k^{d'+1}m^{g}2^{g}\,\overline{\caE}_{1,N}^{d'+1},
\end{displaymath}
\end{theorem}
where we recall that $d'=g(g+1)/2$ and $d=d'+g$.
\begin{proof}
By the multilinearity of the arithmetic intersection product, we compute
\begin{displaymath}
\overline{\caD}_{k,m,N}^{d+1}=\sum_{j=0}^{d+1}\binom{d+1}{j}k^{d+1-j}m^{j}\overline{\caD}_{1,0,N}^{d+1-j}\,\overline{\caD}_{0,1,N}^{j}. 
\end{displaymath}
The next step is the global arithmetic analogue of the fact used in course of the proof of Theorem~\ref{thm:9} that the arithmetic intersection products 
$\overline{\caD}_{1,0,N}^{d+1-j}\,\overline{\caD}_{0,1,N}^{j}$ vanish unless $j=g$. Indeed, let $[2]\colon\sB_{g,N}\rightarrow\sB_{g,N}$ be the morphism 
given by fiberwise multiplication by $2$, which has degree $2^{2g}$. We will then compute the arithmetic intersection product $([2]^{\ast}_{\mathrm{adel}}
\overline{\caD}_{1,0,N})^{d+1-j}([2]^{\ast}_{\mathrm{adel}}\overline{\caD}_{0,1,N})^{j}$ in two ways. In the first case, we use that $\overline{\caD}_{1,0,N}$ 
is the pull-back of an adelic arithmetic divisor from the base $\sA_{g,N}$, whence $[2]^{\ast}_{\mathrm{adel}}\overline{\caD}_{1,0,N}=\overline{\caD}_{1,
0,N}$, while Lemma~\ref{lem:ellN} shows that $[2]^{\ast}_{\mathrm{adel}}\overline{\caD}_{0,1,N}\sim 2^{2}\,\overline{\caD}_{0,1,N}$ (linear equivalence 
of adelic arithmetic divisors). Therefore, we find in the first case
\begin{displaymath}
([2]^{\ast}_{\mathrm{adel}}\overline{\caD}_{1,0,N})^{d+1-j}([2]^{\ast}_{\mathrm{adel}}\overline{\caD}_{0,1,N})^{j}=2^{2j}\,\overline{\caD}_{1,0,N}^{d+1-j}\,
\overline{\caD}_{0,1,N}^{j}.
\end{displaymath}
In the second case, we use Proposition~\ref{prop:14} to deduce
\begin{displaymath}
([2]^{\ast}_{\mathrm{adel}}\overline{\caD}_{1,0,N})^{d+1-j}([2]^{\ast}_{\mathrm{adel}}\overline{\caD}_{0,1,N})^{j}=2^{2g}\,\overline{\caD}_{1,0,N}^{d+1-j}\,
\overline{\caD}_{0,1,N}^{j},   
\end{displaymath}
which proves the claim. As a consequence, we derive
\begin{displaymath}
\overline{\caD}_{k,m,N}^{d+1}=\binom{d+1}{g}k^{d'+1}m^{g}\overline{\caD}_{1,0,N}^{d'+1}\,\overline{\caD}_{0,1,N}^{g}.
\end{displaymath}
Finally, we use that for any fiber $F$ of the projection $\sB_{g,N}\rightarrow\sA_{g,N}$, the restriction of $\sL_{N}$ to $F$ equals twice the principal
polarization of the abelian variety $F$. We deduce that $\deg(\caD_{0,1,N}\vert_{F}^{g})=2^{g}g!$. Therefore, Proposition~\ref{prop:18} implies that
\begin{displaymath}
\overline{\caD}_{k,m,N}^{d+1}=\binom{d+1}{g}k^{d'+1}m^{g}2^{g}g!\,\overline{\caD}_{1,0,N}^{d'+1}=\frac{(d+1)!}{(d'+1)!}k^{d'+1}m^{g}2^{g}\,\overline
{\caE}_{1,N}^{d'+1},
\end{displaymath}
which proves the theorem.
\end{proof}

\begin{remark}
When $g=1$, the arithmetic self-intersection product of the line bundle of modular forms has been computed independently by U.~K\"uhn in~\cite
{Kuehn:gainc} and J.-B.~Bost in~\cite{Bost}. After normalizing the Faltings metric $\Vert\cdot\Vert_{\mathrm{Fal}}$ on the Hodge bundle $\overline
{\omega}_{N}$ by the factor $1/(4\pi)$, they obtained
\begin{displaymath}
\overline{\caE}_{1,N}^{2}=\frac{[\SL_{2}(\Z):\Gamma(N)]}{2}\Bigg(\frac{1}{2}\zeta_{\Q}(-1)+\zeta'_{\Q}(-1)\Bigg)=-\deg(\caE_{1,N})\Bigg(\frac{\zeta'_
{\Q}(-1)}{\zeta_{\Q}(-1)}+\frac{1}{2}\Bigg),
\end{displaymath}
where $\zeta_{\Q}(s)$ denotes the classical Riemann zeta function. For general $g$ and $N=1$, upon suitable normalizations, it is conjectured (see 
also~\cite{MR}) that a formula of the type
\begin{displaymath}
\overline{\caE}_{1,1}^{d'+1}=-\deg(\caE_{1,1})\Bigg(\frac{\zeta'_{\Q}(1-2g)}{\zeta_{\Q}(1-2g)}+\frac{\zeta'_{\Q}(3-2g)}{\zeta_{\Q}(3-2g)}+\ldots+\frac
{\zeta'_{\Q}(-1)}{\zeta_{\Q}(-1)}+a\log(2)+b\Bigg)
\end{displaymath}
with $a,b\in\Q$ holds. This has been mostly verified in the paper~\cite{JvP} by B.~Jung and A.~Pippich for $g=2$. There is further confirmation of 
an analogue of this formula for Hilbert modular surfaces in the paper~\cite{BBGK}. 

We wonder whether Theorem~\ref{thm:10} can be used to compute the arithmetic self-intersection product of the line bundle of Siegel modular forms 
by an inductive argument mimicking, in the arithmetic setting, Siegel's original proof of the formula for the volume of the fundamental domain of $\Sp
(2g,\Z)$, i.\,e., the formula
\begin{displaymath}
\deg(\caE_{1,1})=(-1)^{d'}\zeta_{\Q}(1-2g)\cdot\zeta_{\Q}(3-2g)\cdot\ldots\cdot\zeta_{\Q}(-1).
\end{displaymath}
Siegel's proof consists of the following two steps: (1)~Start with the quantity $\deg(\caE_{1,1})$ on $\sA_{g+1}(\C)$ and sweep it to the boundary 
$\partial\overline{\sA}_{g+1}(\C)\cong\sB_{g}(\C)$ using Eisenstein series. (2)~Push the resulting quantity down from $\sB_{g}(\C)$ to $\sA_{g}(\C)$, 
which leads to the desired inductive relation of $\deg(\caE_{1,1})$ on $\sA_{g+1}(\C)$ and on $\sA_{g}(\C)$. The formula established in Theorem
\ref{thm:10} corresponds, in the arithmetic setting, to step~(2) in Siegel's proof in the geometric setting.
\end{remark}

%\bibliographystyle{alpha}
%\bibliography{biblio}

\end{document}